\documentclass[a4paper,11pt]{amsart}

\usepackage{amsmath, amssymb, amsthm, float, mathtools,enumerate}
\usepackage{thmtools}
\usepackage[margin=1in]{geometry}
\usepackage{url}
\usepackage{bm}
\usepackage{dsfont}
\usepackage{multicol}
\usepackage[aligntableaux=center]{ytableau}
\usepackage{tabularray}
\usepackage[parfill]{parskip}
\usepackage{tikz}
\usepackage{hyperref}	

\hypersetup{
     colorlinks=true,
     linkcolor=blue,
     filecolor=magenta,
     urlcolor=cyan,
     citecolor=red }

\usepackage{cleveref}

\theoremstyle{plain}
\newtheorem{thm}{Theorem}
\numberwithin{thm}{section}
\newtheorem{thmx}{Theorem}

\newtheorem{lem}[thm]{Lemma}
\newtheorem{prop}[thm]{Proposition}
\newtheorem{cor}[thm]{Corollary}
\newtheorem{rmk}[thm]{Remark}

\theoremstyle{definition}
\newtheorem{defn}[thm]{Definition}

\newtheorem{conj}[thm]{Conjecture}
\newtheorem{ex}[thm]{Example}

\usepackage{tikz}
\tikzstyle{V}=[draw, fill =black, circle, inner sep=0pt, minimum size=3pt]

\begin{document}

\title[Kostant's problem for permutations 
of shape $(n-2,1,1)$ and $(n-3,2,1)$]{Kostant's problem for permutations\\ 
of shape $(n-2,1,1)$ and $(n-3,2,1)$}

\author{Samuel Creedon and Volodymyr Mazorchuk}

\date{}

\begin{abstract}
For a permutation $z$ in the symmetric group $\mathrm{S}_{n}$, denote by $L_{z}$ the corresponding simple highest weight module in the principal block of the BGG category $\mathcal{O}$ for the Lie algebra $\mathfrak{sl}_{n}(\mathbb{C})$. In this paper, we provide a combinatorial answer to Kostant's problem for the modules $L_{z}$ when $z$ has shape (associated Young diagram/integer partition via Robinson-Schensted correspondence) equal to $(n-2,1,1)$ or $(n-3,2,1)$. Moreover, we verify that certain closely related conjectures hold for such permutations, including the Indecomposability Conjecture, which states that applying any indecomposable projective functor to the corresponding simple highest weight module outputs either an indecomposable module or zero. 
\end{abstract}

\maketitle

\section{Introduction}\label{Sec1}

\subsection*{Background}

Let $\mathfrak{g}$ be a complex finite dimensional Lie algebra, and $M$ a $\mathfrak{g}$-module. Then \emph{Kostant's problem} (see \cite{Jo80}) is the question of whether the universal enveloping algebra $U(\mathfrak{g})$ of $\mathfrak{g}$ surjects onto the algebra of ad-finite linear endomorphisms of $M$. The answer to Kostant's problem is known for some classes of modules, but is wide open in general. See \cite{Ma23} for a recent overview.

The simple highest weight modules of the principal block of the Bernstein-Gelfand-Gelfand category $\mathcal{O}$ for $\mathfrak{g}=\mathfrak{sl}_{n}$ are indexed by permutations of the symmetric group $\mathrm{S}_{n}$. As such, we call a permutation \emph{Kostant positive}, if the answer to Kostant's problem is positive for the corresponding simple highest weight module, and \emph{Kostant negative} otherwise. Thus we may speak of answering Kostant's problem for permutations. Even in this setting, a general answer to Kostant's problem is not known. However, significant theory has been developed and many special cases have been resolved.

The results from \cite{KMM23} connected Kostant's problem for permutations to the action of projective functors (see \cite{BG80}) on the corresponding simple highest weight modules. As a result, Kostant's problem is shown to be closely linked to both the \emph{Indecomposability Conjecture} (\cite[Conjecture~2]{KiM16}) and 
\emph{K{\aa}hrstr{\"o}m's Conjecture} (\cite[Conjecture~1.2]{KMM23}). The 
former claims that applying a projective functor on such a simple highest 
weight module outputs a module that is either zero or indecomposable, while the latter gives a completely combinatorial reformulation of Kostant's problem, for these simple highest weight modules, via Hecke algebra combinatorics. Furthermore, it was proven in \cite{MS08} that the answer to Kostant's problem on permutations is an invariant of the Kazhdan-Lusztig left cell. Each Kazhdan-Lusztig left cell of $\mathrm{S}_{n}$ contains a unique involution, thus to answer Kostant's problem for all permutations in $\mathrm{S}_{n}$, it suffices to answer it for all involutions $\mathrm{Inv}_{n}\subset\mathrm{S}_{n}$. 

By the results from \cite{KM10}, \cite{KMM23}, and \cite{CM25-1}, Kostant's problem has been answered for all permutations in $\mathrm{S}_{n}$, for $1\leq n\leq 7$. It was also recently shown by the authors in \cite{CM24} that, as $n$ tends to infinity, almost all permutations are Kostant negative. This was proven by employing the result, from the same paper, that any permutation is Kostant negative whenever it consecutively contains the pattern $2143$. Moreover, very recently in \cite{CM26}, it was shown that any permutation is Kostant negative whenever it consecutively contains a pattern $p$ which is itself Kostant's negative when viewed as a permutation. With this in mind, the main results of this paper concerning Kostant's problem (as summarised in the next section) are given in terms of consecutive patterns. 

Lastly, the main result of \cite{MMM24} was an explicit combinatorial answer to Kostant's problem for all fully commutative permutations in $\mathrm{S}_{n}$. A fully commutative permutation is one whose shape has at most two rows, where shape refers to the Young diagram/integer partition $\lambda$ associated to the permutation via the Robinson-Schensted correspondence. In the same paper, both the Indecomposability Conjecture and 
K{\aa}hrstr{\"o}m's Conjecture were shown to hold for all fully commutative permutations. The former was shown by employing results of Brundan and Stroppel, while the latter was shown by translating Hecke algebra combinatorics into combinatorics of Temperley-Lieb diagrams. It was also proven in \cite[Theorem 6.1 (b)]{MMM24} that the proportion of Kostant positive fully commutative permutations of shape $\lambda$, out of all fully commutative permutations of shape $\lambda$, tends to $1$ as the number of boxes in the first row of $\lambda$ increases. It was then conjectured that this holds for any shape $\lambda$, not just those with at most two rows, see \cite[Section 6.9]{MMM24}. We refer to this conjecture as the \emph{Asymptotic Shape Conjecture}.

\subsection*{Results and paper outline}

The goal of this paper is to answer Kostant's problem, and investigate the conjectures mentioned above, for all permutations in $\mathrm{S}_{n}$ whose shape is either $\lambda=(n-2,1,1)$ or $\lambda=(n-3,2,1)$. These two families of shapes are two of the simplest non-fully commutative (at least three rows) families to consider, and are closed under adding boxes to the top row. Therefore, they make for natural candidates to study next.

To begin, we set up our notation, and recall all the background we will need in \Cref{Sec:2}. Then, in \Cref{Sec:3}, we establish various technical/reduction type results connecting the consecutive containment of patterns with Hecke algebra combinatorics relevant to Kostant's problem. Most of these results concern the Kazhdan-Lusztig preorders, and may be of independent interest.  

\Cref{Sec:4} focuses on answering Kostant's problem, and verifying the related conjectures, for all permutations of shape $(n-2,1,1)$. Firstly, by employing a result from \cite{CMZ19}, the Indecomposability Conjecture is quickly shown to hold for all such permutations. Then, following a technical proposition, we give the first main theorem of the paper, \Cref{Thm:4:11}, which provides a combinatorial answer to Kostant's problem for all involutions of shape $(n-2,1,1)$, as follows:    
\begin{thmx}
An involution with shape $\lambda=(n-2,1,1)$, where $n\geq 3$, is Kostant negative if and only if it consecutive contains the pattern $14325$.
\end{thmx}
By the results from \cite{MS08}, this theorem gives an answer to Kostant's problem for all permutations of shape $(n-2,1,1)$, not just involutions. Proving the ``if'' statement follows quickly from a result in \cite{CM25-1}, while the ``only if'' statement is shown to reduce down to establishing a family of inequalities within the Hecke algebra $\mathrm{H}_{n}$ associated to $\mathrm{S}_{n}$, which the technical proposition preceding the theorem accomplishes. It turns out, the manner in which this theorem is proved, also verifies that K{\aa}hrstr{\"o}m's Conjecture holds for such involutions. We close out this section by proving that the Asymptotic Shape Conjecture holds for the family of shapes $(n-2,1,1)$ for $n\geq3$, or, in the notation of \Cref{Sec:2.13}, holds for shape $(1,1)$.

Lastly, \Cref{Sec:5} focuses on answering Kostant's problem, and investigating the related conjectures, for all permutations of shape $(n-3,2,1)$. This shape is significantly more difficult than $(n-2,1,1)$, with \Cref{Sec:5} occupying most of the paper. We begin by classifying the involutions of shape $(n-3,2,1)$ into 13 different \emph{types}, and explain that, up to symmetry, only 7 such types need to be considered. We then go through each of these 7 types, one at a time, and, for the corresponding involutions, we prove that the Indecomposable Conjecture holds, and we give an answer to Kostant's problem. To accomplish this, we employ various known results within the literature, the reduction type results from \Cref{Sec:3}, and numerous ad hoc arguments. We then bring together the answers to Kostant's problem for all 13 types to produce the last main theorem of the paper, which gives a combinatorial answer to Kostant's problem for all involutions of shape $(n-3,2,1)$, as follows:
\begin{thmx}
An involution with shape $\lambda=(n-3,2,1)$, for any $n\geq 5$, is Kostant negative if and only if it consecutive contains one of the following patterns: $2143$, $14325$, $1536247$, $1462537$.
\end{thmx}
Lastly, with this theorem in hand, we are able to prove that the Asymptotic Shape Conjecture holds for the family of shapes $(n-3,2,1)$ for $n\geq5$, i.e. holds for shape $(2,1)$, using the notation from \Cref{Sec:2.13}. Unlike for the involutions of shape $(n-2,1,1)$, we were unable to verify K{\aa}hrstr{\"o}m's Conjecture for involutions of shape $(n-3,2,1)$. However, it can be deduced (with a little bit of work) that, from the manner in which we answered Kostant's problem for the 13 types of involutions of shape $(n-3,2,1)$, if \cite[Conjecture 41]{CM25-2} and \cite[Conjecture 42]{CM25-2} hold, then K{\aa}hrstr{\"o}m's Conjecture holds for involutions of shape $(n-3,2,1)$.

\subsection*{Acknowledgements}

The first author is partially supported by Vergstiftelsen.
The second author is partially supported by the Swedish Research Council.
All computations were done using the CHEVIE package in GAP3 \cite{Mi15}, and are easily replicable (see the appendix in \cite{CM25-1}). 

\section{Preliminaries}\label{Sec:2}

\subsection{Symmetric group}\label{Sec:2.1}

For $n\geq 0$, let $\mathrm{S}_{n}$ denote the symmetric group of permutations of the set $[n]:=\{1,2,\dots,n\}$ and let $1_{n}\in\mathrm{S}_{n}$ be the identity element. Furthermore, we will denote by
\[ \mathrm{Sim}_{n}:=\{s_{1},s_{2},\dots,s_{n-1}\} \hspace{2mm} \text{ and } \hspace{2mm} \mathrm{Inv}_{n}:=\{x\in\mathrm{S}_{n}\ | \ x^{2}=1_{n}\}, \]
the subsets of $\mathrm{S}_{n}$ consisting of the \emph{simple transpositions} and \emph{involutions}, respectively. Hence, given $i\in[n-1]$,  $s_{i}=(i,i+1)$ is the permutation exchanging $i$ and $i+1$, and fixing everything else.

For $x\in\mathrm{S}_{n}$, a string $s_{i_{l}}s_{i_{l-1}}\dots s_{i_{1}}$, with $s_{i_{j}}\in\mathrm{Sim}_{n}$, is called an \emph{expression of }$x$ if $x=s_{i_{l}}s_{i_{l-1}}\dots s_{i_{1}}$. The expression is \emph{reduced} if $l$ is minimal, in which case $l$ is the \emph{length} of $x$, written $\ell(x):=l$. We will write $\underline{x}=s_{i_{l}}s_{i_{l-1}}\dots s_{i_{1}}$ to refer to both $x$ and the expression $s_{i_{l}}s_{i_{l-1}}\dots s_{i_{1}}$ representing it. 

We let $\leq$ denote the \emph{Bruhat order} on $\mathrm{S}_{n}$: $x\leq y$ if there exists a \emph{substring} of a (equivalently all) reduced expression of $y$ which is a reduced expression of $x$. Let $w_{0,n}\in\mathrm{S}_{n}$ be the maximum element with respect to $\leq$, which is the permutation that reverses the list $1,2,\dots,n$. We will denote by 
\[ \mathrm{Des}_{L}(x):=\{s\in \mathrm{Sim}_{n}\ | \ sx<x\} \hspace{1mm} \text{ and } \hspace{1mm} \mathrm{Des}_{R}(x):=\{s\in \mathrm{Sim}_{n}\ | \ xs<x\}, \]
the \emph{left} and \emph{right descent sets} of $x$, respectively. Note, $\mathrm{Des}_{L}(x)=\mathrm{Des}_{R}(x)$ for all $x\in\mathrm{Inv}_{n}$. Also, we will denote by $\mathrm{Sup}(x)$ the \emph{support of }$x$, which we define as the subset of $\mathrm{Sim}_{n}$ consisting of all simple transpositions which appear in a (equivalently all) reduced expression for $x$.

\subsection{Consecutive patterns}\label{Sec:2.2}

For $x\in\mathrm{S}_{n}$ and $i\in[n]$, we write $x(i)\in[n]$ to denote the image of $i$ under the permutation $x$. We compose permutations as functions from right to left. The \emph{one-line notation} of $x$ is given by $x(1)x(2)\cdots x(n)$, which is considered to be a word with letters in $[n]$. 

For any $x,y\in\mathrm{S}_{n}$, the one-line notation of $xy$ is obtained from that of $x$ by permuting the places of the letters according to $y$. On the other hand, the one-line notation of $xy$ is obtained from that of $y$ by acting on the letters themselves by $x$. As special cases, the one-line notation of $xs_{i}$ is obtained from that of $x$ by swapping the positions of the letters $x(i)$ and $x(i+1)$, while the one-line notation of $s_{i}x$ is obtained from that of $x$ by swapping the letters $i$ and $i+1$ around.

\begin{ex}\label{Ex:2.2:1}
Consider $x=s_{1}s_{2}s_{1}s_{3}\in\mathrm{S}_{4}$, with one-line notation $x=3241$. Then we have that
\[ s_{1}x=s_{2}s_{1}s_{3}=3{\color{teal}1}4{\color{teal}2} \ \text{ and } \ xs_{1}=s_{1}s_{2}s_{3}={\color{red}23}41. \]
The {\color{teal}teal} colours here represent the action of $s_{1}$ on the letters of $x$, swapping $1$ and $2$ around, while {\color{red}red} represents the action of $s_{1}$ on the position of the letters of $x$, swapping $x(1)$ and $x(2)$ around.
\end{ex}

Now, given $m\leq n$, $p\in\mathrm{S}_{m}$, and $x\in\mathrm{S}_{n}$, we say that $x$ \emph{consecutively contains} $p$ \emph{as a pattern} if there exists an $i\in[n]$ such that the letters appearing within the two words
\[ p(1)p(2)\cdots p(m) \hspace{2mm} \text{ and } \hspace{2mm} x(i)x(i+1)\cdots x(i+m-1) \]
share the same relative order. Explicitly, this means that, for any $a,b\in[m]$, the relation $p(a)<p(b)$ holds if and only if $x(a+i-1)<x(b+i-1)$ holds. If this is the case, then we say that the pattern $p$ has \emph{consecutively appeared in} $x$ \emph{at position} $i$. We also say that the pattern $p$ is of \emph{size} $m$.

\begin{ex}\label{Ex:2.2:2}
The pattern $2143$ consecutively appears in $21743856$ at positions $1$ and $4$.
\end{ex}

The following is a well-known result, see for example \cite[Lemma 2.1.4]{BB05}.

\begin{lem}\label{Lem:2.2:3}
Given $x\in\mathrm{S}_{n}$ and $i\in[n-1]$, then $s_{i}\in\mathrm{Des}_{R}(x)$ if and only if $x(i+1)<x(i)$, with the latter condition being equivalent to the pattern $21$ consecutively appearing in $x$ at position $i$.
\end{lem}

\subsection{Compatible Bruhat walks}\label{Sec:2.3}

A pair $x,y\in\mathrm{S}_{n}$ are called \emph{Bruhat neighbours} if $x\leq y$ or $y\leq x$, and $|\ell(x)-\ell(y)|=1$. Equivalently, either $y$ covers $x$ with respect to the Bruhat order, or $x$ covers $y$. We call $x,y\in\mathrm{S}_{n}$ \emph{strong right Bruhat neighbours} if there exists an $s\in\mathrm{Sim}_{n}$ such that $y=xs$. We call any tuple $\mathrm{w}=(z_{1},z_{2},\dots,z_{l})\in\mathrm{S}_{n}^{\times l}$ a \emph{strong right Bruhat walk} if each consecutive pair $z_{i}$ and $z_{i+1}$ are strong right Bruhat neighbours. For a reduced expression $\underline{x}=s_{i_{1}}s_{i_{2}}\dots s_{i_{l}}$, we say that the strong right Bruhat walk $\mathrm{w}$ is $\underline{x}$-\emph{compatible} if, for all valid $k$,
\begin{equation}\label{Eq:2.3:1}
z_{k}s_{i_{k}}<z_{k} \ \text{ and } \ z_{k}s_{i_{k\pm1}}>z_{k}.
\end{equation}

\begin{ex}\label{Ex:2.3:1}
Let $z:=235146\in\mathrm{S}_{6}$, given in one-line notation. Consider the reduced expression $\underline{x}:=s_{3}s_{2}s_{1}$ (so $s_{i_{1}}=s_{3}$, $s_{i_{2}}=s_{2}$, and $s_{i_{3}}=s_{1}$) and the strong right Bruhat walk given by
\[ \mathrm{w}=(z_{1},z_{2},z_{3}):=(z,zs_{3},zs_{3}s_{2})=(235146,231546,213546)\in\mathrm{S}_{6}^{\times 3}. \]
By \Cref{Lem:2.2:3}, one can show that $\mathrm{w}$ is $\underline{x}$-compatible, in other words, the Bruhat relations in \eqref{Eq:2.3:1} are upheld. There are seven such Bruhat relations to check, and we present the details here:
\begin{itemize}
\item[(i)] $zs_{3}<z$ holds since $z(4)=1<5=z(3)$.
\item[(ii)] $(zs_{3})s_{2}<zs_{3}$ holds since $(zs_{3})(3)=1<3=(zs_{3})(2)$.
\item[(iii)] $(zs_{3}s_{2})s_{1}<zs_{3}s_{2}$ holds since $(zs_{3}s_{2})(2)=1<2=(zs_{3}s_{2})(1)$.
\item[(iv)] $zs_{2}>z$ holds since $z(3)=5>3=z(2)$.
\item[(v)] $(zs_{3})s_{1}>zs_{3}$ holds since $(zs_{3})(2)=3>2=(zs_{3})(1)$.
\item[(vi)] $(zs_{3})s_{3}>zs_{3}$ holds since $(zs_{3})(4)=5>1=(zs_{3})(3)$.
\item[(vii)] $(zs_{3}s_{2})s_{2}>zs_{3}s_{2}$ holds since $(zs_{3}s_{2})(3)=3>1=(zs_{3}s_{2})(2)$.
\end{itemize}
\end{ex}

\subsection{Robinson-Schensted correspondence}\label{Sec:2.4}

For $n\geq 0$, let $\Lambda_{n}$ denote the set of all \emph{Young diagrams of size} $n$ (equivalently \emph{integer partitions of} $n$), and $\preceq$ the usual \emph{dominance order} on $\Lambda_{n}$. For $\lambda\in\Lambda_{n}$, let $\mathtt{SYT}_{n}(\lambda)$ denote the set of \emph{standard Young tableaux of shape} $\lambda$, and we denote by
\[ \mathtt{RS}_{n}:\mathrm{S}_{n}\xrightarrow{\sim}\bigsqcup_{\lambda\in\Lambda_{n}}\mathtt{SYT}_{n}(\lambda)\times\mathtt{SYT}_{n}(\lambda) \] 
the \emph{Robinson-Schensted correspondence}, defined by \emph{Schensted's insertion algorithm} from \cite{Sc61}, see also \cite{Sa01}. For $x\in\mathrm{S}_{n}$, we let $\mathtt{RS}_{n}(x)=:(\mathtt{P}_{x},\mathtt{Q}_{x})$. So $\mathtt{P}_{x}$ and $\mathtt{Q}_{x}$ denote the \emph{insertion} and \emph{recording} tableau associated to $x$, respectively. We also let $\mathtt{sh}(x)\in\Lambda_{n}$ denote the underlying Young diagram of $\mathtt{Q}_{x}$, or equivalently of $\mathtt{P}_{x}$, which we call the \emph{shape} of $x$. 

Let $\lambda=(\lambda_{1},\dots,\lambda_{r})\in\Lambda_{n}$ be a Young diagram of size $n$, and let $\lambda^{-}$ denote the Young diagram of size $n-\lambda_{1}$ obtained from $\lambda$ by removing the top row. Naturally, $\lambda$ can be uniquely recovered from knowing $\lambda^{-}$ and $\lambda_{1}$. As such, we will employ the notation $(\lambda^{-})^{\langle\lambda_{1}\rangle}:=\lambda$. Now let $\mathtt{T}\in\mathtt{SYT}_{n}(\lambda)$ be a standard Young tableau of size $n$ and shape $\lambda$, and similarly let $\mathtt{T}^{-}$ denote the standard Young tableau of size $n-\lambda_{1}$ and shape $\lambda^{-}$ obtained from $\mathtt{T}$ by removing the top row (hence the boxes of $\mathtt{T}^{-}$ are filled by the elements of some strict subset of $[n]$). Then $\mathtt{T}$ can be uniquely recovered from knowing $\mathtt{T}^{-}$ and $\lambda_{1}$, hence we will write $(\mathtt{T}^{-})^{\langle\lambda_{1}\rangle}:=\mathtt{T}$. 

It is known that $\mathtt{P}_{x^{-1}}=\mathtt{Q}_{x}$ and $\mathtt{Q}_{x^{-1}}=\mathtt{P}_{x}$, and therefore we have that $\mathtt{P}_{x}=\mathtt{Q}_{x}$ whenever $x\in\mathrm{Inv}_{n}$. Because of this, we will also work with the induced bijection on involutions
\begin{equation}\label{Eq:2.4:0}
\tilde{\mathtt{RS}}_{n}:\mathrm{Inv}_{n}\xrightarrow{\sim}\bigsqcup_{\lambda\in\Lambda_{n}}\mathtt{SYT}_{n}(\lambda), \end{equation}
which sends $x\in\mathrm{Inv}_{n}$ to the standard Young tableau $\mathtt{P}_{x}=\mathtt{Q}_{x}$. We now record the following well known result for latter use, and give an example illustrating the above definitions.

\begin{lem}\label{Lem:2.4:1}
For $w\in\mathrm{S}_{n}$, then $s_{i}\in\mathrm{Des}_{R}(w)$ if and only if the box of $\mathtt{Q}_{w}$ containing $i$ belongs to a higher row than that of the box containing $i+1$. Equivalently, $s_{i}\in\mathrm{Des}_{L}(w)$ if and only if the box of $\mathtt{P}_{w}$ containing $i$ belongs to a higher row than that of the box containing $i+1$.
\end{lem}

\begin{ex}\label{Ex:2.4:2}
Consider $x=s_{1}s_{2}s_{3}s_{4}s_{3}s_{2}s_{1}=5234167\in\mathrm{Inv}_{7}$. Then we have that
\[ \tilde{\mathtt{RS}}_{7}(x)=\mathtt{P}_{x}=\mathtt{Q}_{x}\hspace{1mm}=\hspace{1mm}{\small{\begin{ytableau}1&3&4&6&7\\2\\5\end{ytableau}}}\hspace{1mm}=\hspace{1mm}{\small{\begin{ytableau}2\\5\end{ytableau}}}^{\ \langle5\rangle}, \]
thus $\mathtt{sh}(x)=(5,1,1)=(1,1)^{\langle5\rangle}$. As an example of the dominance order, $(3,2,2)\preceq\mathtt{sh}(x)\preceq(5,2)$. Also, we have that $\mathrm{Des}_{R}(x)=\mathrm{Des}_{L}(x)=\{1,4\}$, and it is precisely the boxes containing $1$ and $4$ in $\mathtt{P}_{x}=\mathtt{Q}_{x}$ which appear in a higher row than the boxes containing their successors.
\end{ex}

\subsection{Hecke algebra}\label{Sec:2.5}

Let $\mathbb{A}:=\mathbb{Z}[v,v^{-1}]$ denote the ring of Laurent polynomials with variable $v$ and coefficients from $\mathbb{Z}$. Given any element $a=a(v)\in\mathbb{A}$ and any integer $i\in\mathbb{Z}$, we will let $[v^{i}](a)\in\mathbb{Z}$ denote the coefficient of $v^{i}$ appearing within the element $a$. 

We will let $\mathrm{H}_{n}^{\mathbb{A}}$ denote the \emph{Hecke algebra associated to} $\mathrm{S}_{n}$, given over the ring $\mathbb{A}$. Then we have the \emph{standard $\mathbb{A}$-basis} $\mathcal{T}_{n}:=\{T_{x}^{(n)} \ | \ x\in\mathrm{S}_{n}\}$ of $\mathrm{H}_{n}^{\mathbb{A}}$, whose element satisfy the relations
\begin{equation}\label{Eq:2.5:1}
T_{x}^{(n)}T_{y}^{(n)}=T_{xy}^{(n)} \hspace{1mm} \text{ and } \hspace{1mm} T_{z}^{(n)}T_{s}^{(n)}=(v^{-1}-v)T_{z}^{(n)}+T_{zs}^{(n)},
\end{equation}
for $x,y\in\mathrm{S}_{n}$ such that $\ell(xy)=\ell(x)+\ell(y)$, and all $z\in\mathrm{S}_{n}$ and $s\in\mathrm{Des}_{R}(z)$. Later we will often work with different Hecke algebras simultaneously, hence we will be adding $n$ as a superscript to much of our notation, as done here for $\mathcal{T}_{n}$, since this will aid in clarifying where objects live. The collection $\mathcal{T}_{n}$, together with the above relations, give a presentation of $\mathrm{H}_{n}^{\mathbb{A}}$ as an $\mathbb{A}$-algebra. Consider the ring epimorphism $\varepsilon:\mathbb{A}\rightarrow\mathbb{Z}$ given by the evaluation $\varepsilon(a(v))=a(1)$. Then we let
\[ \mathrm{H}_{n}^{\mathbb{Z}}:=\mathbb{Z}\otimes_{\mathbb{A}}\mathrm{H}_{n}^{\mathbb{A}}, \]
the $\mathbb{Z}$-algebra obtained from $\mathrm{H}_{n}^{\mathbb{A}}$ by extension of scalars, where $\mathbb{Z}$ is viewed as an $\mathbb{A}$ via $\varepsilon$. For an $h\in\mathrm{H}_{n}^{\mathbb{A}}$, we will abuse notation and simply write $h$ for the element $1\otimes h$ belonging to $\mathrm{H}_{n}^{\mathbb{Z}}$.

Lastly, we setup the following notation/terminology: Let $\mathcal{B}=\{B_{x} \ | \ x\in\mathrm{S}_{n}\}\subset\mathrm{H}_{n}^{\mathbb{A}}$ be an $\mathbb{A}$-basis. For any $h\in\mathrm{H}_{n}^{\mathbb{A}}$, we write $[B_{x}](h)\in\mathbb{A}$ to denote the coefficient of $B_{x}$ appearing in $h$ when expressed in terms of the $\mathbb{A}$-basis $\mathcal{B}$. We write $B_{x}\in h$ to denote $[B_{x}](h)\neq0$, and we say that $B_{x}$ appears in $h$. For $i\in\mathbb{Z}$, we write $v^{i}B_{x}\in h$ to denote that $[v^{i}][B_{x}](h)\neq0$, and in this case we say that $B_{x}$ appears in $h$ at degree $i$, or that $v^{i}B_{x}$ appears in $h$. We adopt the same notation for $\mathrm{H}_{n}^{\mathbb{Z}}$.

\subsection{Kazhdan-Lusztig bases}\label{Sec:2.6}

Any standard basis element of $\mathrm{H}_{n}^{\mathbb{A}}$ is invertible, which allows one to define the \emph{bar involution} $\overline{(-)}:\mathrm{H}_{n}^{\mathbb{A}}\rightarrow\mathrm{H}_{n}^{\mathbb{A}}$, a ring involution defined on the standard basis by
\[ \overline{T^{(n)}_{x}}:=\left(T_{x^{-1}}^{(n)}\right)^{-1}, \]
and on the variable $v$ by $\overline{v}:=v^{-1}$. Originally introduced within \cite{KL79}, we let $\mathcal{C}_{n}:=\{C_{x}^{(n)} \ | \ x\in\mathrm{S}_{n}\}$ denote the \emph{Kazhdan-Lusztig basis}, however, we are using the normalisation from \cite{So07}. 

The basis element $C_{x}^{(n)}$ is uniquely determined by the following properties:
\begin{equation}\label{Eq:2.6:1}
\overline{C^{(n)}_{x}}=C_{x}^{(n)}, \hspace{1mm} \text{ and } \hspace{1mm} C_{x}^{(n)}=T_{x}^{(n)}+\sum_{y<x}p_{x,y}^{(n)}T_{y}^{(n)} \hspace{1mm} \text{ where } \hspace{1mm} p_{x,y}^{(n)}\in v\mathbb{Z}[v].
\end{equation}
Extending the definition of these polynomials $p_{x,y}^{(n)}$ by setting $p_{x,y}^{(n)}:=0$ whenever $x\not<y$, and $p_{x,x}^{(n)}:=1$, then these are the well known \emph{Kazhdan-Lusztig polynomials}. Also, for $y\leq x$, the \emph{Kazhdan-Lusztig $\mu$-function} $\mu:\mathrm{S}_{n}\times\mathrm{S}_{n}\rightarrow\mathbb{Z}$ is defined as the coefficient
\[ \mu^{(n)}(x,y)=\mu^{(n)}(y,x):=[v]p_{x,y}^{(n)}. \]
It is known that $\mu^{(n)}(x,y)=1$ if $x$ and $y$ are Bruhat neighbours, and $\mu^{(n)}(x,y)=0$ if $\ell(x)-\ell(y)\in2\mathbb{Z}$. Also, for any $w\in\mathrm{S}_{n}$ and $s\in\mathrm{Sim}_{n}$, we have
\begin{equation}\label{Eq:2.6:1.5R}
C_{w}^{(n)}C_{s}^{(n)}=
\begin{cases}
(v+v^{-1})C_{w}^{(n)}, & ws<w, \\
C_{ws}^{(n)}+{\displaystyle\sum_{\substack{x<w \\ xs<x}}\mu^{(n)}(x,w)C_{x}^{(n)}}, & ws>w.
\end{cases}
\end{equation}
\begin{equation}\label{Eq:2.6:1.5L}
C_{s}^{(n)}C_{w}^{(n)}=
\begin{cases}
(v+v^{-1})C_{w}^{(n)}, & sw<w, \\
C_{sw}^{(n)}+{\displaystyle\sum_{\substack{x<w \\ sx<x}}\mu^{(n)}(x,w)C_{x}^{(n)}}, & sw>w.
\end{cases}
\end{equation}

\begin{lem}\label{Lem:2.6:2}
For $x,y\in\mathrm{S}_{n}$ such that $\mathrm{Sup}(x)\cap\mathrm{Sup}(y)=\emptyset$, we have $C_{x}^{(n)}C_{y}^{(n)}=C_{xy}^{(n)}$.
\end{lem}

\begin{proof}
Since $\mathsf{Sup}(x)\cap\mathsf{Sup}(y)=\emptyset$, then, given any $a\leq x$ and $b\leq y$, we have that $\ell(ab)=\ell(a)+\ell(b)$. Therefore, we see that we have the equality of sets
\[ \{c\in\mathrm{S}_{n} \ | \ c\leq xy\}=\{ab\in\mathrm{S}_{n} \ | \ a\leq x, b\leq y \}. \] 
Hence, by \Cref{Eq:2.5:1} we have that
\[ C_{x}^{(n)}C_{y}^{(n)}=\left(T_{x}^{(n)}+\sum_{a<x}p_{a,x}^{(n)}T_{a}^{(n)}\right)\left(T_{y}^{(n)}+\sum_{b<y}p_{b,y}^{(n)}T_{b}^{(n)}\right)=T_{xy}^{(n)}+\sum_{\substack{ab<xy \\ a\leq x, b\leq y}}p_{a,x}^{(n)}p_{b,y}^{(n)}T_{ab}^{(n)}. \]
Moreover, $C_{x}^{(n)}C_{y}^{(n)}$ is naturally invariant under the bar involution. Therefore, $C_{x}^{(n)}C_{y}^{(n)}$ satisfies the defining properties of the element $C_{xy}^{(n)}$ as given in \Cref{Eq:2.6:1}, hence they must be equal.
\end{proof}

\subsection{Lusztig's $\mathbf{a}$-function}\label{Sec:2.7}

Given any $x,y\in\mathrm{S}_{n}$, it is well known that 
\begin{equation}\label{Eq:2.7:0}
C_{x}^{(n)}C_{y}^{(n)}=\sum_{z\in\mathrm{S}_{n}}\gamma_{x,y,z}^{(n)}C_{z}^{(n)}, \hspace{1mm} \text{ where } \hspace{1mm} \gamma_{x,y,z}^{(n)}\in\mathbb{A}_{\geq0}:=\mathbb{Z}_{\geq0}[v,v^{-1}].
\end{equation}
The fact that these structure constants $\gamma_{x,y,z}^{(n)}$ have non-negative coefficients is a consequence of the Kazhdan-Lusztig conjecture
(now theorem). By \Cref{Eq:2.6:1}, one can deduce that we have
\begin{equation}\label{Eq:2.7:1}
\gamma_{x,y,xy}^{(n)}\neq0 \hspace{1mm} \text{ whenever } \hspace{1mm} \ell(xy)=\ell(x)+\ell(y).
\end{equation}
The \emph{Lusztig's $\mathbf{a}$-function} $\mathbf{a}:\mathrm{S}_{n}\rightarrow\mathbb{Z}_{\geq0}$ is defined as follows:
\[ \mathbf{a}_{n}(z):=\mathrm{max}\{\mathrm{deg}(\gamma_{x,y,z}^{(n)}) \ | \ x,y\in\mathrm{S}_{n}\}. \]
In our setting (Type A), this function is equivalent to $\mathbf{a}_{n}(z)=\ell(w_{0,n}^{P})$ whenever $\mathtt{sh}(z)=\mathtt{sh}(w_{0,n}^{P})$, where $P\subset\mathrm{Sim}_{n}$ and $w_{0,n}^{P}$ is the longest element in the parabolic subgroup $\langle P\rangle\subseteq\mathrm{S}_{n}$.

\subsection{Dual Kazhdan-Lusztig bases}\label{Sec:2.8}

Let $\tau_{n}:\mathrm{H}_{n}^{\mathbb{A}}\rightarrow\mathbb{A}$ be the \emph{standard trace}, which is the $\mathbb{A}$-linear map defined on the standard basis by $\tau_{n}(T_{x}^{(n)})=\delta_{x,1_{n}}$ with $\delta$ the Kronecker delta function. 

The \emph{dual Kazhdan-Lusztig basis} $\mathcal{D}_{n}:=\{D_{x}^{(n)} \ | \ x\in\mathrm{S}_{n}\}$ is uniquely defined by $\tau_{n}(D_{x}^{(n)}C_{y^{-1}}^{(n)})=\delta_{x,y}$, for all $y\in\mathrm{S}_{n}$. For any $x\in\mathrm{S}_{n}$ and $s\in\mathrm{Sim}_{n}$, using \cite[Proposition 10.8]{Lu03} we have
\begin{equation}\label{Eq:2.8:1}
D_{x}^{(n)}C_{s}^{(n)}=
\begin{cases}
(v+v^{-1})D_{x}^{(n)}+D_{xs}^{(n)}+{\displaystyle\sum_{x<y<ys}\mu^{(n)}(x,y)D_{y}^{(n)}}, & s\in\mathrm{Des}_{R}(x), \\
0, & s\not\in\mathrm{Des}_{R}(x).
\end{cases}
\end{equation}
By \cite[Proposition 3.3]{KMM23}, for any $a,b,c\in\mathrm{S}_{n}$ we have the equality of coefficients
\begin{equation}\label{Eq:2.8:2}
[D_{a}^{(n)}](D_{b}^{(n)}C_{c}^{(n)})=[C_{b}^{(n)}](C_{a}^{(n)}C_{c^{-1}}^{(n)}).
\end{equation}

\subsection{Kazhdan-Lusztig preorders}\label{Sec:2.9}

For $x,y\in\mathrm{S}_{n}$, write $x\leq_{R}^{(n)}y$ if there exists $h\in\mathrm{H}_{n}^{\mathbb{A}}$ where
\[ [C_{y}^{(n)}](C_{x}^{(n)}h)\neq0. \]
This gives a preorder on $\mathrm{S}_{n}$ called the \emph{right Kazhdan-Lusztig preorder}. The \emph{left Kazhdan-Lusztig preorder} $\leq_{L}^{(n)}$ is defined similarly with left multiplication. The induced equivalence relations are denoted by $\sim_{R}^{(n)}$ and $\sim_{L}^{(n)}$, where the respective equivalence classes are called \emph{right} and \emph{left cells}.

From both \cite[Section 5]{KL79} and \cite[Theorem 5.1]{Ge06}, for example, we have that
\begin{equation}\label{Eq:2.9:1}
x\sim_{L}^{(n)}y \iff \mathtt{Q}_{x}=\mathtt{Q}_{y}, \hspace{1mm} \text{ and } \hspace{2mm} x\sim_{R}^{(n)}y \iff \mathtt{P}_{x}=\mathtt{P}_{y}.
\end{equation}
\begin{equation}\label{Eq:2.9:2}
x\leq_{R}^{(n)}y \iff x^{-1}\leq_{L}^{(n)}y^{-1}.
\end{equation}
\begin{equation}\label{Eq:2.9:3}
x\leq_{R}^{(n)}y \iff w_{0,n}xw_{0,n}\leq_{R}^{(n)}w_{0,n}yw_{0,n}.
\end{equation}
\begin{equation}\label{Eq:2.9:4}
x\leq_{R}^{(n)}y \implies \mathrm{Des}_{L}(x)\subset\mathrm{Des}_{L}(y) \ \text{ and } \  \mathtt{sh}(y)\preceq\mathtt{sh}(x).
\end{equation}
\begin{equation}\label{Eq:2.9:4.5}
x\leq_{R}^{(n)}y \hspace{1mm} \text{ and } \hspace{1mm} \mathtt{sh}(y)=\mathtt{sh}(x) \implies x\sim_{L}^{(n)}y.
\end{equation}
The equations above tell us that each left/right cell contains a unique involution. It is also worth mentioning that our convention on the direction of the left and right Kazhdan-Lusztig orders is often the opposite to that employed by others, including \cite{Ge06} and \cite{KL79}. In particular, for us the identity $1_{n}$ is the minimum element under both preorders.

Now, from \cite[Lemma 12]{MM11}, for example, and \cite[Proposition 27]{CM25-2}), we have that
\begin{equation}\label{Eq:2.9:5}
D_{z}^{(n)}C_{x}^{(n)}\neq0 \iff x\leq_{R}^{(n)}z^{-1}.
\end{equation}
\begin{equation}\label{Eq:2.9:6}
D_{z}^{(n)}C_{x}^{(n)}=D_{z}^{(n)}C_{y}^{(n)}\neq0 \implies x\sim_{L}^{(n)}y.
\end{equation}
Lastly, we record a well-known property which we will often use without reference, and which follows from \Cref{Eq:2.8:1,Eq:2.9:2,Eq:2.9:5}. Given any $z\in\mathrm{S}_{n}$ and $s_{i}\in\mathrm{Sim}_{n}$, we have that
\begin{equation}\label{Eq:2.9:7}
D_{z}^{(n)}C_{s_{i}}^{(n)}\neq0 \iff s_{i}\leq_{L}^{(n)}z \iff s_{i}\in\mathrm{Des}_{R}(z).
\end{equation}

\subsection{Categories $^{\mathbb{Z}}\mathcal{O}_{0}$ and $^{\mathbb{Z}}\mathcal{P}_{0}$}\label{Sec:2.10}

We let $\mathfrak{sl}_{n}:=\mathfrak{sl}_{n}(\mathbb{C})$ be the \emph{complex special linear Lie algebra} of all traceless $n\times n$ matrices, and we consider the \emph{standard triangular decomposition} 
\begin{equation}\label{Eq:2.10:1}
\mathfrak{sl}_{n}=\mathfrak{n}^{-}\oplus\mathfrak{h}\oplus\mathfrak{n}^{+}.
\end{equation}
The Weyl group of $(\mathfrak{sl}_{n},\mathfrak{h})$ is $\mathrm{S}_{n}$, and it acts naturally on the dual space $\mathfrak{h}^{*}$. For $\rho\in\mathfrak{h}^{*}$ the half sum of all positive roots, $x\in\mathrm{S}_{n}$, and $\lambda\in\mathfrak{h}^{*}$, then the \emph{dot-action} is given by $x\cdot\lambda:=x(\lambda+\rho)-\rho$. 

We let $\mathcal{O}^{(n)}:=\mathcal{O}(\mathfrak{sl}_{n})$ denote the \emph{BGG category} associated to \Cref{Eq:2.10:1} (see \cite{BGG76,Hu08}). The simple objects of $\mathcal{O}^{(n)}$ are (up to isomorphism) the \emph{simple highest weight modules} $L(\lambda)$ for $\lambda\in\mathfrak{h}^{*}$. We let $\mathcal{O}_{0}^{(n)}:=\mathcal{O}_{0}(\mathfrak{sl}_{n})$ be the \emph{principal block}, the indecomposable summand of $\mathcal{O}^{(n)}$ containing the trivial $\mathfrak{sl}_{n}$-module. The simple objects of $\mathcal{O}_{0}^{(n)}$ are indexed by the elements of $\mathrm{S}_{n}$ via $L_{x}^{(n)}:=L(x\cdot 0)$. We denote the indecomposable projective cover of $L_{x}^{(n)}$ within $\mathcal{O}_{0}^{(n)}$ by $P_{x}^{(n)}$. 

The principal block $\mathcal{O}_{0}^{(n)}$ is equivalent to the left module category for some finite-dimensional basic associative algebra $A$ (unique up to isomorphism). This algebra $A$ is Koszul by \cite{So90}, and therefore admits a Koszul $\mathbb{Z}$-grading. We let $^{\mathbb{Z}}\mathcal{O}_{0}^{(n)}$ denote the corresponding $\mathbb{Z}$-graded version of $\mathcal{O}_{0}^{(n)}$, and $\langle 1\rangle$ the grading shift functor decreasing the degree by one (see \cite{St03}). For each $x\in\mathrm{S}_{n}$, the module $L_{x}^{(n)}$ admits a graded lift, placing it in degree $0$, and we use the same symbol to denote it. 

A \emph{projective functor} of $\mathcal{O}_{0}^{(n)}$ is a direct summand of $(V\otimes-)$ with $V$ some finite dimensional $\mathfrak{sl}_{n}$-module. Let $\mathcal{P}_{0}^{(n)}:=\mathcal{P}_{0}(\mathfrak{sl}_{n})$ be the monoidal category of such functors. By \cite[Theorem~3.3]{BG80}, the indecomposable objects in $\mathcal{P}_{0}^{(n)}$ are in bijection with $\mathrm{S}_{n}$, where, for each $x\in\mathrm{S}_{n}$, we let $\theta_{x}^{(n)}\in\mathcal{P}_{0}^{(n)}$ be the unique (up to isomorphism) indecomposable projective functor normalized such that $\theta_{x}^{(n)}P_{1_{n}}^{(n)}\cong P_{x}^{(n)}$. By \cite[Theorem 8.2]{St03}, each $\theta_{x}^{(n)}$ admits a graded lift to an endofunctor of $^{\mathbb{Z}}\mathcal{O}_{0}^{(n)}$, which we represent with the same symbol. We let $^{\mathbb{Z}}\mathcal{P}_{0}^{(n)}$ denote the $\mathbb{Z}$-graded version of $\mathcal{P}_{0}^{(n)}$.

We recall two results which will be used in later sections. The first is well-known (see for example \cite[Proposition 46]{CMZ19}), and the second follows from the natural symmetry of the root system of $\mathfrak{sl}_{n}$, corresponding on $\mathrm{S}_{n}$ to conjugation by $w_{0,n}$.

\begin{prop}\label{Prop:2.10:2}
Let $z\in\mathrm{S}_{n}$ and $s\in S_{n}$. Then $\theta_{s}^{(n)}L_{z}^{(n)}\neq0$ if and only if $s\in\mathrm{Des}_{R}(z)$. In this case $\theta_{s}^{(n)}L_{z}^{(n)}$ is indecomposable, of graded length three, has simple top $L_{w}^{(n)}$ in degree $-1$, simple socle $L_{w}^{(n)}$ in degree $1$, and semi-simple module $\mathrm{J}_{s}^{(n)}(z)$ in degree zero \emph{(}\emph{Jantzen middle}\emph{)} where 
\[ \mathrm{J}_{s}^{(n)}(z)\cong L_{zs}^{(n)}\oplus\bigoplus_{z<x<xs}(L_{x}^{(n)})^{\oplus \mu(z,x)}. \]
\end{prop}

\begin{lem}\label{Lem:2.10:3}
We have auto-equivalences $F:{^{\mathbb{Z}}}\mathcal{O}_{0}^{(n)}\xrightarrow{\sim}{^{\mathbb{Z}}}\mathcal{O}_{0}^{(n)}$ and $G:{^{\mathbb{Z}}}\mathcal{P}_{0}^{(n)}\xrightarrow{\sim}{^{\mathbb{Z}}}\mathcal{P}_{0}^{(n)}$ given by
\[ L_{x}^{(n)}\mapsto L_{w_{0,n}xw_{0,n}}^{(n)} \ \text{ and } \ \theta_{x}^{(n)}\mapsto\theta_{w_{0,n}xw_{0,n}}^{(n)}, \]
respectively. Moreover, the natural action of   $^{\mathbb{Z}}\mathcal{P}_{0}^{(n)}$ on $^{\mathbb{Z}}\mathcal{O}_{0}^{(n)}$ commutes with these auto-equivalences, meaning that, for any $x,z\in\mathrm{S}_{n}$, we have $F(\theta_{x}^{(n)}L_{z}^{(n)})=G(\theta_{x}^{(n)})F(L_{z}^{(n)})$.
\end{lem}

\subsection{Decategorification}\label{Sec:2.11}

We let $\mathbf{Gr}(^{\mathbb{Z}}\mathcal{O}_{0}^{(n)})$ and $\mathbf{Gr}(\mathcal{O}_{0}^{(n)})$ be the Grothendieck groups of $^{\mathbb{Z}}\mathcal{O}_{0}^{(n)}$ and $\mathcal{O}_{0}^{(n)}$, respectively. We view the former as an $\mathbb{A}$-module where $v$ acts by the shift functor $\langle -1\rangle$. For a module $M$ in either $^{\mathbb{Z}}\mathcal{O}_{0}^{(n)}$ or $\mathcal{O}_{0}^{(n)}$, let $[M]$ denote its image in the respective Grothendieck group. Then by \cite{BB,BK,So92}, we have the two module isomorphism 
\[ \mathbf{Gr}(^{\mathbb{Z}}\mathcal{O}_{0}^{(n)})\xrightarrow{\sim}\mathrm{H}_{n}^{\mathbb{A}} \ \text{ and } \ \mathbf{Gr}(\mathcal{O}_{0}^{(n)})\xrightarrow{\sim}\mathrm{H}_{n}^{\mathbb{Z}}, \]
where both are defined by $[L_{x}^{(n)}]\mapsto D_{x}^{(n)}$, for $x\in\mathrm{S}_{n}$. Recall that the dual Kazhdan-Lusztig basis of $\mathrm{H}_{n}^{\mathbb{Z}}$ is obtained from that of $\mathrm{H}_{n}^{\mathbb{A}}$ by projecting through $\varepsilon:\mathrm{H}_{n}^{\mathbb{A}}\rightarrow\mathrm{H}_{n}^{\mathbb{Z}}$. In particular, the projection $\varepsilon$ is the decategorification of the forgetful functor $^{\mathbb{Z}}\mathcal{O}_{0}^{(n)}\rightarrow\mathcal{O}_{0}^{(n)}$ which simply forgets the grading.

We also let $\mathbf{Gr}({^{\mathbb{Z}}}\mathcal{P}_{0}^{(n)})$ and $\mathbf{Gr}(\mathcal{P}_{0}^{(n)})$ denote the split Grothendieck rings of ${^{\mathbb{Z}}}\mathcal{P}_{0}^{(n)}$ and $\mathcal{P}_{0}^{(n)}$, respectively. We view the former as an $\mathbb{A}$-algebra where $v$ again acts by $\langle -1\rangle$. For a projective functor $\theta$ belonging to either ${^{\mathbb{Z}}}\mathcal{P}_{0}^{(n)}$ or $\mathcal{P}_{0}^{(n)}$, let $[\theta]$ denote its image in the respective split Grothendieck ring. By, for example \cite{So92, Ma12}, we have the two ring isomorphism
\[ \mathbf{Gr}(^{\mathbb{Z}}\mathcal{P}_{0}^{(n)})\xrightarrow{\sim}(\mathrm{H}_{n}^{\mathbb{A}})^{\mathsf{op}} \ \text{ and } \ \mathbf{Gr}(\mathcal{P}_{0}^{(n)})\xrightarrow{\sim}(\mathrm{H}_{n}^{\mathbb{Z}})^{\mathsf{op}}, \]
where both are defined by $[\theta_{x}^{(n)}]\mapsto C_{x}^{(n)}$, for $x\in\mathrm{S}_{n}$. The Kazhdan-Lusztig basis of $\mathrm{H}_{n}^{\mathbb{Z}}$ is obtained from that of $\mathrm{H}_{n}^{\mathbb{A}}$ by projecting through $\varepsilon:\mathrm{H}_{n}^{\mathbb{A}}\rightarrow\mathrm{H}_{n}^{\mathbb{Z}}$, and again, the projection $\varepsilon$ is the decategorification of the forgetful functor $^{\mathbb{Z}}\mathcal{P}_{0}^{(n)}\rightarrow\mathcal{P}_{0}^{(n)}$. Furthermore, the natural action of $^{\mathbb{Z}}\mathcal{P}_{0}^{(n)}$ on $^{\mathbb{Z}}\mathcal{O}_{0}^{(n)}$, and $\mathcal{P}_{0}^{(n)}$ on $\mathcal{O}_{0}^{(n)}$, decategorify to the regular right action of $\mathrm{H}_{n}^{\mathbb{A}}$ and $\mathrm{H}_{n}^{\mathbb{Z}}$, respectively. 

From above, $[\theta_{x}^{(n)}L_{z}^{(n)}]=[L_{z}^{(n)}][\theta_{x}^{(n)}]=D_{z}^{(n)}C_{x}^{(n)}$, and so \Cref{Eq:2.9:5,Eq:2.9:6} imply
\begin{equation}\label{Eq:2.11:1}
\theta_{x}^{(n)}L_{z}^{(n)}\neq0 \iff x\leq_{R}^{(n)}z^{-1}, \ \text{and}
\end{equation}
\begin{equation}\label{Eq:2.11:2}
\theta_{x}^{(n)}L_{z}^{(n)}\cong\theta_{y}^{(n)}L_{z}^{(n)}\neq0 \implies x\sim_{L}^{(n)}y.
\end{equation}
Moreover, suppose we have $x\sim_{L}^{(n)}y$ and $x'\sim_{L}^{(n)}y'$ such that $x\sim_{R}^{(n)}x'$ and $y\sim_{R}^{(n)}y'$. Then it was proven in \cite[Proposition 32]{CM25-2}, that, for any $z\in\mathrm{S}_{n}$, we have the equivalence
\begin{equation}\label{Eq:2.11:3}
\theta_{x}^{(n)}L_{z}^{(n)}\cong\theta_{y}^{(n)}L_{z}^{(n)}\neq0 \iff \theta_{x'}^{(n)}L_{z}^{(n)}\cong\theta_{y'}^{(n)}L_{z}^{(n)}\neq0.
\end{equation}
As such, if one finds a pair $x,y\in\mathrm{S}_{n}$ such that $\theta_{x}L_{z}\cong\theta_{y}L_{z}\neq0$, then one immediately has a family of pairs which give analogous isomorphisms obtained by matching up the right cells of $x$ and $y$ via the left equivalence $\sim_{L}^{(n)}$. It is worth remarking that we only know that this holds categorically, that is to say, we do not know whether the analogous decategorified statement to \Cref{Eq:2.11:3} holds. This has been conjectured in \cite[Conjectures 39 and 40]{CM25-2}.

We conclude by recalling the following two results: The first is implicit in \cite{KMM23}, and is partially proved in \cite[Lemma 5.3]{CM25-1}, and the second follows from \cite[Lemma 3.5]{CM25-1}.

\begin{lem}\label{Lem:2.11:4}
For any $z\in\mathrm{S}_{n}$ and involution $x\in\mathrm{Inv}_{n}$, the following hold:
\begin{itemize}
\item[(1)] We have the two equivalences
\[ \theta_{x}^{(n)}L_{z}^{(n)}\neq0 \iff [v^{\mathbf{a}_{n}(x)}][D_{z}^{(n)}](D_{z}^{(n)}C_{x}^{(n)})\neq0 \iff [D_{z}^{(n)}](D_{z}^{(n)}C_{x}^{(n)})\neq0. \]
\item[(2)] The module $\theta_{x}^{(n)}L_{z}^{(n)}$ is indecomposable whenever $[v^{\mathbf{a}_{n}(x)}][D_{z}^{(n)}](D_{z}^{(n)}C_{x}^{(n)})=1$.
\end{itemize} 
\end{lem}

\begin{proof}
In the proof of \cite[Proposition 5.4]{KMM23}, see also \cite[Lemma 5.3]{CM25-1}, it was shown that
\[ \mathrm{dim}\left(\mathrm{Hom}_{^{\mathbb{Z}}\mathcal{O}_{0}}(\theta_{x}^{(n)}L_{z}^{(n)}, \theta_{x}^{(n)}L_{z}^{(n)})\right)=\mathrm{dim}\left(\mathrm{Hom}_{^{\mathbb{Z}}\mathcal{O}_{0}}(\theta_{x}^{(n)}L_{z}^{(n)},L_{z}^{(n)}\langle\mathbf{a}_{n}(x)\rangle)\right), \]
where morphisms here are \emph{homogeneous of degree zero}. The right hand side here is clearly bounded above by the multiplicity of $L_{z}^{(n)}\langle\mathbf{a}_{n}(x)\rangle$ within $\theta_{x}^{(n)}L_{z}^{(n)}$, and so from the above discussions, we have
\[ \mathrm{dim}\left(\mathrm{Hom}_{^{\mathbb{Z}}\mathcal{O}_{0}}(\theta_{x}^{(n)}L_{z}^{(n)}, \theta_{x}^{(n)}L_{z}^{(n)})\right)\leq d:=[v^{\mathbf{a}_{n}(x)}][D_{z}^{(n)}](D_{z}^{(n)}C_{x}^{(n)})\in\mathbb{Z}_{\geq0}. \]

To prove Claim~(1): For the first equivalence, we have that $\theta_{x}^{(n)}L_{z}^{(n)}\neq0$ if and only if the above endomorphism space is non-zero since such a space contains the identity morphism. Hence, $\theta_{x}^{(n)}L_{z}^{(n)}\neq0$ if and only if $1\leq d$ if and only if $d\neq0$, since $d$ is a non-negative integer. This gives the first equivalence, for the second, the forward implication is immediate, while for the latter implication we have that $[D_{z}^{(n)}](D_{z}^{(n)}C_{x}^{(n)})\neq0$ which implies $D_{z}^{(n)}C_{x}^{(n)}\neq0$ and, in turn, $\theta_{x}^{(n)}L_{z}^{(n)}\neq0$, and by the first equivalence, this implies $[v^{\mathbf{a}_{n}(x)}][D_{z}^{(n)}](D_{z}^{(n)}C_{x}^{(n)})\neq0$, completing the proof of Claim~(1).

To prove Claim~(2): If $d=1$, then the dimension of the above endomorphism space of $\theta_{x}^{(n)}L_{z}^{(n)}$ is bounded above by 1, but it is also bounded below by 1 since $\theta_{x}^{(n)}L_{z}^{(n)}\neq0$, which is implied by the fact that
\[ [\theta_{x}^{(n)}L_{z}^{(n)}]=D_{z}^{(n)}C_{x}^{(n)} \hspace{1mm} \text{ and } \hspace{1mm} d=[v^{\mathbf{a}_{n}(x)}][D_{z}^{(n)}](D_{z}^{(n)}C_{x}^{(n)})=1. \]
So the dimension of this endomorphism space is indeed $1$. Claim~(2) now 
follows by \cite[Corollary~5.2]{KMM23}.
\end{proof}

\begin{lem}\label{Lem:2.11:5}
Let $\underline{x}=s_{i_{1}}s_{i_{2}}\cdots s_{i_{l}}\in\mathrm{S}_{n}$ be a reduced expression and $\mathrm{w}=(w_{1},\dots,w_{l})\in\mathrm{S}_{n}^{\times l}$ a strong right Bruhat walk which is $\underline{x}$-compatible, as described in \Cref{Sec:2.3}. Then we have that
\[ \emph{(a)} \hspace{2mm} D_{w_{1}}^{(n)}C_{s_{i_{1}}}^{(n)}C_{s_{i_{2}}}^{(n)}\cdots C_{s_{i_{l}}}^{(n)}=D_{w_{l}}^{(n)}C_{s_{i_{l}}}^{(n)}\neq0, \hspace{2mm} \text{ and } \hspace{2mm} \emph{(b)} \hspace{2mm} \theta_{s_{i_{l}}}^{(n)}\cdots\theta_{s_{i_{2}}}^{(n)}\theta_{s_{i_{1}}}^{(n)}L_{w_{1}}^{(n)}\cong\theta_{s_{i_{l}}}^{(n)}L_{w_{l}}^{(n)}\neq0. \]
\end{lem}

\begin{proof}
Item (b) is precisely (ii) of \cite[Lemma 3.5]{CM25-1}. Decategorifying this gives (a), noting that the natural action of ${^\mathbb{Z}}\mathcal{P}_{0}^{(n)}$ on ${^\mathbb{Z}}\mathcal{O}_{0}^{(n)}$ decategorifies to the \emph{right} regular action of $\mathrm{H}_{n}^{\mathbb{A}}$ on itself. It is worth mentioning that the definition of $\underline{x}$-compatibility in \cite{CM25-1} corresponds to $\underline{x^{-1}}$-compatibility for us. This difference occurs as \cite[Lemma 3.5]{CM25-1} gives the categorical statement, written with a left action, while here we focus on the decategorified statement, written with a right action. 
\end{proof}

\subsection{Indecomposability conjecture}\label{Sec:2.12}

For $x,z\in\mathrm{S}_{n}$, let $\mathbf{KM}_{n}(x,z)\in\{\mathtt{true}, \mathtt{false}\}$ be the \emph{truth value} of the statement ``the module  $\theta_{x}^{(n)}L_{z}^{(n)}$ in $^{\mathbb{Z}}\mathcal{O}_{0}^{(n)}$ (or equivalently $\mathcal{O}_{0}^{(n)}$) is zero or indecomposable''. Let $\mathbf{KM}_{n}(x,\star)$ be the conjunction of $\mathbf{KM}_{n}(x,z)$ for all $z\in\mathfrak{S}_{n}$, and define $\mathbf{KM}_{n}(\star,z)$ similarly. Then the \emph{Indecomposability Conjecture}, first presented in \cite{KiM16}, is the following:

\begin{conj}\label{Conj:2.12:1}
For all $z\in\mathrm{S}_{n}$ we have that $\mathbf{KM}_{n}(\star,z)=\mathtt{true}$.
\end{conj}

This conjecture has been confirmed up to $n=7$, and many special cases and properties are known, some of which we recall now: By \cite[Proposition 2]{CMZ19} and \cite[Corollary 5.6]{CM25-1}, we have
\begin{equation}\label{Eq:2.12:2}
x\sim_{R}^{(n)}x' \implies \mathbf{KM}_{n}(x,\star)=\mathbf{KM}_{n}(x',\star).
\end{equation}
\begin{equation}\label{Eq:2.12:3}
z\sim_{L}^{(n)}z' \implies \mathbf{KM}_{n}(\star,z)=\mathbf{KM}_{n}(\star,z').
\end{equation}
\begin{equation}\label{Eq:2.12:4}
|\mathrm{Sup}(x)|\leq6 \implies \mathbf{KM}_{n}(x,\star)=\mathtt{true}.
\end{equation}
Moreover, from \Cref{Lem:2.10:3} and \Cref{Eq:2.9:3}, we have, for any $z\in\mathrm{S}_{n}$, that
\begin{equation}\label{Eq:2.12:5}
\mathbf{KM}_{n}(\star,z)=\mathbf{KM}_{n}(\star,w_{0,n}zw_{0,n}).
\end{equation}
It was also shown in \cite[Section 4.3]{MMM24} that $\mathbf{KM}_{n}(\star,z)=\mathtt{true}$ for all $z\in\mathrm{S}_{n}$ such that $z$ is fully commutative, that is to say, all $z$ with shape $(n-m,m)$, for any $m\leq n$. In this paper, we similarly confirm the indecomposability conjecture for all $z$ with shape $(n-3,2,1)$ and $(n-2,1,1)$.

\subsection{Kostant's problem}\label{Sec:2.13}

For $z\in\mathrm{S}_{n}$, and specialising to our setting, \emph{Kostant's problem} asks if the universal enveloping algebra of $\mathfrak{sl}_{n}$ surjects onto the adjointly finite endomorphism algebra of $L_{z}^{(n)}$ (see for example \cite{Jo80}). If true, we call $z$ \emph{Kostant positive}, otherwise it is called \emph{Kostant negative}. For our purposes, we will be exclusively working with the following equivalent reformulation of Kostant's problem given by Theorem 8.16 within \cite{KMM23}:

\begin{thm}\label{Thm:2.13:1}
An element $z\in\mathrm{S}_{n}$ is Kostant positive if and only if the following hold:
\begin{itemize}
\item[(1)] For any $x\leq_{R}^{(n)}z^{-1}$, the module $\theta_{x}^{(n)}L_{z}^{(n)}$ is indecomposable.
\item[(2)] For any distinct $x,y\leq_{R}^{(n)}z^{-1}$, then $\theta_{x}^{(n)}L_{z}^{(n)}\not\cong\theta_{y}^{(n)}L_{z}^{(n)}$ in $^{\mathbb{Z}}\mathcal{O}_{0}^{(n)}$.
\end{itemize}
\end{thm}

By \Cref{Eq:2.11:1}, we see that (1) above is equivalent to $\mathbf{KM}_{n}(\star,z)=\mathtt{true}$. We let $\mathbf{Kh}_{n}(z)$ denote the truth value of statement (2) above, and we let $\mathbf{K}_{n}(z)$ denote the truth value of the conjunction of both (1) and (2). In other words, $\mathbf{K}_{n}(z)=\mathtt{true}$ if and only if $z$ is Kostant positive.

Kostant's problem has been answered up to $n=7$, and many special cases and properties are known. Important for the results of this paper, Kostant's problem is a left cell invariant. That is,
\begin{equation}\label{Eq:2.13:2}
z\sim_{L}^{(n)}z' \implies \mathbf{K}_{n}(z)=\mathbf{K}_{n}(z').
\end{equation}
This is given by \cite[Theorem 61]{MS08}. Also, both (1) and (2) from \Cref{Thm:2.13:1} are clearly preserved under the auto-equivalence of \Cref{Lem:2.10:3}. Hence, for all $z\in\mathrm{S}_{n}$, we have that
\begin{equation}\label{Eq:2.13:3}
\mathbf{K}_{n}(z)=\mathbf{K}_{n}(w_{0,n}zw_{0,n}).
\end{equation}
In \cite{MMM24}, Kostant's problem was answered for all fully commutative permutations, that is, all permutations of shape $(n,m)$ for any $m\leq n$. We now recall a conjecture in \cite[Section 6.9]{MMM24}. Let $\lambda=(\lambda_{1},\dots,\lambda_{r})\in\Lambda_{m}$ be a Young diagram with $r$ rows, and here $\lambda_{1}$ is the number of boxes appearing in the first row. Then, for any $n\geq\lambda_{1}$, we define the two quantities
\begin{align*}
\mathbf{k}_{m+n}^{+}(\lambda^{\langle n\rangle})&:=|\{x\in\mathrm{S}_{m+n} \ | \ \mathtt{sh}(x)=\lambda^{\langle n\rangle} \text{ and } \mathbf{K}_{m+n}(x)=\mathtt{true}\}|, \\
\mathbf{k}_{m+n}^{-}(\lambda^{\langle n\rangle})&:=|\{x\in\mathrm{S}_{m+n} \ | \ \mathtt{sh}(x)=\lambda^{\langle n\rangle} \text{ and } \mathbf{K}_{m+n}(x)=\mathtt{false}\}|.
\end{align*}
Hence $\mathbf{k}_{m+n}^{+}(\lambda^{\langle n\rangle})$ and $\mathbf{k}_{m+n}^{+}(\lambda^{\langle n\rangle})$ count the number of Kostant positive and negative permutations in $\mathrm{S}_{m+n}$ with shape $\lambda^{\langle n\rangle}$, respectively, recalling that $\lambda^{\langle n\rangle}$ is the Young diagram obtained from $\lambda$ by adding a new first row consisting of $n$ (which is no smaller than $\lambda_{1}$) boxes. The following conjecture was made in \cite[Section 6.9]{MMM24} (although are notation and formulation is different), and we refer to this conjecture as the \emph{Asymptotic Shape Conjecture}:

\begin{conj}\label{Conj:2.13:5}
Let $m\geq0$ and $\lambda\in\Lambda_{m}$. Then the proportion of permutations within $\mathrm{S}_{m+n}$, and of shape $\lambda^{(n)}$, which are Kostant positive is asymptotically equal to 1. In other words, we have that
\[ \frac{\mathbf{k}_{m+n}^{+}(\lambda^{(n)})}{|\mathtt{SYT}_{m+n}(\lambda^{(n)})|^{2}}
=
\frac{\mathbf{k}_{m+n}^{+}(\lambda^{(n)})}{
\mathbf{k}_{m+n}^{+}(\lambda^{(n)})+\mathbf{k}_{m+n}^{-}(\lambda^{(n)})}
\rightarrow 1 \ \text{(as $n\geq\lambda_{1}$ tends towards $\infty$)}. \] 
\end{conj}

This conjecture was made since it was proved for all fully commutative permutations, equivalently, for all one row shapes $\lambda=(m)$. We will prove this conjecture for shapes $\lambda=(1,1)$ and $\lambda=(2,1)$. 

\subsection{K{\aa}hrstr{\"o}m's conjecture}\label{Sec:2.14}

If the Indecomposability conjecture is true, then from \Cref{Thm:2.13:1} we know that $\mathbf{K}_{n}(z)=\mathbf{Kh}_{n}(z)$ for all $z\in\mathrm{S}_{n}$. The following conjecture is \emph{K\aa hrstr\"om's Conjecture}, which was first proposed in \cite[Conjecture 1.2]{KMM23}, and suggests, when dealing with involutions, we can strengthen the previous equality by ungrading and decategorifying $\mathbf{Kh}_{n}$:

\begin{conj}\label{Conj:2.14:1}
For any involution $z\in\mathrm{Inv}_{n}$, the following statements are equivalent:
\begin{itemize}
\item[(1)] The involution $z$ is Kostant positive.
\item[(2)] For any distinct $x,y\leq_{R}^{(n)}z^{-1}$, then $\theta_{x}^{(n)}L_{z}^{(n)}\not\cong\theta_{y}^{(n)}L_{z}^{(n)}$ in $^{\mathbb{Z}}\mathcal{O}_{0}^{(n)}$.
\item[(3)] For any distinct $x,y\leq_{R}^{(n)}z^{-1}$, then $\theta_{x}^{(n)}L_{z}^{(n)}\not\cong\theta_{y}^{(n)}L_{z}^{(n)}$ in $\mathcal{O}_{0}^{(n)}$.
\item[(4)] For any distinct $x,y\leq_{R}^{(n)}z^{-1}$, then $D_{z}^{(n)}C_{x}^{(n)}\neq D_{z}^{(n)}C_{y}^{(n)}$ in $\mathrm{H}_{n}^{\mathbb{A}}$.
\item[(5)] For any distinct $x,y\leq_{R}^{(n)}z^{-1}$, then $D_{z}^{(n)}C_{x}^{(n)}\neq D_{z}^{(n)}C_{y}^{(n)}$ in $\mathrm{H}_{n}^{\mathbb{Z}}$.
\end{itemize}
\end{conj}

For any $z\in\mathrm{S}_{n}$, we let $[\mathbf{Kh}_{n}](z)$ denote the truth value of statement (4) above, and we let $[\mathbf{Kh}_{n}^{\mathbb{Z}}](z)$ denote the truth value of statement (5) above. We call these statements the {\em graded and ungraded 
combinatorial K{\aa}hrstr{\"o}m conditions}, respectively.

This conjecture suggests that Kostant problem has an answer involving only Hecke algebra combinatorics. Certain natural implications that we will employ throughout are
\begin{equation}\label{Eq:2.14:2}
(5) \implies (4), (3), \text{ and } (2).
\end{equation}
So, if $\mathbf{KM}_{m}(\star,z)=\mathtt{true}$ and $[\mathbf{Kh}_{n}^{\mathbb{Z}}](z)=\mathtt{true}$, for $z\in\mathrm{Inv}_{n}$, then, by \Cref{Thm:2.13:1}, $z$ is Kostant positive and satisfies K\aa hrstr\"om's Conjecture. On the other hand, we have the implications
\begin{equation}\label{Eq:2.14:3}
\neg (2) \implies \neg (1), \neg (3), \neg (4), \text{ and } \neg (5).
\end{equation}
Thus, if $\mathbf{Kh}_{n}(z)=\mathtt{false}$ for $z\in\mathrm{Inv}_{n}$, then it is Kostant negative and satisfies K\aa hrstr\"om's Conjecture. Lastly, \cite[Conjecture 37]{CM25-2} suggests that \Cref{Conj:2.14:1} holds for all $z\in\mathrm{S}_{n}$.

\section{Left Kazhdan-Lusztig preorder and consecutive patterns}\label{Sec:3}

In this section we present various results which will be helpful for later sections. In particular, most of these results will allow one to investigate whether we have
\[ D_{z}^{(n)}C_{x^{-1}}^{(n)}\neq0, \hspace{1mm} \text{ equivalently } \hspace{1mm} x\leq_{L}^{(n)}z, \]
by reducing to smaller Hecke algebras/symmetric groups via consecutive patterns. We will often be able to reduce down to a small enough situation which is in reach of GAP3 computations.

\begin{defn}\label{Defn:3:1}
For $2\leq m\leq n$ and $i\in[n-m+1]$, let $F_{m,n}^{i}$ represent the following two morphisms:
\begin{itemize}
\item[(i)] The group monomorphism $F_{m,n}^{i}:\mathrm{S}_{m}\hookrightarrow\mathrm{S}_{n}$ given on the simple transpositions $\mathrm{Sim}_{m}$ by 
\[ s_{a}\mapsto s_{a+i-1}. \] 
\item[(ii)] The $\mathbb{A}$-algebra monomorphism $F_{m,n}^{i}:\mathrm{H}_{m}^{\mathbb{A}}\hookrightarrow\mathrm{H}_{n}^{\mathbb{A}}$ given on the standard basis $\mathcal{T}_{m}$ by
\[ T_{x}^{(m)}\mapsto T_{F_{m,n}^{i}(x)}^{(n)}. \]
\end{itemize}
As (ii) is naturally induced from (i), the overloading of the symbol $F_{m,n}^{i}$ should cause no issue. We will often just write $F_{m,n}^{i}$, where the conditions $2\leq m\leq n$ and $i\in[n-m+1]$ are implied. We will refer to such morphisms simply as \emph{shifts}, and refer to the image $F_{m,n}^{i}(x)$ of any $x\in\mathrm{S}_{m}$ as the \emph{shift of} $x$, or more specifically, as the \emph{shift of} $x$ \emph{to position} $i$.
\end{defn}

\begin{ex}\label{Ex:3:2}
Consider $x=s_{2}s_{1}={\color{teal}312}\in\mathrm{S}_{3}$, then we have $F_{3,7}^{4}(x)=s_{5}s_{4}=123{\color{teal}645}7\in\mathrm{S}_{7}$.
\end{ex}

\begin{ex}\label{Ex:3:2.5}
When $i=1$, then the shift $F_{m,n}^{1}$ corresponds to the natural inclusion $\mathrm{S}_{m}\subset\mathrm{S}_{n}$.
\end{ex}

\begin{lem}\label{Lem:3:3}
For any $x,y\in\mathrm{S}_{m}$, we have the following:
\[ (a) \hspace{1mm} F_{m,n}^{i}(C_{x}^{(m)})=C_{F_{m,n}^{i}(x)}^{(n)} \hspace{1mm} \text{ and } \hspace{1mm} (b) \hspace{1mm} p_{x,y}^{(m)}=p_{F_{m,n}^{i}(x),F_{m,n}^{i}(y)}^{(n)}. \] 
\end{lem}

\begin{proof}
By the definition of how the shift acts on the standard basis, it is easy to confirm that the shift is compatible with the bar involutions $\overline{(-)}:\mathrm{H}_{m}^{\mathbb{A}}\rightarrow\mathrm{H}_{m}^{\mathbb{A}}$ and $\overline{(-)}:\mathrm{H}_{n}^{\mathbb{A}}\rightarrow\mathrm{H}_{n}^{\mathbb{A}}$, that is to say,
\[ F_{m,n}^{i}\circ\overline{(-)}=\overline{(-)}\circ F_{m,n}^{i}. \]
Thus, as $C_{x}^{(m)}$ is invariant under $\overline{(-)}$, then so is $F_{m,n}^{i}(C_{x}^{(m)})$. Lastly, by \Cref{Eq:2.6:1}, we have that
\begin{equation}\label{Eq:3:3-1}
F_{m,n}^{i}(C_{x}^{(m)})=T_{F_{m,n}^{i}(x)}^{(n)}+\sum_{F_{m,n}^{i}(y)<F_{m,n}^{i}(x)}p_{x,y}^{(m)}T_{F_{m,n}^{i}(y)}^{(n)},
\end{equation}
where $p_{x,y}^{(m)}\in v\mathbb{Z}[v]$. Hence, $F_{m,n}^{i}(C_{x}^{(n)})$ also satisfies the unitriangular property from \Cref{Eq:2.6:1}. These properties uniquely define the element $C_{F_{m,n}^{i}(x)}^{(n)}$, and thus $F_{m,n}^{i}(C_{x}^{(m)})=C_{F_{m,n}^{i}(x)}^{(n)}$, giving (a). Then (b) follows from the definition of the Kazhdan-Lusztig polynomials and \Cref{Eq:3:3-1}.
\end{proof}

\begin{rmk}\label{Rmk:3:4}
For $x\in\mathrm{S}_{m}$, then, in general, $F_{m,n}^{i}(D_{x}^{(m)})\neq D_{F_{m,n}^{i}(x)}^{(n)}$. In fact, we have that
\[ F_{m,n}^{i}(D_{x}^{(m)})=D_{F_{m,n}^{i}(x)w_{0}'w_{0,n}}^{(n)}(T_{w_{0,n}}^{(n)})^{-1}T_{w_{0}'}^{(n)}, \]
where $w_{0}'$ is the longest element in the parabolic subgroup $\mathrm{Im}(F_{m,n}^{i})\subset\mathrm{S}_{n}$. We will not need this, but to see why it is true one can consult \cite[Equation 5.1.8]{Lu84}, see also \cite[Equation 4]{CM25-2}, and note that the shifts commute with the ring involution $\bm{\beta}$ given in the latter reference.
\end{rmk}

\begin{lem}\label{Lem:3:5}
Let $F_{m,n}^{i}:\mathrm{S}_{m}\hookrightarrow\mathrm{S}_{n}$ be a shift. Consider the subset
\[ \mathrm{X}_{m,n}^{i}:=\{x\in\mathrm{S}_{n} \ | \ x \text{ consecutively contains } 1_{m}\in\mathrm{S}_{m} \text{ at position } i\}\subset\mathrm{S}_{n}. \]
Then $\mathrm{X}_{m,n}^{i}$ is the set of left coset representatives of $\mathrm{S}_{n}/\mathrm{Im}(F_{m,n}^{i})$ having minimal length. Moreover, consider any $z\in\mathrm{S}_{n}$ and suppose that $p\in\mathrm{S}_{m}$ is the pattern of size $m$ consecutively appearing in $z$ at position $i$. Then we have the decomposition
\[ z=xF_{m,n}^{i}(p) \]
for some unique $x\in\mathrm{X}_{m,n}^{i}$, and where $\ell(xF_{m,n}^{i}(p))=\ell(x)+\ell(F_{m,n}^{i}(p))$.
\end{lem}

\begin{proof}
Firstly, let $z$ and $p$ be as given in the statement of the lemma. Then in one-line notation,
\[ z=z(1)\cdots z(i-1){\color{teal}z(i)z(i+1)\cdots z(i+m-2)z(i+m-1)}z(i+m)\cdots z(n) \]
where the letters appearing in the {\color{teal}teal} subword share the same relative order as those of the pattern $p$, i.e. for any $a,b\in[m]$ we have that $p(a)<p(b)$ if and only if $z(i+a-1)<z(i+b-1)$. Multiplying $z$ on the right by a permutation corresponds, in one-line notation, to shuffling the positions of letters of $z$ according to such a permutation. As the {\color{teal}teal} subword realizes the consecutive containment of the pattern $p$ in $z$ at position $i$, one can deduce that multiplying $z$ on the right by $F_{m,n}^{i}(p^{-1})=F_{m,n}^{i}(p)^{-1}$ amounts to shuffling the letters in this {\color{teal}teal} subword so they are in increasing order. This means that $zF_{m,n}^{i}(p)^{-1}$ consecutively contains the pattern $1_{m}$ at position $i$, and hence
\begin{equation}\label{Eq:N3:3-1}
zF_{m,n}^{i}(p)^{-1}=x,
\end{equation}
for some $x\in\mathrm{X}_{m,n}^{i}$. Moreover, this $x$ is uniquely determined since it must agree with $z$ on the non-teal letters, that is $x(k)=z(k)$ for all $k\in[1,i-1]\cup[i+m,n]$, and, clearly, only one such element of $\mathrm{X}_{m,n}^{i}$ can satisfy this. Solving \Cref{Eq:N3:3-1} in terms of $z$ gives the desired decomposition claimed in the statement of the lemma. Furthermore, since both $z$ and $p$ were arbitrary, we have also shown that $\mathrm{X}_{m,n}^{i}$ is indeed the set of left coset representatives of $\mathrm{S}_{n}/\mathrm{Im}(F_{m,n}^{i})$. 

Lastly, let $x\in\mathrm{X}_{m,n}^{i}$ and $y\in\mathrm{Im}(F_{m,n}^{i})$. Recall that, for  $s_{k}\in\mathrm{Sim}_{n}$, we have $\ell(xs_{k})=\ell(x)+\ell(s_{k})$ if and only if $x$ consecutively contains the pattern $12\in\mathrm{S}_{2}$ at position $k$ (this is essentially \Cref{Lem:2.2:3}). Therefore, since $x$ consecutively contains $1_{m}$ at position $i$, one can confirm via induction on the length of $y$ that we have $\ell(xy)=\ell(x)+\ell(y)$. This establishes the property that the decomposition of $z$ discussed above respects the length function, and that $x$ is of minimal length in its coset.  
\end{proof}

We now recall a known result regarding the left Kazhdan-Lusztig preorder, but rewritten in terms of consecutive patterns in light of the above lemma.

\begin{thm}\label{Thm:3:6}
Let $m\leq n$, $p_{1},p_{2}\in\mathrm{S}_{m}$, and $z_{1},z_{2}\in\mathrm{S}_{n}$. Suppose, for some $i\in[n-m+1]$, that $z_{j}$ consecutively contains $p_{j}$ at position $i$, for $j\in\{1,2\}$. Then we have the following implication:
\begin{equation}\label{Eq:3:6-1}
z_{1}\leq_{L}^{(n)}z_{2} \implies p_{1}\leq_{L}^{(m)}p_{2}.
\end{equation}
\end{thm}

\begin{proof}
By \Cref{Lem:3:5} we know that $z_{1}=x_{1}F_{m,n}^{i}(p_{1})$ and $z_{2}=x_{2}F_{m,n}^{i}(p_{2})$, for some $x_{1},x_{2}\in X_{m,n}^{i}$. Therefore, we need to establish the implication
\begin{equation*}
x_{1}F_{m,n}^{i}(p_{1})\leq_{L}^{(n)}x_{2}F_{m,n}^{i}(p_{2}) \implies p_{1}\leq_{L}^{(m)}p_{2}.
\end{equation*}
By Equation $(\dagger)$ in \cite[Section 4]{Ge03}, we have the implication
\[ x_{1}F_{m,n}^{i}(p_{1})\leq_{L}^{(n)}x_{2}F_{m,n}^{i}(p_{2}) \implies F_{m,n}^{i}(p_{1})\leq_{L}'F_{m,n}^{i}(p_{2}), \]
where $\leq_{L}'$ denotes the corresponding left Kazhdan-Lusztig preorder for the parabolic subgroup $\mathrm{Im}(F_{m,n}^{i})$. By definition of $F_{m,n}^{i}$, we have an isomorphism $\mathrm{Im}(F_{m,n}^{i})\xrightarrow{\sim}\mathrm{S}_{m}$ given by $F_{m,n}^{i}(p)\mapsto p$. Hence $F_{m,n}^{i}(p_{1})\leq_{L}'F_{m,n}^{i}(p_{2})$ is equivalent to $p_{1}\leq_{L}^{(m)}p_{2}$.
\end{proof}

In general, the implication in Equation~\eqref{Eq:3:6-1} cannot be reversed, but we now demonstrate, in the following proposition, certain specialisations which do allow one to promote this implication into an equivalence.

\begin{prop}\label{Prop:3:7}
For $m\leq n$, let $x,p\in\mathrm{S}_{m}$ and $z\in\mathrm{S}_{n}$, where $z$ consecutively contains $p$ as a pattern at position $i\in[n-m+1]$. Then we have the following two equivalences:
\begin{itemize}
\item[(a)] $x\leq_{L}^{(m)}p \iff F_{m,n}^{i}(x)\leq_{L}^{(n)}F_{m,n}^{i}(p)$.
\item[(b)] $x\leq_{L}^{(m)}p \iff F_{m,n}^{i}(x)\leq_{L}^{(n)}z$.
\end{itemize}
\end{prop}

\begin{proof}
Claim~(a): Firstly, for any $w\in\mathrm{S}_{m}$, we know that $F_{m,n}^{i}(w)$ consecutively contains $w$ at position $i$. Thus, the backwards implication of (a) holds from \Cref{Thm:3:6}. For the forward implication, since $x\leq_{L}^{(m)}p$, then, by definition, there exists $h\in\mathrm{H}_{m}^{\mathbb{A}}$ such that
\[ [C_{p}^{(m)}](hC_{x}^{(m)})\neq0. \]
Then, by \Cref{Lem:3:3}, acting on $hC_{x}$ by the monomorphism $F_{m,n}^{i}:\mathrm{H}_{m}^{\mathbb{A}}\rightarrow\mathrm{H}_{n}^{\mathbb{A}}$, we see that
\[ [C_{F_{m,n}^{i}(p)}^{(n)}](F_{m,n}^{i}(h)C_{F_{m,n}^{i}(x)}^{(n)})\neq0. \]
By definition, this implies that $F_{m,n}^{i}(x)\leq_{L}^{(n)}F_{m,n}^{i}(y)$, as desired. 

Claim~(b): The reverse implication $F_{m,n}^{i}(x)\leq_{L}^{(n)}z \implies x\leq_{L}^{(m)}p$ holds by \Cref{Thm:3:6}. Also, from Claim~(a), to prove the forward implication it suffices to prove the implication
\begin{equation}\label{Eq:3:7-1}
F_{m,n}^{i}(x)\leq_{L}F_{m,n}^{i}(p) \implies F_{m,n}^{i}(x)\leq_{L}z.
\end{equation}
Moreover, by \Cref{Lem:3:5}, we know that $z=yF_{m,n}^{i}(p)$, where $y$ belongs to the set $X_{m,n}^{i}$ of left coset representatives of $\mathrm{S}_{m}/\mathrm{Im}(F_{m,n}^{i})$ of minimal length. In particular, we have that
\begin{equation}\label{Eq:3:7-2}
\ell(yF_{m,n}^{i}(p))=\ell(y)+\ell(F_{m,n}^{i}(p)).
\end{equation}
Therefore, the implication in Equation~\eqref{Eq:3:7-1} is equivalent to the following:
\begin{equation}\label{Eq:3:7-3}
F_{m,n}^{i}(x)\leq_{L}F_{m,n}^{i}(p) \implies F_{m,n}^{i}(x)\leq_{L}yF_{m,n}^{i}(p).
\end{equation}
To prove this implication, it suffices to show that $F_{m,n}^{i}(p)\leq_{L}^{(n)}yF_{m,n}^{i}(p)$. By \Cref{Eq:2.5:1}, and since the Kazhdan-Lusztig basis $\mathcal{C}_{n}$ has a unitriangular relationship with the standard basis $\mathcal{T}_{n}$ as described in \Cref{Eq:2.6:1}, one can deduce that we have the equation 
\[ C_{y}^{(n)}C_{F_{m,n}^{i}(p)}^{(n)}=C_{yF_{m,n}^{i}(p)}^{(n)}+\sum_{w<yF_{m,n}^{i}(p)}\gamma_{y,F_{m,n}^{i}(p),w}^{(n)}C_{w}^{(n)}. \]
Therefore $[C_{yF_{m,n}^{i}(p)}^{(n)}](C_{y}^{(n)}C_{F_{m,n}^{i}(p)}^{(n)})=1\neq0$, and thus $F_{m,n}^{i}(p)\leq_{L}^{(n)}yF_{m,n}^{i}(p)$, as desired.
\end{proof}

\begin{rmk}\label{Rmk:3:8}
\Cref{Eq:2.9:7} says that $s_{i}\leq_{L}^{(n)}z$ if and only if $s_{i}\in\mathrm{Des}_{R}(z)$, equivalently, that $F_{2,n}^{i}(s_{1})\leq_{L}^{(n)}z$ if and only if the pattern $p\in\mathrm{S}_{2}$ consecutively appearing in $z$ at position $i$ is $p=21$. This is precisely Claim~(b) of \Cref{Prop:3:7} in the special case of $x=s_{1}$ and $m=2$. Thus Claim~(b) of \Cref{Prop:3:7} can be viewed as a generalisation of \Cref{Eq:2.9:7} where we replace $s_{1}\in\mathrm{S}_{2}$ with any permutation $x\in\mathrm{S}_{m}$.
\end{rmk}

\begin{ex}\label{Ex:3:9}
Consider the involution $x=s_{2}s_{1}s_{3}s_{2}=3412\in\mathrm{Inv}_{4}$. We first describe the left cell structure of $\mathrm{S}_{4}$, with a focus given to the involution $x$, by the following figure:
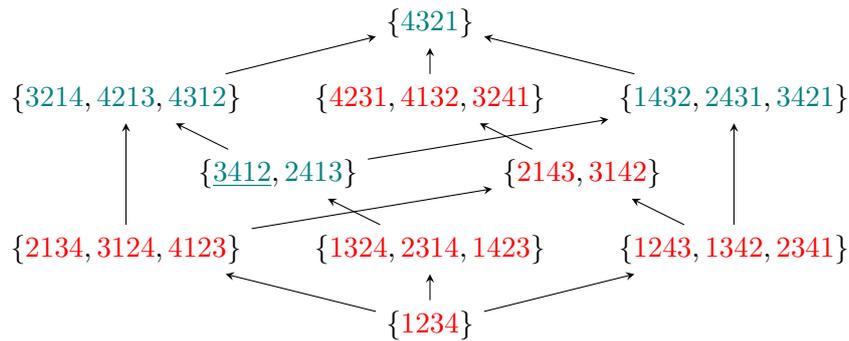
\begin{figure}[H]
\begin{tikzpicture}
\node (4321) at (0,4) {$\{{\color{teal}4321}\}$};

\node (3214) at (-4,3) {$\{{\color{teal}3214},{\color{teal}4213},{\color{teal}4312}\}$};
\node (4231) at (0,3) {$\{{\color{red}4231},{\color{red}4132},{\color{red}3241}\}$};
\node (1432) at (4,3) {$\{{\color{teal}1432},{\color{teal}2431},{\color{teal}3421}\}$};

\node (3412) at (-2,2) {$\{{\color{teal}\underline{3412}},{\color{teal}2413}\}$};
\node (2143) at (2,2) {$\{{\color{red}2143},{\color{red}3142}\}$};

\node (2134) at (-4,1) {$\{{\color{red}2134},{\color{red}3124},{\color{red}4123}\}$};
\node (1324) at (0,1) {$\{{\color{red}1324},{\color{red}2314},{\color{red}1423}\}$};
\node (1243) at (4,1) {$\{{\color{red}1243},{\color{red}1342},{\color{red}2341}\}$};

\node (1234) at (0,0) {$\{{\color{red}1234}\}$};

\draw [-stealth](1432) -- (4321);
\draw [-stealth](4231) -- (4321);
\draw [-stealth](3214) -- (4321);

\draw [-stealth](2143) -- (4231);
\draw [-stealth](3412) -- (1432);
\draw [-stealth](3412) -- (3214);

\draw [-stealth](1243) -- (1432);
\draw [-stealth](1243) -- (2143);
\draw [-stealth](1324) -- (3412);
\draw [-stealth](2134) -- (3214);
\draw [-stealth](2134) -- (2143);

\draw [-stealth](1234) -- (1243);
\draw [-stealth](1234) -- (1324);
\draw [-stealth](1234) -- (2134);
\end{tikzpicture}
\caption{Vertices are all the left cells of $\mathrm{S}_{4}$, with permutations written in one-line notation. An arrow $P\rightarrow Q$ implies $p<_{L}^{(4)}q$ for all $p\in P$ and $q\in Q$. The involution $x=s_{2}s_{1}s_{3}s_{2}$ is underlined. A pattern $p$ is {\color{teal}teal} if $x\leq_{L}^{(4)}p$, and {\color{red}red} if $x\not\leq_{L}^{(4)}p$.}
\label{Fig:3:9-1}
\end{figure}
Now consider any $n\geq 4$, $i\in[n-3]$, and $z\in\mathrm{S}_{n}$, and let $p\in\mathrm{S}_{4}$ be the pattern of size $4$ which consecutively appears in $z$ at position $i$. Then, from Claim~(b) of \Cref{Prop:3:7}, and \Cref{Fig:3:9-1}, we have
\[ F_{4,n}^{i}(x)=s_{i+1}s_{i}s_{i+2}s_{i+1}\leq_{L}^{(n)}z \iff x\leq_{L}^{(4)}p \iff p \text{ is a {\color{teal}teal} pattern in \Cref{Fig:3:9-1}}. \]
Hence, by extracting the consecutive pattern of $z$ at position $i$, understanding whether $F_{4,n}^{i}(x)\leq_{L}^{(n)}z$ or $F_{4,n}^{i}(x)\not\leq_{L}^{(n)}z$ holds in $\mathrm{S}_{n}$, is reduced down to understanding whether $x\leq_{L}^{(4)}p$ or $x\not\leq_{L}^{(4)}p$ holds in $\mathrm{S}_{4}$. For explicit examples, let $n=9$ and $i=4$, then we have that
\[ F_{4,9}^{4}(x)=s_{5}s_{4}s_{6}s_{5}\leq_{L}^{(9)}269731854 \ \text{ and } \ F_{4,9}^{4}(x)=s_{5}s_{4}s_{6}s_{5}\not\leq_{L}^{(9)}216395748. \]
The former relation holds since $269731854$ consecutively contains the patterns $3214$ at position $4$ (which is a {\color{teal}teal} pattern in \Cref{Fig:3:9-1}), while the latter non-relation holds since $216395748$ consecutively contains the pattern $1423$ at position $4$ (which is a {\color{red}red} pattern in \Cref{Fig:3:9-1}).
\end{ex}

With Claim~(b) of \Cref{Prop:3:7} in hand, we now present a mild generalisation and consequence in the form of two corollaries, given below. We first need to set up and generalise some notation. For any $k\geq0$ and tuple $\underline{m}:=(m_{1},\dots,m_{k})\in\mathbb{Z}_{\geq 2}^{\times k}$, we will write
\[ |\underline{m}|:=\sum_{1\leq j\leq k}m_{j} \hspace{2mm} \text{ and } \hspace{2mm} \mathrm{S}_{\underline{m}}:=\mathrm{S}_{m_{1}}\times\cdots\times\mathrm{S}_{m_{k}}. \] 
Also, given any $\underline{x}=(x_{1},\dots,x_{k}), \underline{y}=(y_{1},\dots,y_{k})\in\mathrm{S}_{\underline{m}}$, then we will write $\underline{x}\leq_{L}^{(\underline{m})}\underline{y}$ to mean
\[ x_{j}\leq_{L}^{(m_{j})}y_{j} \hspace{1mm} \text{ for all } \hspace{1mm} j\in[k]. \]
In \Cref{Sec:2.2} we described what it means for a pattern $p\in\mathrm{S}_{m}$ to consecutively appear within a permutation $z\in\mathrm{S}_{n}$ at position $i$, where $m\leq n$ and $i\in[n-m+1]$. We now slightly generalise this notion by replacing $m$, $p$ and $i$, with tuple-counterparts $\underline{m}$, $\underline{p}$, and $\underline{i}$.

\begin{defn}\label{Defn:3:10}
Let $n,k\geq 0$ and $\underline{m}=(m_{1},\dots,m_{k})\in\mathbb{Z}_{\geq 2}^{\times k}$ be such that $|\underline{m}|\leq n$. Furthermore, let $\underline{i}:=(i_{1},\dots,i_{k})\in[n-m_{1}+1]\times\cdots\times[n-m_{k}+1]$ be a tuple satisfying the property that
\begin{equation}\label{Eq:3:10-1}
i_{j+1}>i_{j}+m_{j}-1 \hspace{1mm} \text{ for all } \hspace{1mm} j\in[k-1].
\end{equation}
Then, for $\underline{p}=(p_{1},\dots,p_{k})\in\mathrm{S}_{\underline{m}}$ and $z\in\mathrm{S}_{n}$, we say that $z$ \emph{consecutively contains the pattern} $\underline{p}$ at position $\underline{i}$ if and only if $z$ consecutively contains the pattern $p_{j}$ at position $i_{j}$, for each $j\in[k]$.
\end{defn}

\begin{ex}\label{Ex:3:11}
Let $n=9$, and consider $z:=s_{2}s_{6}s_{5}s_{8}s_{7}s_{6}s_{7}s_{8}=132479685\in\mathrm{S}_{9}$. Let $k=2$,
\[ \underline{m}=(3,4)\in\mathbb{Z}_{\geq 2}^{\times 2}, \hspace{1mm} \text{ and } \hspace{1mm} \underline{i}=(2,6)\in[9-3+1]\times[9-4+1]=[7]\times[6]. \]
Note that \Cref{Eq:3:10-1} holds. Then we see that $z$ consecutively contains the pattern 
$$\underline{p}=(p_{1},p_{2}):=(s_{1},s_{2}s_{1}s_{3})=(213,4231)$$ 
at position $\underline{i}=(2,6)$, that is, the pattern $p_{1}=213$ consecutively appears in $z$ at position $2$, and the pattern $p_{2}=4231$ consecutively appears in $z$ at position $6$.
\end{ex}

\begin{rmk}\label{Rmk:3:12}
Assume the setup from \Cref{Defn:3:10}. Then, for any $j\in[k]$, the pattern $p_{j}$ is of size $m_{j}$ \emph{(}it belongs to $\mathrm{S}_{m_{j}}$\emph{)}. Hence, if $p_{j}$ consecutively appears in $z$ at position $i_{j}$, this means that the letters of $z$ in positions $i_{j},i_{j}+1,\dots,i_{j}+m_{j}-1$ realise the pattern $p_{j}$. From this, we can note two things about the definition of $\underline{i}$. The first is that the coordinate $i_{j}$ needs to belong to the set $[n-m_{j}+1]$ for it to make sense to consider a pattern of size $m_{j}$ at position $i_{j}$. The second is that the condition given by \Cref{Eq:3:10-1} is saying that the positions in which each pattern $p_{j}$ consecutive appear in $z$ are mutually disjoint, i.e. the supports of the patterns in $z$ do not overlap.
\end{rmk}

\begin{defn}\label{Defn:3:13}
Let $n$, $k$, $\underline{m}$, and $\underline{i}$, be as given in \Cref{Defn:3:10}. Then we define the \emph{multi-shift map} $F_{\underline{m},n}^{\underline{i}}:\mathrm{S}_{\underline{m}}\hookrightarrow\mathrm{S}_{n}$ as the group monomorphism given by
\[ F_{\underline{m},n}^{\underline{i}}(\underline{x})=\prod_{1\leq j\leq k}F_{m_{j},n}^{i_{j}}(x_{j}), \hspace{1mm} \text{ for any } \hspace{1mm} \underline{x}=(x_{1},\dots,x_{k})\in\mathrm{S}_{\underline{m}}. \]
By \Cref{Rmk:3:12}, the images of $F_{m_{j},n}^{i_{j}}$, for each $j\in[k]$, pairwise commute, hence $F_{\underline{m},n}^{\underline{i}}$ is well-defined.
\end{defn}

\begin{ex}\label{Ex:3:14}
Let $n=9$, $\underline{m}=(3,4)$, and $\underline{i}=(2,6)$, be as given in \Cref{Ex:3:11}, and consider
\[ \underline{x}:=(s_{2}s_{1},s_{2}s_{1}s_{3})=(312,3142)\in\mathrm{S}_{3}\times\mathrm{S}_{4}. \] 
Then the multi-shift is $F_{\underline{m},n}^{\underline{i}}(\underline{x})=F_{(3,4),9}^{(2,6)}({\color{blue}s_{2}s_{1}},{\color{purple}s_{2}s_{1}s_{3}})={\color{blue}s_{3}s_{2}}{\color{purple}s_{7}s_{6}s_{8}}=1{\color{blue}423}5{\color{purple}8697}\in\mathrm{S}_{9}$.
\end{ex}

When we write a multi-shift map $F_{\underline{m},n}^{\underline{i}}$, it is implied that $n$, $k$, $\underline{m}$, and $\underline{i}$, satisfy all the conditions described in \Cref{Defn:3:10}. We can think of $F_{\underline{m},n}^{\underline{i}}$ as the ``disjoint product'' of the shift maps $F_{m_{j},n}^{i_{j}}$ for all $j\in[k]$. More precisely, restricting this multi-shift map to the $j$-th component $S_{m_{j}}$ of $\mathrm{S}_{\underline{m}}$ recovers the usual shift map $F_{m_{j},n}^{i_{j}}$, and from \Cref{Rmk:3:12}, \Cref{Eq:3:10-1} is the condition that the images of all these restrictions of $F_{\underline{m},n}^{\underline{i}}$ pairwise commute, i.e. for any distinct $a,b\in[k]$,
\[ xy=yx \hspace{1mm} \text{ for all } \hspace{1mm} x\in\mathrm{Im}(F_{m_{a},n}^{i_{a}})\cong\mathrm{S}_{m_{a}} \hspace{1mm} \text{ and } \hspace{1mm} y\in\mathrm{Im}(F_{m_{b},n}^{i_{b}})\cong\mathrm{S}_{m_{b}}. \]
With this in mind, we give the following lemma:

\begin{lem}\label{Lem:2.9:6.5}
For any $x,y,z\in\mathrm{S}_{n}$ such that $xy=yx$ and $\mathrm{Sup}(x)\cap\mathrm{Sup}(y)=\emptyset$, we have
\[ xy\leq_{L}^{(n)}z \iff x\leq_{L}^{(n)}z \text{ and } y\leq_{L}^{(n)}z. \]
\end{lem}

\begin{proof}
For the forward implication, the relation $xy\leq_{L}^{(n)}z$ means there exists $h\in\mathrm{H}_{n}^{\mathbb{A}}$ such that
\[ [C_{z}^{(n)}](hC_{xy}^{(n)})\neq0. \]
Since $\mathrm{Sup}(x)\cap\mathrm{Sup}(y)=\emptyset$, then by \Cref{Lem:2.6:2}, $C_{xy}^{(n)}=C_{x}^{(n)}C_{y}^{(n)}$, and, since $xy=yx$, we also have that $C_{xy}^{(n)}=C_{yx}^{(n)}=C_{y}^{(n)}C_{x}^{(n)}$. Thus, setting both $h_{x}:=hC_{x}^{(n)}$ and $h_{y}:=hC_{y}^{(n)}$, we see that
\[ hC_{xy}^{(n)}=h_{y}C_{x}^{(n)}=h_{x}C_{y}^{(n)}. \]
Hence $[C_{z}^{(n)}](hC_{xy}^{(n)})\neq0$ implies $[C_{z}^{(n)}](h_{y}C_{x}^{(n)})\neq0$ and $[C_{z}^{(n)}](h_{x}C_{y}^{(n)})\neq0$, and by definition, these two conditions imply that $x\leq_{L}^{(n)}z$ and $y\leq_{L}^{(n)}z$, respectively. 

Now, for the reverse implication, we seek to prove that
\begin{equation}\label{Eq:2.9:6.5-1}
x\leq_{L}^{(n)}z \hspace{1mm} \text{ and } \hspace{1mm} y\leq_{L}^{(n)}z \implies xy\leq_{L}^{(n)}z.
\end{equation}
Let $x'$ and $y'$ be the unqiue involutions such that $x\sim_{L}^{(n)}x'$ and $y\sim_{L}^{(n)}y'$. Note, we must have that $\mathrm{Sup}(x')\subset\mathrm{Sup}(x)$. This is because $x'$ corresponds to the unqiue involution in the parabolic subgroup $P:=\langle\mathrm{Sup}(x)\rangle\subset\mathrm{S}_{n}$ which belongs to the same left cell as $x$ in $P$, i.e. $x\sim_{L}^{P}x'$ where $\sim_{L}^{P}$ is the $P$ counterpart to $\sim_{L}^{(n)}$ defined in the obvious manner, and since $x\sim_{L}^{P}x'\implies x\sim_{L}^{(n)}x'$, which follows from definitions. Similarly, $\mathrm{Sup}(y')\subset\mathrm{Sup}(y)$, and hence $x'y'$ is also an involution. Moreover, from \Cref{Lem:2.6:2} and \Cref{Eq:2.7:0}, one can deduce that $x'y'\sim_{L}^{(n)}xy$. Thus, to prove the implication to Equation~\eqref{Eq:2.9:6.5-1}, it suffices to assume that $x$ and $y$ are involutions, and so $xy$ is also an involution. Under this assumption, by Claim~(1) of \Cref{Lem:2.11:4} and Equations \eqref{Eq:2.9:5} and \eqref{Eq:2.9:2}, we have 
\begin{equation}\label{Eq:2.9:6.5-2}
[D_{z}^{(n)}](D_{z}^{(n)}C_{x}^{(n)})\neq0 \hspace{1mm} \text{ and } \hspace{1mm} [D_{z}^{(n)}](D_{z}^{(n)}C_{y}^{(n)})\neq0.
\end{equation}
By Equations \eqref{Eq:2.8:2} and \eqref{Eq:2.7:0}, the coefficients $[D_{a}^{(n)}](D_{b}^{(n)}C_{c}^{(n)})$ belong to $\mathbb{A}_{\geq0}$ (i.e. Laurent polynomials with non-negative coeeficients). This implies, together with \Cref{Eq:2.9:6.5-2} and \Cref{Lem:2.6:2}, that
\[ [D_{z}^{(n)}](D_{z}^{(n)}C_{xy}^{(n)})=[D_{z}^{(n)}](D_{z}^{(n)}C_{x}^{(n)}C_{y}^{(n)})\neq0, \]
and again by (1) of \Cref{Lem:2.11:4} and Equations~\eqref{Eq:2.9:5} and \eqref{Eq:2.9:2}, this is equivalent to $xy\leq_{L}^{(n)}z$. 
\end{proof}

We are now in a position to generalise Claim~(b) of \Cref{Prop:3:7} via the following corollary.

\begin{cor}\label{Cor:3:15}
Let $n$, $k$, $\underline{m}$, and $\underline{i}$, be as given in \Cref{Defn:3:10}. Let $\underline{p},\underline{x}\in\mathrm{S}_{\underline{m}}$ and $z\in\mathrm{S}_{n}$, and suppose that $z$ consecutively contains $\underline{p}$ at position $\underline{i}$. Then we have the following equivalence:
\begin{equation}\label{Eq:3:15-1}
\underline{x}\leq_{L}^{(\underline{m})}\underline{p} \iff F_{\underline{m},n}^{\underline{i}}(\underline{x})\leq_{L}^{(n)}z.
\end{equation}
\end{cor}

\begin{proof}
Starting with the right hand side of (\ref{Eq:3:15-1}), we have the following equivalences:
\begin{equation}\label{Eq:3:15-2}
F_{\underline{m},n}^{\underline{i}}(\underline{x})\leq_{L}^{(n)}z \iff \prod_{1\leq j\leq k}F_{m_{j},n}^{i_{j}}(x_{j})\leq_{L}^{(n)}z \iff F_{m_{j},n}^{i_{j}}(x_{j})\leq_{L}^{(n)}z, \hspace{2mm} \forall j\in[k].
\end{equation}
The first equivalence holds by the definition of a multi-shift map, while the second equivalence holds from \Cref{Lem:2.9:6.5}, noting that all the shifts $F_{m_{j},n}^{i_{j}}(x_{j})$, for each $j\in[k]$, pairwise commute and have disjoint supports, as discussed in \Cref{Rmk:3:12}. Now, by Claim~(b) of \Cref{Prop:3:7}, the condition on the right hand side of the series of equivalences in 
Equation~\eqref{Eq:3:15-2} is equivalent to
\[ x_{j}\leq_{L}^{(m_{j})}p_{j}, \hspace{1mm} \text{ for all } \hspace{1mm} j\in[k], \] 
since $z$ consecutively contains $p_{j}$ at position $i_{j}$. By definition this is equivalent to $\underline{x}\leq_{L}^{(\underline{m})}\underline{p}$. 
\end{proof}

\begin{ex}\label{Ex:3:16}
Continuing from both \Cref{Ex:3:11} and \Cref{Ex:3:14}, one can check that 
\[ \underline{x}=({\color{blue}312},{\color{purple}3142})\leq_{L}^{(3,4)}({\color{cyan}213},{\color{red}4231})=\underline{p}, \]
since both ${\color{blue}312}\leq_{L}^{(3)}{\color{cyan}213}$ and ${\color{purple}3142}\leq_{L}^{(4)}{\color{red}4231}$. Therefore, by \Cref{Cor:3:15}, and since $z=1{\color{cyan}324}7{\color{red}9685}$ consecutively contains $\underline{p}=({\color{cyan}213},{\color{red}4231})$ at position $\underline{i}=(2,6)$, we have that
\[ F_{(3,4),9}^{(2,6)}(\underline{x})=1{\color{blue}423}5{\color{purple}8697}\leq_{L}^{(9)}1{\color{cyan}324}7{\color{red}9685}=z. \]
\end{ex}

\begin{cor}\label{Cor:3:17}
To understand all $\leq_{L}^{(n)}$-relations, it suffices to understand all $\leq_{L}^{(m)}$-relations with $m\leq n$, as well as all the relations of the form $x\leq_{L}^{(m)}p$ where $x,p\in\mathrm{Inv}_{m}$ and $x$ has \emph{full support}, i.e. $\mathrm{Sup}(x)=\mathrm{Sim}_{m}$.
\end{cor}

\begin{proof}
For $y,z\in\mathrm{S}_{n}$, we need to understand the relation $y\leq_{L}^{(n)}z$.
Since there is a unique involution in each left cell, we can assume that both 
$y$ and $z$ are involutions. If $y$ has full support, the result is immediate.
If $y$ does not have full support, we need to show that the relation 
$y\leq_{L}^{(n)}z$ holds if and only if a finite collection of relations 
$x_{j}\leq_{L}^{(m_{j})}p_{j}$ hold of the form described in the statement of the corollary. We have that
\[ \mathrm{Sup}(y)=\{s_{i_{1}},s_{i_{1}+1},\dots,s_{i_{1}+m_{1}-2}\}\sqcup\cdots\sqcup\{s_{i_{k}},s_{i_{k}+1},\dots,s_{i_{k}+m_{k}-2}\}, \]
for some $k\geq 1$, $\underline{m}=(m_{1},\dots,m_{k})\in\mathbb{Z}_{\geq 2}^{\times k}$, and $\underline{i}=(i_{1},\dots,i_{k})\in[n-m_{1}+1]\times\cdots\times[n-m_{k}+1]$, where $k$ is maximal. In particular, the fact that $k$ is maximal means that the tuples $\underline{m}$ and $\underline{i}$ satisfy \Cref{Eq:3:10-1}. Therefore, $y$ must decompose into the product $y=y_{1}y_{2}\cdots y_{k}$ where the elements $y_{j}$ are pairwise commuting involutions whose supports are given by
\[ \mathrm{Sup}(y_{j})=\{s_{i_{j}},s_{i_{j}+1},\dots,s_{i_{j}+m_{j}-2}\}, \text{ for each } j\in[k]. \]
From this we see that there exists $x_{j}\in\mathrm{Inv}_{m_{j}}$ with full support such that $F_{m_{j},n}^{i_{j}}(x_{j})=y_{j}$, and thus
\[ y=F_{m_{1},n}^{i_{1}}(x_{1})F_{m_{2},n}^{j_{2}}(x_{2})\cdots F_{m_{k},n}^{j_{k}}(x_{k}). \]
Therefore, the relation $y\leq_{L}^{(n)}z$, which we started with, is equivalent to
\[ F_{m_{1},n}^{j_{1}}(x_{1})F_{m_{2},n}^{j_{2}}(x_{2})\cdots F_{m_{k},n}^{j_{k}}(x_{k})\leq_{L}^{(n)}z. \]
By \Cref{Cor:3:15}, this is equivalent to the relations $x_{j}\leq_{L}^{(m_{j})}q_{j}$ holding in $\mathrm{S}_{m_{j}}$ for each $j\in[k]$, where $q_{j}$ is the pattern consecutively appearing in $z$ at position $i_{j}$. Letting $p_{j}$ be the unique involution such that $p_{j}\sim_{L}q_{j}$, for each $j\in[k]$, then the prementioned relations are equivalent to $x_{j}\leq_{L}^{(m_{j})}p_{j}$ for each $j\in[k]$, which are of the desired formed given in the statement of this corollary. 
\end{proof}

We close out this section with the following two lemmata:

\begin{lem}\label{Lem:3:18}
For any $m\leq n$, $x,y,z\in\mathrm{S}_{m}$ such that $x,y\leq_{R}^{(m)}z^{-1}$, and shift $F_{m,n}^{i}$, we have that
\[ D_{\mathit{F}_{m,n}^{i}(z)}^{(n)}C_{\mathit{F}_{m,n}^{i}(x)}^{(n)}=D_{\mathit{F}_{m,n}^{i}(z)}^{(n)}C_{\mathit{F}_{m,n}^{i}(y)}^{(n)}\neq0 \implies D_{z}^{(m)}C_{x}^{(m)}=D_{z}^{(m)}C_{y}^{(m)}\neq0. \]
\end{lem}

\begin{proof}
By assumption we have that 
\[ D_{\mathit{F}_{m,n}^{i}(z)}^{(n)}C_{\mathit{F}_{m,n}^{i}(x)}^{(n)}=D_{\mathit{F}_{m,n}^{i}(z)}^{(n)}C_{\mathit{F}_{m,n}^{i}(y)}^{(n)}\neq0. \]
This is equivalent to the equalities of coefficients
\[ [D_{b}^{(n)}](D_{\mathit{F}_{m,n}^{i}(z)}^{(n)}C_{\mathit{F}_{m,n}^{i}(x)}^{(n)})=[D_{b}^{(n)}](D_{\mathit{F}_{m,n}^{i}(z)}^{(n)}C_{\mathit{F}_{m,n}^{i}(y)}^{(n)}), \]
for all $b\in\mathrm{S}_{n}$, with at least one such coefficient non-zero. This implies the equalities of coefficients
\begin{equation*}
[D_{\mathit{F}_{m,n}^{i}(a)}^{(n)}](D_{\mathit{F}_{m,n}^{i}(z)}^{(n)}C_{\mathit{F}_{m,n}^{i}(x)}^{(n)})=[D_{\mathit{F}_{m,n}^{i}(a)}^{(n)}](D_{\mathit{F}_{m,n}^{i}(z)}^{(n)}C_{\mathit{F}_{m,n}^{i}(y)}^{(n)}),
\end{equation*}
for all $a\in\mathrm{S}_{m}$ (i.e. specialising to $b=\mathit{F}_{m,n}^{i}(a)$). Now, by \Cref{Eq:2.8:2}, we have the equalities
\begin{equation}\label{Eq:3:18-1}
[C_{\mathit{F}_{m,n}^{i}(z)}^{(n)}](C_{\mathit{F}_{m,n}^{i}(a)}^{(n)}C_{\mathit{F}_{m,n}^{i}(x^{-1})}^{(n)})=[C_{\mathit{F}_{m,n}^{i}(z)}^{(n)}](C_{\mathit{F}_{m,n}^{i}(a)}^{(n)}C_{\mathit{F}_{m,n}^{i}(y^{-1})}^{(n)}),
\end{equation}
for all $a\in\mathrm{S}_{m}$. By Claim~(a) of \Cref{Lem:3:3}, $F_{m,n}^{i}:\mathrm{H}_{m}^{\mathbb{A}}\rightarrow\mathrm{H}_{n}^{\mathbb{A}}$ is an $\mathbb{A}$-algebra monomorphism such that
\[ F_{m,n}^{i}(C_{w}^{(m)})=C_{F_{m,n}^{i}(w)}^{(n)}. \]
Thus, we can pull the terms within \Cref{Eq:3:18-1} back through $F_{m,n}^{i}$, obtaining the equalities 
\begin{equation*}
[C_{z}^{(m)}](C_{a}^{(m)}C_{x^{-1}}^{(m)})=[C_{z}^{(m)}](C_{a}^{(m)}C_{y^{-1}}^{(m)}),
\end{equation*}
for all $a\in\mathrm{S}_{m}$. Applying  \Cref{Eq:2.8:2} one last time, we obtain the equalities
\begin{equation}\label{Eq:3:18-2}
[D_{a}^{(m)}](D_{z}^{(m)}C_{x}^{(m)})=[D_{a}^{(m)}](D_{z}^{(m)}C_{y}^{(m)}),
\end{equation}
for all $a\in\mathrm{S}_{m}$. Moreover, since $x,y\leq_{R}z^{-1}$, then by \Cref{Eq:2.9:5}, at least one of these coefficients is non-zero. \Cref{Eq:3:18-2} and this non-zero condition tell us that $D_{z}^{(m)}C_{x}^{(m)}=D_{z}^{(m)}C_{y}^{(m)}\neq0$.
\end{proof}

The following lemma is simply a special case of \cite[Corollary 43]{CMZ19} (up to a natural equivalence of categories). Although this is an immediate application of the referred corollary, we have added a proof just to allow one to easier compare notation. 

\begin{lem}\label{Lem:3:19}
For any $m\leq n$, $x,y\in\mathrm{S}_{m}$ such that $x\leq_{R}^{(m)}y^{-1}$, and shift $F_{m,n}^{i}$, we have that
\[  \mathsf{KM}_{m}(x,y)=\mathsf{KM}_{n}(F_{m,n}^{i}(x),F_{m,n}^{i}(y)). \]
\end{lem}

\begin{proof}
Assume $\mathsf{KM}_{m}(x,y)=\mathtt{true}$. Equivalently, the module $\theta_{x}^{(m)}L_{y}^{(m)}$ is either zero or indecomposable. Let $\mathfrak{l}\subset\mathfrak{sl}_{n}$ denote the Levi factor associated to the parabolic subgroup $\mathrm{Im}(F_{m,n}^{i})\subset\mathrm{S}_{n}$, with $\mathcal{O}_{0}(\mathfrak{l})$ and $\mathcal{P}_{0}(\mathfrak{l})$ the corresponding principal block and monoidal category of projective functors. Since $\mathrm{S}_{m}\cong\mathrm{Im}(F_{m,n}^{i})$ in a natural way, then $\mathfrak{sl}_{m}\cong\mathfrak{l}$, and we have the equivalences of categories
\[ \mathcal{O}_{0}^{(m)}\cong\mathcal{O}_{0}(\mathfrak{l}) \hspace{2mm} \text{ and } \hspace{2mm} \mathcal{P}_{0}^{(m)}\cong\mathcal{P}_{0}(\mathfrak{l}), \]
given on objects by $L_{w}^{(m)}\mapsto L_{F_{m,n}^{i}(w)}^{\mathfrak{l}}$ and $\theta_{w}^{(m)}\mapsto\theta_{F_{m,n}^{i}(w)}^{\mathfrak{l}}$, respectively, for all $w\in\mathrm{S}_{m}$. Here the superscript of $\mathfrak{l}$ is added to stress the category in which the objects live. Moreover, under the former equivalence, we have $\theta_{x}^{(m)}L_{y}^{(m)}\mapsto\theta_{x}^{\mathfrak{l}}L_{y}^{\mathfrak{l}}$. Therefore, we immediately have that
\[ \mathsf{KM}_{m}(x,y)=\mathsf{KM}_{\mathfrak{l}}(F_{m,n}^{i}(x),F_{m,n}^{i}(y)), \]
where $\mathsf{KM}_{\mathfrak{l}}(w_{1},w_{2})$ is the truth value of the statement ``the module $\theta_{w_{1}}^{\mathfrak{l}}L_{w_{1}}^{\mathfrak{l}}$ in $\mathcal{O}_{0}(\mathfrak{l})$ is either zero or indecomposable''. Now by \cite[Corollary 43]{CMZ19}, we have that 
\[ \mathsf{KM}_{\mathfrak{l}}(F_{m,n}^{i}(x),F_{m,n}^{i}(y))=\mathsf{KM}_{n}(F_{m,n}^{i}(x),zF_{m,n}^{i}(y)), \]
where $z$ is any element in $X_{m,n}^{i}$. Picking $z=1_{n}$ completes the proof of the lemma.
\end{proof}

\section{Kostant's problem for permutations of shape $(n-2,1,1)$}\label{Sec:4}

Given $n\geq3$, and for any $i,j\in[2,n]:=\{2,3,\dots,n\}$ such that $i<j$, we consider the following involutions of $\mathrm{S}_{n}$ written alongside their corresponding standard Young tableau via \eqref{Eq:2.4:0}:
\begin{equation}\label{Eq:4:1}
z_{i,j}^{n}:=s_{i-1}s_{i}\cdots s_{j-2}s_{j-1}s_{j-2}\cdots s_{i}s_{i-1}=(i-1,j), \ \text{ where } \ \tilde{\mathtt{RS}}_{n}(z_{i,j}^{n})=\hspace{1mm}{\small{\begin{ytableau}i\\j\end{ytableau}}}^{\ \langle n-2\rangle}.
\end{equation}
Hence $z_{i,j}^{n}$ is the involution whose corresponding standard Young tableau is of shape $(1,1)^{(n-2)}$ and contains $i$ and $j$ as the unique entries in the second and third row, respectively. These account for all involutions of shape $(1,1)^{(n-2)}$, and give the collection of all non-elementary transpositions.

In this section we answer Kostant's problem for $z_{i,j}^{n}$ via a simple criteria on the indices $i$ and $j$, and equivalently a criteria involving the consecutively containment of a pattern. We will also show that $z_{i,j}^{n}$ satisfies both the Indecomposability Conjecture \ref{Conj:2.12:1} and K{\aa}hrstr{\"o}m's Conjecture \ref{Conj:2.14:1}. Lastly, we will prove the Asymptotic Shape Conjecture \ref{Conj:2.13:5} for the corresponding shape $(1,1)$. 

\begin{lem}\label{Lem:4:2}
Every involution $z_{i,j}^{n}\in\mathrm{Inv}_{n}$, as given in \Cref{Eq:4:1}, satisfies the Indecomposability Conjecture \ref{Conj:2.12:1}. That is to say, we have that $\mathbf{KM}_{n}(\star,z_{i,j}^{n})=\mathtt{true}$.
\end{lem}

\begin{proof}
By \Cref{Eq:2.12:2}, it suffices to show $\mathbf{KM}_{n}(x,z_{i,j}^{n})=\mathtt{true}$ for $x\in\mathrm{Inv}_{n}$. By \Cref{Eq:2.11:1}, we can assume $x\leq_{R}z_{i,j}^{n}$, and from \Cref{Eq:2.9:4}, this implies that
\[ \mathtt{sh}(z_{i,j}^{n})=(1,1)^{(n-2)}\preceq\mathtt{sh}(x). \]
Hence, the shape of $x$ must be either $\emptyset^{(n)}$, $(1)^{(n-1)}$, $(2)^{(n-2)}$, or $(1,1)^{(n-2)}$. In the latter case, the element $x$ has the same shape as $z_{i,j}^{n}$, and thus $\mathbf{KM}_{n}(x,z_{i,j}^{n})=\mathtt{true}$ by \cite[Section 5.2]{KiM16}. For the other three shapes, one can check that any involution with such a shape has support of size at most $4$, and thus $\mathbf{KM}_{n}(x,z_{i,j}^{n})=\mathtt{true}$ by 
the implication in Equation~\eqref{Eq:2.12:4}, completing the proof.
\end{proof}

We now seek to answer Kostant's problem for the involutions $z_{i,j}^{n}$, and simultaneously show that K{\aa}hrstr{\"o}m's Conjecture \ref{Conj:2.14:1} holds for these involutions. To help with this, we will first prove the following proposition, which accounts for the most technical aspect of the proof to come.

\begin{prop}\label{Prop:4:3}
Let $z_{i,j}^{n}\in\mathrm{Inv}_{n}$ be as given in \Cref{Eq:4:1}. For any distinct $x,y\leq_{R}z_{i,j}^{n}$ such that $\mathtt{sh}(x)=\mathtt{sh}(y)=(n-1,1)$, the following inequality within $\mathrm{H}_{n}^{\mathbb{Z}}$ holds: 
\[ D_{z_{i,j}^{n}}^{(n)}C_{x}^{(n)}\neq D_{z_{i,j}^{n}}^{(n)}C_{y}^{(n)}. \]
\end{prop}

\begin{proof}
Suppose, temporarily, there exists distinct $x,y\leq_{R}z_{i,j}^{n}$ with $\mathtt{sh}(x)=\mathtt{sh}(y)=(n-1,1)$ and
\[ D_{z_{i,j}^{n}}^{(n)}C_{x}^{(n)}=D_{z_{i,j}^{n}}^{(n)}C_{y}^{(n)}. \]
By the equivalence in Equation~\eqref{Eq:2.9:5}, both sides are non-zero, and by 
the implication in Equation~\eqref{Eq:2.9:4}, both $\mathrm{Des}_{L}(x)$ and $\mathrm{Des}_{L}(y)$ are subsets of $\mathrm{Des}_{L}(z_{i,j}^{n})=\{s_{i-1},s_{j-1}\}$. Moreover, by \Cref{Lem:2.4:1}, this implies that both
\[ \mathtt{P}_{x}, \mathtt{P}_{y}\in\left\{
\begin{ytableau}
\scriptstyle i
\end{ytableau}^{\ \langle n-1\rangle}, \ 
\begin{ytableau}
\scriptstyle j
\end{ytableau}^{\ \langle n-1\rangle} \right\}, \] 
recalling that $\mathtt{P}_{x}$ and $\mathtt{P}_{y}$ are the insertion tableau associated to $x$ and $y$ under the Robinson-Schensted correspondence, respectively. Furthermore, by the implication in Equation~\eqref{Eq:2.9:6}, $x\sim_{L}^{(n)}y$, equivalently $\mathtt{Q}_{x}=\mathtt{Q}_{y}$, and since $x$ and $y$ are distinct, we must have that $\mathtt{P}_{x}\neq\mathtt{P}_{y}$. Hence, up to symmetry, we have that
\begin{equation}\label{Eq:4:4}
\mathtt{P}_{x}=\begin{ytableau}
\scriptstyle i
\end{ytableau}^{\ \langle n-1\rangle} \hspace{3mm} \text{ and } \hspace{4mm}
\mathtt{P}_{y}=\begin{ytableau}
\scriptstyle j
\end{ytableau}^{\ \langle n-1\rangle}.
\end{equation}
Thus, for $D_{z_{i,j}^{n}}^{(n)}C_{x}^{(n)}=D_{z_{i,j}^{n}}^{(n)}C_{y}^{(n)}$ to hold, the pair $(x,y)$ must satisfy $x\sim_{R}^{(n)}s_{i-1}$, $y\sim_{R}^{(n)}s_{j-1}$, and $x\sim_{L}^{(n)}y$. There is one such pair $(x_{k},y_{k})$, for each $k\in[n-1]$, given by the following cases:
\begin{itemize}
\item[(a)] For $k\leq i-1$, then $x_{k}:=s_{i-1}s_{i-2}\cdots s_{k}$ and $y_{k}:=s_{j-1}s_{j-2}\cdots s_{k}$.
\item[(b)] For $i-1<k\leq j-1$, then $x_{k}:=s_{i-1}s_{i}\cdots s_{k}$ and $y_{k}:=s_{j-1}s_{j-2}\cdots s_{k}$.
\item[(c)] For $j-1<k$, then $x_{k}:=s_{i-1}s_{i}\cdots s_{k}$ and $y_{k}:=s_{j-1}s_{j}\cdots s_{k}$.
\end{itemize}
Therefore, to prove this proposition, it suffices to confirm that the inequality
\begin{equation}\label{Eq:4:5}
D_{z_{i,j}^{n}}^{(n)}C_{x_{k}}^{(n)}\neq D_{z_{i,j}^{n}}^{(n)}C_{y_{k}}^{(n)}
\end{equation}
holds within $\mathrm{H}_{n}^{\mathbb{Z}}$, for every pair $(x_{k},y_{k})$ appearing in the above three cases (a), (b), and (c). We will first show that this inequality holds in the following two special cases:

\textbf{Claim}: The inequality in Equation~\eqref{Eq:4:5} holds when $k=j-1$, and when $k=i-1$. In other words,
\begin{equation}\label{Eq:4:6}
D_{z_{i,j}^{n}}^{(n)}C_{s_{i-1}s_{i}\cdots s_{j-1}}^{(n)}\neq D_{z_{i,j}^{n}}^{(n)}C_{s_{j-1}}^{(n)} \ \text{ and } \ D_{z_{i,j}^{n}}^{(n)}C_{s_{i-1}}^{(n)}\neq D_{z_{i,j}^{n}}^{(n)}C_{s_{j-1}s_{j-2}\cdots s_{i-1}}^{(n)}.
\end{equation}
We will prove the first inequality for $k=j-1$, with the other following from analogous arguments. Firstly, by \Cref{Lem:2.6:2}, we must have that
\[ D_{z_{i,j}^{n}}^{(n)}C_{s_{i-1}s_{i}\cdots s_{j-1}}^{(n)}=D_{z_{i,j}^{n}}^{(n)}C_{s_{i-1}}^{(n)}C_{s_{i}}^{(n)}\cdots C_{s_{j-1}}^{(n)}. \]
Moreover, since $z_{i,j}^{n}=s_{i-1}s_{i}\cdots s_{j-2}s_{j-1}s_{j-2}\cdots s_{i}s_{i-1}$, we see that
\begin{equation}\label{Eq:4:7}
s_{i-1}\in\mathrm{Des}_{R}(z_{i,j}^{n}), \ s_{i}\in\mathrm{Des}_{R}(z_{i,j}^{n}s_{i-1}), \ \dots, \ s_{j-1}\in\mathrm{Des}_{R}(z_{i,j}^{n}s_{i-1}s_{i}\cdots s_{j-2}).
\end{equation}
Therefore, if one repeatably applies \Cref{Eq:2.8:1} on the expression $D_{z_{i,j}^{n}}C_{s_{i-1}}C_{s_{i}}\cdots C_{s_{j-1}}$ until it  becomes a linear combination of dual Kazhdan-Lusztig basis elements, we see that the left hand side of the first inequality of Equation~\eqref{Eq:4:6} contains $D_{z_{i,j}^{n}s_{i-1}s_{i}\cdots s_{j-1}}$, that is to say,
\[ D_{z_{i,j}^{n}s_{i-1}s_{i}\cdots s_{j-1}}^{(n)}\in D_{z_{i,j}^{n}}^{(n)}C_{s_{i-1}s_{i}\cdots s_{j-1}}^{(n)} \iff [D_{z_{i,j}^{n}s_{i-1}s_{i}\cdots s_{j-1}}^{(n)}](D_{z_{i,j}^{n}}^{(n)}C_{s_{i-1}s_{i}\cdots s_{j-1}}^{(n)})\neq0.  \]
On the other hand, by \Cref{Eq:2.8:1}, the only term $D_{z}^{(n)}$ appearing within the right hand side of the first inequality of Equation~\eqref{Eq:4:6}, such that $\ell(z)<\ell(z_{i,j}^{n})$, is the term $D_{z_{i,j}^{n}s_{j-1}}^{(n)}$. Moreover, we know that
\[ D_{z_{i,j}^{n}s_{j-1}}^{(n)}\neq D_{z_{i,j}^{n}s_{i-1}s_{i}\cdots s_{j-1}}^{(n)},\]
since $i\neq j$. Hence the first inequality from Equation~\eqref{Eq:4:6} holds. As mentioned above, an analogous argument shows the second inequality, proving the claim. We now use this claim to prove the inequality of Equation~\eqref{Eq:4:5} for each pair $(x_{k},y_{k})$ appearing in the above three cases.

Case (a): Here we have $k\leq i-1$, $x_{k}=s_{i-1}s_{i-2}\cdots s_{k}$, and $y_{k}=s_{j-1}s_{j-2}\cdots s_{k}$. We prove Equation~\eqref{Eq:4:5} by downward induction on $k$, with the base case $k=i-1$ holding from the above claim. Let $k<i-1$, and assume Equation~\eqref{Eq:4:5} holds for all $k'$ such that $k<k'\leq i-1$. Firstly, note that we have 
\begin{align*}
C_{x_{k}}^{(n)}C_{s_{k+1}}^{(n)}&=C_{s_{i-1}s_{i-2}\cdots s_{k+2}s_{k+1}s_{k}}^{(n)}C_{s_{k+1}}^{(n)} \\
&=C_{s_{i-1}}^{(n)}C_{s_{i-2}}^{(n)}\cdots C_{s_{k+2}}^{(n)}(C_{s_{k+1}}^{(n)}C_{s_{k}}^{(n)}C_{s_{k+1}}^{(n)}) \\
&=C_{s_{i-1}}^{(n)}C_{s_{i-2}}^{(n)}\cdots C_{s_{k+2}}^{(n)}(C_{s_{k+1}s_{k}s_{k+1}}^{(n)}+C_{s_{k+1}}^{(n)}) \\
&=C_{x_{k}s_{k+1}}^{(n)}+C_{x_{k+1}}^{(n)}.
\end{align*}
Here the second equality follows by \Cref{Lem:2.6:2}, the third equality follows since
\[ C_{s_{k+1}}^{(n)}C_{s_{k}}^{(n)}C_{s_{k+1}}^{(n)}=C_{s_{k+1}s_{k}s_{k+1}}^{(n)}+C_{s_{k+1}}^{(n)}, \]
which can be confirmed computationally by checking that
\[ C_{s_{2}}^{(3)}C_{s_{1}}^{(3)}C_{s_{2}}^{(3)}=C_{s_{2}s_{1}s_{2}}^{(3)}+C_{s_{2}}^{(3)} \]
holds within $\mathrm{H}_{3}^{\mathbb{Z}}$, and then shifting this equation into $\mathrm{H}_{n}^{\mathbb{Z}}$ via $F_{3,n}^{k}$, and the last equality follows again by  \Cref{Lem:2.6:2}. The same computations give an analogous result for $y_{k}$ too. In summary, we have
\begin{equation}\label{Eq:4:8}
C_{x_{k}}C_{s_{k+1}}^{(n)}=C_{x_{k}s_{k+1}}^{(n)}+C_{x_{k+1}}^{(n)} \ \text{ and } \ C_{y_{k}}^{(n)}C_{s_{k+1}}^{(n)}=C_{y_{k}s_{k+1}}^{(n)}+C_{y_{k+1}}^{(n)}.
\end{equation}
Now assume, for contradiction, that the equation $D_{z_{i,j}^{n}}^{(n)}C_{x_{k}}^{(n)}=D_{z_{i,j}^{n}}^{(n)}C_{y_{k}}^{(n)}$ holds. Then we have that
\begin{align*}
D_{z_{i,j}^{n}}^{(n)}C_{x_{k}}^{(n)}=D_{z_{i,j}^{n}}^{(n)}C_{y_{k}}^{(n)} &\implies D_{z_{i,j}^{n}}^{(n)}C_{x_{k}}^{(n)}C_{s_{k+1}}^{(n)}=D_{z_{i,j}^{n}}^{(n)}C_{y_{k}}^{(n)}C_{s_{k+1}}^{(n)} \\
&\implies D_{z_{i,j}^{n}}^{(n)}(C_{x_{k}s_{k+1}}^{(n)}+C_{x_{k+1}}^{(n)})=D_{z_{i,j}^{n}}^{(n)}(C_{y_{k}s_{k+1}}^{(n)}+C_{y_{k+1}}^{(n)}) \\
&\implies D_{z_{i,j}^{n}}^{(n)}C_{x_{k+1}}^{(n)}=D_{z_{i,j}^{n}}^{(n)}C_{y_{k+1}}^{(n)}.
\end{align*}
Here the second implication follows by Equation~\eqref{Eq:4:8}, while the last implication follows since both
\[ D_{z_{i,j}^{n}}^{(n)}C_{x_{k}s_{k+1}}^{(n)}=0 \hspace{1mm} \text{ and } \hspace{1mm} D_{z_{i,j}^{n}}^{(n)}C_{y_{k}s_{k+1}}^{(n)}=0, \]
which can be deduced from \Cref{Eq:2.9:5,Eq:2.9:4}, and noting that the element $s_{k}$ belongs to both $\mathrm{Des}_{L}(x_{k}s_{k+1})$ and $\mathrm{Des}_{L}(y_{k}s_{k+1})$, but does not belong to $\mathrm{Des}_{L}(z_{i,j}^{n})=\{s_{i-1},s_{j-1}\}$. The last equation above is
\[ D_{z_{i,j}^{n}}^{(n)}C_{x_{k+1}}^{(n)}=D_{z_{i,j}^{n}}^{(n)}C_{y_{k+1}}^{(n)}. \] 
However, by the induction hypothesis, this gives the desired contradiction, completing the proof of Case (a). The other two cases below follow by similar analysis, so we will be lighter on the details.

Case (b): Here we have $i-1<k\leq j-1$, $x_{k}=s_{i-1}s_{i}\cdots s_{k}$, and $y_{k}=s_{j-1}s_{j-2}\cdots s_{k}$. Again, we  prove that the inequality in Equation~\eqref{Eq:4:5} holds by downward induction on $k$, with the base case $k=j-1$ holding from the claim above. Let $i-1<k<j-1$, and assume that Equation~\eqref{Eq:4:5} holds for all $k'$ such that $i-1<k<k'\leq j-1$. Firstly, one can deduce that we have
\begin{equation}\label{Eq:4:9}
C_{x_{k}}^{(n)}C_{s_{k+1}}^{(n)}=C_{x_{k+1}}^{(n)} \ \text{ and } \ C_{y_{k}}^{(n)}C_{s_{k+1}}^{(n)}=C_{y_{k}s_{k+1}}^{(n)}+C_{y_{k+1}}^{(n)}.
\end{equation}
The first equation here holds by \Cref{Lem:2.6:2}, while the second holds by similar considerations made in Case (a) above. Assume, for contradiction, that $D_{z_{i,j}^{n}}^{(n)}C_{x_{k}}^{(n)}=D_{z_{i,j}^{n}}^{(n)}C_{y_{k}}^{(n)}$ holds. Then we have that
\begin{align*}
D_{z_{i,j}^{n}}^{(n)}C_{x_{k}}^{(n)}=D_{z_{i,j}^{n}}^{(n)}C_{y_{k}}^{(n)} &\implies D_{z_{i,j}^{n}}^{(n)}C_{x_{k}}^{(n)}C_{s_{k+1}}^{(n)}=D_{z_{i,j}^{n}}^{(n)}C_{y_{k}}^{(n)}C_{s_{k+1}}^{(n)} \\
&\implies D_{z_{i,j}^{n}}^{(n)}C_{x_{k+1}}^{(n)}=D_{z_{i,j}^{n}}^{(n)}(C_{y_{k}s_{k+1}}^{(n)}+C_{y_{k+1}}^{(n)}) \\
&\implies D_{z_{i,j}^{n}}^{(n)}C_{x_{k+1}}^{(n)}=D_{z_{i,j}^{n}}^{(n)}C_{y_{k+1}}^{(n)}.
\end{align*}
The second implication here follows by \Cref{Eq:4:9}, and again, the third implication follows since
\[ D_{z_{i,j}^{n}}^{(n)}C_{y_{k}s_{k+1}}^{(n)}=0, \]
which can be seen as $s_{k}$ belongs to $\mathrm{Des}_{L}(y_{k}s_{k+1})$ but not to $\mathrm{Des}_{L}(z_{i,j}^{n})$. The last equality above is
\[ D_{z_{i,j}^{n}}^{(n)}C_{x_{k+1}}^{(n)}=D_{z_{i,j}^{n}}^{(n)}C_{y_{k+1}}^{(n)}, \]
which gives the desired contradiction by the induction hypothesis, completing the proof of Case (b).

Case (c): Here we have $j-1<k$, $x_{k}=s_{i-1}s_{i}\cdots s_{k}$, and $y_{k}=s_{j-1}s_{j}\cdots s_{k}$. This time, we will prove that the inequality in Equation~\eqref{Eq:4:5} holds by regular induction on $k$ for $j-1\leq k$, with the base case $k=j-1$ holding by the above claim/Case (b). Let $j-1<k$, and assume Equation~\eqref{Eq:4:5} holds for all $k'$ such that $j-1\leq k'<k$. Now, by similar considerations made in Case (a), one can deduce that
\begin{equation}\label{Eq:4:10}
C_{x_{k}}^{(n)}C_{s_{k-1}}^{(n)}=C_{x_{k}s_{k-1}}^{(n)}+C_{x_{k-1}}^{(n)} \ \text{ and } \ C_{y_{k}}^{(n)}C_{s_{k-1}}^{(n)}=C_{y_{k}s_{k-1}}^{(n)}+C_{y_{k-1}}^{(n)}.
\end{equation}
Assume, for contradiction, that $D_{z_{i,j}^{n}}^{(n)}C_{x_{k}}^{(n)}=D_{z_{i,j}^{n}}^{(n)}C_{y_{k}}^{(n)}$ holds. Then we have that
\begin{align*}
D_{z_{i,j}^{n}}^{(n)}C_{x_{k}}^{(n)}=D_{z_{i,j}^{n}}^{(n)}C_{y_{k}}^{(n)} &\implies D_{z_{i,j}^{n}}^{(n)}C_{x_{k}}^{(n)}C_{s_{k-1}}^{(n)}=D_{z_{i,j}^{n}}^{(n)}C_{y_{k}}^{(n)}C_{s_{k-1}}^{(n)} \\
&\implies D_{z_{i,j}^{n}}^{(n)}(C_{x_{k}s_{k-1}}^{(n)}+C_{x_{k-1}}^{(n)})=D_{z_{i,j}^{n}}^{(n)}(C_{y_{k}s_{k-1}}^{(n)}+C_{y_{k-1}}^{(n)}) \\
&\implies D_{z_{i,j}^{n}}^{(n)}C_{x_{k-1}}^{(n)}=D_{z_{i,j}^{n}}^{(n)}C_{y_{k-1}}^{(n)}.
\end{align*}
The second implication follows by \Cref{Eq:4:10}, and the third follows since both 
\[ D_{z_{i,j}^{n}}^{(n)}C_{x_{k}s_{k-1}}^{(n)}=0 \hspace{1mm} \text{ and } \hspace{1mm} D_{z_{i,j}^{n}}^{(n)}C_{y_{k}s_{k-1}}^{(n)}=0, \]
which holds since $s_{k}$ belongs to both $\mathrm{Des}_{L}(x_{k}s_{k-1})$ and $\mathrm{Des}_{L}(y_{k}s_{k-1})$, but not to $\mathrm{Des}_{L}(z_{i,j}^{n})$. The last equality above is $D_{z_{i,j}^{n}}^{(n)}C_{x_{k-1}}^{(n)}=D_{z_{i,j}^{n}}^{(n)}C_{y_{k-1}}^{(n)}$, which gives the desired contradiction by the induction hypothesis, completing the proof of Case (c), and the proposition.
\end{proof}

\begin{thm}\label{Thm:4:11}
Let $z_{i,j}^{n}\in\mathrm{Inv}_{n}$ be as given in \Cref{Eq:4:1}.
Then the following conditions are equivalent.
\begin{enumerate}[$($i$)$]
\item $z_{i,j}^{n}$ is Kostant negative;
\item $z_{i,j}^{n}$ consecutively contains the pattern $14325$;
\item $j=i+1$ and $\{i,i+1\}\cap\{2,n\}=\emptyset$.
\end{enumerate}
Moreover, all involutions $z_{i,j}^{n}$ from \Cref{Eq:4:1} satisfy K{\aa}hrstr{\"o}m's Conjecture \ref{Conj:2.14:1}.
\end{thm}

\begin{proof}
As $z_{i,j}^{n}=(i-1,j)$, its one-line notation differs from the identity by exchanging the letters $i-1$ and $j$. Therefore, one can deduce that $z_{i,j}^{n}$ consecutively contains the pattern $p:=14325=(2,4)$ if and only if it is a shift of $p$, which occurs precisely when $j=i+1$ and $\{i,i+1\}\cap\{2,n\}=\emptyset$. Thus the equivalence of conditions (ii) and (iii) mentioned in the statement of the theorem holds. Moreover, if $z_{i,j}^{n}$ consecutively contains $p$ at position $i$, then by \cite[Theorem 3.6]{CM25-1} and \cite[Corollary 3.7]{CM25-1}, it is Kostant negative and satisfies K\aa hrstr\"om's Conjecture. 

What remains to be proven is that $z_{i,j}^{n}$ is Kostant positive, and satisfies K\aa hrstr\"om's Conjecture, whenever it consecutively avoids the pattern $14325$. By \Cref{Lem:4:2} and \Cref{Eq:2.14:2,}, it suffices to prove that Condition~(5) from K\aa hrstr\"om's Conjecture \ref{Conj:2.14:1} holds. That is to say, given $z_{i,j}^{n}$ which consecutively avoids the pattern $14325$, we seek to show that for any distinct $x,y\leq_{R}^{(n)}z_{i,j}^{n}$, the inequality
\begin{equation}\label{Eq:4:12}
D_{z_{i,j}^{n}}^{(n)}C_{x}^{(n)}\neq D_{z_{i,j}^{n}}^{(n)}C_{y}^{(n)}
\end{equation}
holds in $\mathrm{H}_{n}^{\mathbb{Z}}$. By above, the involutions $z_{i,j}^{n}$ which consecutively avoid the pattern $14325$ are precisely those such that $j\neq i+1$ and/or $\{i,j\}\cap\{2,n\}\neq0$. We split this into the following three cases:
\begin{itemize}
\item[(a)] $i=2$ and $3\leq j\leq n$,
\item[(b)] $j=n$ and $2\leq i\leq n-1$,
\item[(c)] $3\leq i,j\leq n-1$ and $|i-j|>1$.
\end{itemize}
In Case (a), $z_{2,j}^{n}=(1,j)$, while in Case (b), $z_{i,n}^{n}=(i-1,n)$. Therefore, the collection of involutions from Case (a) are in bijection with those from Case (b) via conjugation by $w_{0,n}$. By \cite[Lemma 30]{CM25-2} and \cite[Lemma 31]{CM25-2}, conjugating Inequality \eqref{Eq:4:12} by $T_{w_{0,n}}^{(n)}$ gives
\[ D_{z_{i,j}^{n}}^{(n)}C_{x}^{(n)}\neq D_{z_{i,j}^{n}}^{(n)}C_{y}^{(n)} \iff D_{w_{0,n}z_{i,j}^{n}w_{0,n}}^{(n)}C_{w_{0,n}xw_{0,n}}^{(n)}\neq D_{w_{0,n}z_{i,j}^{n}w_{0,n}}^{(n)}C_{w_{0,n}yw_{0,n}}^{(n)}. \]
Hence, proving Equation~\eqref{Eq:4:12} for involutions in Case~(a) is equivalently to showing the same for Case~(b). Thus, the result follows by confirming Equation~\eqref{Eq:4:12} for Cases~(a) and (c), which we now do in turn.

Case (a): We seek to prove the inequality in Equation~\eqref{Eq:4:12} when $z_{2,j}^{n}=(1,j)$ and $3\leq j\leq n$. It will first be helpful to understand what permutations $x$ have a chance at satisfying the relation $x\leq_{R}^{(n)}z_{i,j}^{n}$. By the implication in Equation~\eqref{Eq:2.9:4}, we must have that $\mathrm{Des}_{L}(x)\subseteq\mathrm{Des}_{L}(z_{2,j}^{n})=\{s_{1},s_{j-1}\}$ and $\mathtt{sh}(z_{2,j}^{n})=(n-2,1,1)\preceq\mathtt{sh}(x)$. Recalling \Cref{Lem:2.4:1}, this implies that $\mathtt{P}_{x}$ (the insertion tableau of $x$) must be one of the following:
\begin{equation}\label{Eq:4:13}
\emptyset^{\langle n\rangle} \hspace{6mm}
\begin{ytableau}
\scriptstyle 2
\end{ytableau}^{\ \langle n-1\rangle} \hspace{6mm}
\begin{ytableau}
\scriptstyle j
\end{ytableau}^{\ \langle n-1\rangle} \hspace{6mm}
{\color{blue}
\begin{ytableau}
\scriptstyle 2&\scriptstyle j
\end{ytableau}^{\ \langle n-2\rangle}} \hspace{6mm}
{\color{red}
\begin{ytableau}
\scriptstyle j&\scriptstyle j+1
\end{ytableau}^{\ \langle n-2\rangle}} \hspace{6mm}
\begin{ytableau}
\scriptstyle 2\\\scriptstyle j
\end{ytableau}^{\ \langle n-2\rangle}.
\end{equation}
Here the {\color{blue}blue} tableau only exists when $j>3$, while the {\color{red}red} tableau only exists when $j<n$. To show that the inequality in Equation~\eqref{Eq:4:12} holds, let us assume otherwise, and find a contradiction. Therefore, assume there exists distinct $x,y\leq_{R}^{(n)}z_{2,j}^{n}$ such that we have the equality
\[ D_{z_{2,j}^{n}}^{(n)}C_{x}^{(n)}=D_{z_{2,j}^{n}}^{(n)}C_{y}^{(n)}, \]
noting that both sides are non-zero by the equivalence in Equation~\eqref{Eq:2.9:5}. As $x,y\leq_{R}^{(n)}z_{2,j}^{n}$, both $\mathtt{P}_{x}$ and $\mathtt{P}_{y}$ belong to List \eqref{Eq:4:13}. By the implication  in Equation~\eqref{Eq:2.9:6}, $x\sim_{L}^{(n)}y$, equivalently $\mathtt{Q}_{x}=\mathtt{Q}_{y}$, and, in particular, $\mathtt{sh}(x)=\mathtt{sh}(y)$. Furthermore, as $x$ and $y$ are distinct, this implies that $\mathtt{P}_{x}\neq\mathtt{P}_{y}$. Hence altogether, both $\mathtt{P}_{x}$ and $\mathtt{P}_{y}$ must belong to the list  in Equation~\eqref{Eq:4:13}, share the same shape, but be distinct. The only two options we have are
\[ \{\mathtt{P}_{x},\mathtt{P}_{y}\}=
\left\{\begin{ytableau}
\scriptstyle 2
\end{ytableau}^{\ \langle n-1\rangle}, \ 
\begin{ytableau}
\scriptstyle j
\end{ytableau}^{\ \langle n-1\rangle}\right\} \ \text{ or } \ 
 \{\mathtt{P}_{x},\mathtt{P}_{y}\}=
\left\{{\color{blue}
\begin{ytableau}
\scriptstyle 2&\scriptstyle j
\end{ytableau}^{\ \langle n-2\rangle}}, \ 
{\color{red}
\begin{ytableau}
\scriptstyle j&\scriptstyle j+1
\end{ytableau}^{\ \langle n-2\rangle}}\right\}. \]
The former situation cannot occur by \Cref{Prop:4:3}. As for the latter situation, we require that $3<j<n$ for both tableaux to exist. Assume that this is the case, and (without loss of generality) let $\mathtt{P}_{x}$ be the {\color{red}red} tableau from the list in Equation~\eqref{Eq:4:13}. Now let $a\sim_{R}^{(n)}x$ be the unique involution in the same right cell as $x$, hence $\mathtt{P}_{a}=\mathtt{P}_{x}$. Then one can check that $a=s_{j-1}s_{j}s_{j-2}s_{j-1}$. Since $3<j<n$, we can see that $z_{2,j}^{n}=(1,j)$ consecutively contains the pattern $2314$ at position $j-2$. So by Claim~(b) of \Cref{Prop:3:7},
\[ F_{4,n}^{j-2}(s_{2}s_{3}s_{1}s_{2})=s_{j-1}s_{j}s_{j-2}s_{j-1}=a\not\leq_{L}^{(n)}z_{2,j}^{n} \ \text{ since } \ s_{2}s_{3}s_{1}s_{2}\not\leq_{L}^{(4)}2314, \]
where the latter non-relation can be checked computationally, or by consulting \Cref{Fig:3:9-1}. Now, since $a$ and $z_{2,j}^{n}$ are involutions, then applying the equivalence in Equation~\eqref{Eq:2.9:2} tells us that $a\not\leq_{R}^{(n)}z_{i,j}^{n}$. However, this implies $x\not\leq_{R}^{(n)}z_{2,j}^{n}$ as $x\sim_{R}^{(n)}a$, which gives the desired contradiction. This completes the proof of  Case~(a).

Case (c): We seek to prove the inequality in Equation~\eqref{Eq:4:12} when $z_{i,j}^{n}=(i-1,j)$, $3\leq i,j\leq n-1$, and $|i-j|>1$. It will again be helpful to first understand the permutations $x$ which have a chance at satisfying the relation $x\leq_{R}^{(n)}z_{i,j}^{n}$. By the implication  in Equation~\eqref{Eq:2.9:4}, we have that $\mathrm{Des}_{L}(x)\subseteq\mathrm{Des}_{L}(z_{i,j}^{n})=\{s_{i-1},s_{j-1}\}$ and $\mathtt{sh}(z_{i,j}^{n})=(n-2,1,1)\preceq\mathtt{sh}(x)$. By \Cref{Lem:2.4:1}, this implies that $\mathtt{P}_{x}$ is one of the following:
\begin{equation}
\emptyset^{\langle n\rangle} \hspace{6mm}
\begin{ytableau}
\scriptstyle i
\end{ytableau}^{\ \langle n-1\rangle} \hspace{6mm}
\begin{ytableau}
\scriptstyle j
\end{ytableau}^{\ \langle n-1\rangle} \hspace{6mm}
\begin{ytableau}
\scriptstyle i&\scriptstyle j
\end{ytableau}^{\ \langle n-2\rangle} \hspace{6mm}
{\color{red}
\begin{ytableau}
\scriptstyle i&\scriptstyle i+1
\end{ytableau}^{\ \langle n-2\rangle}} \hspace{6mm}
{\color{magenta}
\begin{ytableau}
\scriptstyle j&\scriptstyle j+1
\end{ytableau}^{\ \langle n-2\rangle}} \hspace{6mm}
\begin{ytableau}
\scriptstyle i\\\scriptstyle j
\end{ytableau}^{\ \langle n-2\rangle}.
\end{equation}
Let $a=s_{i-1}s_{i}s_{i-2}s_{i-1}$ and $b=s_{j-1}s_{j}s_{j-2}s_{j-1}$ be the involutions with {\color{red}red} and {\color{magenta}magenta} insertion tableaux above, respectively. Since $3\leq i, j\leq n-1$ and $|i-j|>1$, one can see that $z_{i,j}^{n}=(i-1,j)$ consecutively contains the pattern $1423$ at position $i-2$, and the pattern $2314$ at position $j-2$. Therefore, by Claim~(b) of \Cref{Prop:3:7}, we have that
\begin{align*}
F_{4,n}^{i-2}(s_{2}s_{3}s_{1}s_{2})&=a\not\leq_{L}^{(n)}z_{i,j}^{n} \ \text{ since } \ s_{2}s_{3}s_{1}s_{2}\not\leq_{L}^{(4)}1423, \ \text{ and } \\
F_{4,n}^{j-2}(s_{2}s_{3}s_{1}s_{2})&=b\not\leq_{L}^{(n)}z_{i,j}^{n} \ \text{ since } \ s_{2}s_{3}s_{1}s_{2}\not\leq_{L}^{(4)}2314,
\end{align*}
where both the latter non-relations in $\mathrm{S}_{4}$ can be confirmed computationally, or by consulting \Cref{Fig:3:9-1}. Since $a$, $b$, and $z_{i,j}^{n}$ are involutions, then applying the equivalence  in Equation~\eqref{Eq:2.9:2} tells us that $a\not\leq_{R}^{(n)}z_{i,j}^{n}$ and $b\not\leq_{R}^{(n)}z_{i,j}^{n}$. So, if $x\leq_{R}^{(n)}z_{i,j}^{n}$, then $\mathtt{P}_{x}$ cannot be the {\color{red}red} or {\color{magenta}magenta} tableaux above, as otherwise $x\sim_{R}^{(n)}a$ or $x\sim_{R}^{(n)}b$, which would imply $x\not\leq_{R}^{(n)}z_{i,j}^{n}$. Thus if $x\leq_{R}^{(n)}z_{i,j}^{n}$, then $\mathtt{P}_{x}$ must belong to
\begin{equation}\label{Eq:4:14}
\emptyset^{\langle n\rangle} \hspace{6mm}
\begin{ytableau}
\scriptstyle i
\end{ytableau}^{\ \langle n-1\rangle} \hspace{6mm}
\begin{ytableau}
\scriptstyle j
\end{ytableau}^{\ \langle n-1\rangle} \hspace{6mm}
\begin{ytableau}
\scriptstyle i&\scriptstyle j
\end{ytableau}^{\ \langle n-2\rangle} \hspace{6mm}
\begin{ytableau}
\scriptstyle i\\\scriptstyle j
\end{ytableau}^{\ \langle n-2\rangle}.
\end{equation}
To show that the inequality  in Equation~\eqref{Eq:4:12} holds, let us assume otherwise, and find a contradiction. Therefore, assume there exists distinct $x,y\leq_{R}^{(n)}z_{i,j}^{n}$ such that we have the equality
\[ D_{z_{i,j}^{n}}^{(n)}C_{x}^{(n)}=D_{z_{i,j}^{n}}^{(n)}C_{y}^{(n)}, \]
noting that both sides are non-zero by the equivalence  in Equation~\eqref{Eq:2.9:5}. As $x,y\leq_{R}^{(n)}z_{2,j}^{n}$, both $\mathtt{P}_{x}$ and $\mathtt{P}_{y}$ belong to the list  in Equation~\eqref{Eq:4:14}. By the implication  in Equation~\eqref{Eq:2.9:6}, we know that $x\sim_{L}^{(n)}y$, which implies that $\mathtt{Q}_{x}=\mathtt{Q}_{y}$ and, in particular, $\mathtt{sh}(x)=\mathtt{sh}(y)$. Furthermore, since $x$ and $y$ are distinct, then $\mathtt{P}_{x}\neq\mathtt{P}_{y}$. Therefore, altogether $\mathtt{P}_{x}$ and $\mathtt{P}_{y}$ belong to the list in Equation~\eqref{Eq:4:14}, share the same shape, but are distinct. The only option is
\[ \{\mathtt{P}_{x},\mathtt{P}_{y}\}=
\left\{\begin{ytableau}
\scriptstyle i
\end{ytableau}^{\ \langle n-1\rangle}, \ 
\begin{ytableau}
\scriptstyle j
\end{ytableau}^{\ \langle n-1\rangle}\right\}. \]
However, this pairing cannot produce the assumed equality $D_{z_{i,j}^{n}}^{(n)}C_{x}^{(n)}=D_{z_{i,j}^{n}}^{(n)}C_{y}^{(n)}$ by \Cref{Prop:4:3}, giving the desired contradiction, and completing the proof of the theorem.
\end{proof}

\begin{rmk}
By the implications  in Equations~\eqref{Eq:2.12:3} and \eqref{Eq:2.13:2}, both the Indecomposability Conjecture and the answer to Kostant's problem are left cell invariants. Therefore, as the involutions $z_{i,j}^{n}$ given in \Cref{Eq:4:1} account for all involutions of shape $(n-2,1,1)$,  \Cref{Lem:4:2} confirms the Indecomposability Conjecture for all permutations with shape $(n-2,1,1)$, and \Cref{Thm:4:11} provides an answer to Kostant's problem for all permutations with shape $(n-2,1,1)$.
\end{rmk}

We conclude this section by confirming the Asymptotic Shape Conjecture \ref{Conj:2.13:5} in this case.

\begin{prop}\label{Prop:4:15}
The Asymptotic Shape Conjecture \ref{Conj:2.13:5} holds for the shape $\lambda=(1,1)$.
\end{prop}

\begin{proof}
It suffices to prove that the proportion of permutations of shape $(1,1)^{\langle n\rangle}$, which are Kostant negative, is asymptotically equal to 0, that is to say,
\[ \frac{\mathbf{k}_{2+n}^{-}((1,1)^{\langle n\rangle})}{|\mathtt{SYT}_{2+n}((1,1)^{\langle n\rangle})|^{2}}\rightarrow 0 \hspace{2mm} \text{(as $n\geq1$ tends towards $\infty$)}. \] 
By the \emph{hook length formula}, one can deduce that
\begin{equation}\label{Eq:4:16}
|\mathtt{SYT}_{2+n}((1,1)^{\langle n\rangle})|^{2}=\left(\frac{(n+2)!}{(n+2)\cdot(n-1)!\cdot 2}\right)^{2}=\frac{n^{2}(n+1)^{2}}{4}.
\end{equation}
Moreover, by \Cref{Thm:4:11}, we see that
\[
|\{z_{i,j}^{n+2} \ | \ \mathbf{K}_{n+2}(z_{i,j}^{n+2})=\mathtt{false}\}|=|\{z_{i,i+1}^{n+2} \ | \ i\in[2,n+1], \{i,i+1\}\cap\{2,n+2\}=\emptyset\}|=(n-2).
\]
Hence the number of involutions $z_{i,j}^{n+2}$ which are Kostant negative is $(n-2)$. Since these involutions account for all involutions of shape $(1,1)^{\langle n\rangle}$, to obtain the number of all Kostant negative permutations of shape $(1,1)^{\langle n\rangle}$, we simply need to multiply by $|\mathtt{SYT}_{2+n}((1,1)^{\langle n\rangle})|$, which is the number of permutations in the same left cell as such involutions. Therefore, we have that 
\begin{align*}
\mathbf{k}_{2+n}^{-}((1,1)^{\langle n\rangle})&=|\{z_{i,j}^{n+2} \ | \ \mathbf{K}_{n+2}(z_{i,j}^{n+2})=\mathtt{false}\}|\cdot|\mathtt{SYT}_{2+n}((1,1)^{\langle n\rangle})|\\
&=(n-2)\cdot\frac{(n+2)!}{(n+2)\cdot(n-1)!\cdot 2}=\frac{n(n-2)(n+1)}{2}.
\end{align*}
Combining this with \Cref{Eq:4:16}, we obtain
\[
\frac{\mathbf{k}_{2+n}^{-}((1,1)^{\langle n\rangle})}{|\mathtt{SYT}_{2+n}((1,1)^{\langle n\rangle})|^{2}}=\frac{4n(n-2)(n+1)}{2n^{2}(n+1)^{2}}=\frac{2(n-2)}{n(n+1)}\rightarrow 0
\]
as $n\geq1$ tends towards $\infty$, since the numerator is linear in $n$ while the denominator is quadratic.
\end{proof}

\section{Kostant's problem for permutations of shape $(n-3,2,1)$}\label{Sec:5}

Given $n\geq5$, let $i,j,k\in[2,n]:=\{2,3,\dots,n\}$ be such that $i<j$, $i<k$
and the set $[n]\setminus\{1,i,j,k\}$ contains an element smaller than $j$. 
Then we will denote by $z_{i,j,k}^{n}\in\mathrm{Inv}_{n}$ the involution whose corresponding standard Young tableaux is given by
\begin{equation}\label{Eq:5:1}
\tilde{\mathtt{RS}}_{n}(z_{i,j,k}^{n})=\hspace{1mm}{\small{\begin{ytableau}i&j\\k\end{ytableau}}}^{\ \langle n-3\rangle}.
\end{equation} 
Naturally, the collection of all such involutions accounts for all involutions of shape $(n-3,2,1)$. In this section we seek to provide an answer to Kostant's problem for these involutions. Moreover, we will confirm that the Indecomposability Conjecture \ref{Conj:2.12:1} holds for these involutions, and that the Asymptotic Shape Conjecture \ref{Conj:2.13:5} holds for the corresponding shape $(2,1)$. To do this, we will need to break the involution into different types (given below), which will be tackled one at a time.

Firstly, let us assume that $i<k<j$. Then we have the following exhaustive list of cases: 

\hspace{2mm} \textbf{Type (1)}: For $i\in[3,n-2]$ and $n\geq 5$, we have that
\begin{equation}\label{Type1} 
\tilde{\mathtt{RS}}_{n}(z_{i,i+2,i+1}^{n})={\begin{ytableau}\scriptstyle i&\scriptstyle i+2\\\scriptstyle i+1\end{ytableau}}^{\ \langle n-3\rangle} \hspace{4mm} \text{ and } \hspace{4mm} z_{i,i+2,i+1}^{n}=(i-2, i+1)(i-1, i+2).
\end{equation}

\hspace{2mm} \textbf{Type (2)}: For $i\in[2,n-3]$ and $n\geq 5$, we have that
\begin{equation}\label{Type2} 
\tilde{\mathtt{RS}}_{n}(z_{i,i+3,i+2}^{n})={\begin{ytableau}\scriptstyle i&\scriptstyle i+3\\\scriptstyle i+2\end{ytableau}}^{\ \langle n-3\rangle} \hspace{4mm} \text{ and } \hspace{4mm} z_{i,i+3,i+2}^{n}=(i-1, i+2)(i+1,i+3).
\end{equation}

\hspace{2mm} \textbf{Type (3)}: For $i\in[2,n-4]$, $i+2<k<n$ and $n\geq 6$, we have that
\begin{equation}\label{Type3} 
\tilde{\mathtt{RS}}_{n}(z_{i,k+1,k}^{n})={\begin{ytableau}\scriptstyle i&\scriptstyle k+1\\\scriptstyle k\end{ytableau}}^{\ \langle n-3\rangle} \hspace{4mm} \text{ and } \hspace{4mm} z_{i,k+1,k}^{n}=(i-1, k)(k-1,k+1).
\end{equation}

\hspace{2mm} \textbf{Type (4)}: For $i\in[2,n-3]$ and $n\geq 5$, we have that
\begin{equation}\label{Type4} 
\tilde{\mathtt{RS}}_{n}(z_{i,i+3,i+1}^{n})={\begin{ytableau}\scriptstyle i&\scriptstyle i+3\\\scriptstyle i+1\end{ytableau}}^{\ \langle n-3\rangle} \hspace{4mm} \text{ and } \hspace{4mm} z_{i,k,i+1}^{n}=(i-1, i+1)(i+2, i+3).
\end{equation}

\hspace{2mm} \textbf{Type (5)}: For $i\in[2,n-4]$, $i+1<k<n-1$, and $n\geq 6$, we have that
\begin{equation}\label{Type5} 
\tilde{\mathtt{RS}}_{n}(z_{i,k+2,k}^{n})={\begin{ytableau}\scriptstyle i&\scriptstyle k+2\\\scriptstyle k\end{ytableau}}^{\ \langle n-3\rangle} \hspace{4mm} \text{ and } \hspace{4mm} z_{i,k+2,k}^{n}=(i-1, k)(k+1, k+2).
\end{equation}

\hspace{2mm} \textbf{Type (6)}: For $i\in[2,n-4]$, $i+3<j\leq n$, and $n\geq 6$, we have that
\begin{equation}\label{Type6} 
\tilde{\mathtt{RS}}_{n}(z_{i,j,i+1}^{n})={\begin{ytableau}\scriptstyle i&\scriptstyle j\\\scriptstyle i+1\end{ytableau}}^{\ \langle n-3\rangle} \hspace{4mm} \text{ and } \hspace{4mm} z_{i,j,i+1}^{n}=(i-1, i+1)(j-1, j).
\end{equation}

\hspace{2mm} \textbf{Type (7)}: For $i\in[2,n-5]$, $i+1<k<n$, $k+2<j\leq n$, and $n\geq 7$, we have that
\begin{equation}\label{Type7} 
\tilde{\mathtt{RS}}_{n}(z_{i,j,k}^{n})={\begin{ytableau}\scriptstyle i&\scriptstyle j\\\scriptstyle k\end{ytableau}}^{\ \langle n-3\rangle} \hspace{4mm} \text{ and } \hspace{4mm} z_{i,j,k}^{n}=(i-1, k)(j-1, j).
\end{equation}

We briefly discuss why, under the assumption that $i<k<j$, the above seven types give a complete list of cases for the involution $z_{i,j,k}^{n}$. Firstly, for each of these types, one can check that the conditions imposed on $i$, $j$, $k$, and $n$, are such that the corresponding standard Young tableau is well-defined. Next, a pairwise comparison tells us that these types are disjoint. Lastly, we need to argue why any $z_{i,j,k}^{n}$ must belong to one of these types. To see this, note that types (1), (2), and (3) account for all standard Young tableaux with $j=k+1$, types (4) and (5) account for all standard Young tableaux with $j=k+2$, while types (6) and (7) account for the remaining cases for when $j>k+2$.

Now assume that $i<j<k$. Then we have the following exhaustive list of cases: 

\hspace{2mm} \textbf{Type (2}$^{\ast}$\textbf{)}: For $i\in[3,n-2]$ and $n\geq 5$, we have that
\begin{equation}\label{Type2star}
\tilde{\mathtt{RS}}_{n}(z_{i,i+1,i+2}^{n})={\begin{ytableau}\scriptstyle i&\scriptstyle i+1\\\scriptstyle i+2\end{ytableau}}^{\ \langle n-3\rangle} \hspace{4mm} \text{ and } \hspace{4mm} z_{i,i+1,i+2}^{n}=(i-2, i)(i-1,i+2).
\end{equation}

\hspace{2mm} \textbf{Type (3}$^{\ast}$\textbf{)}: For $i\in[3,n-3]$, $i+2< k\leq n$, and $n\geq 6$, we have that
\begin{equation}\label{Type3star}
\tilde{\mathtt{RS}}_{n}(z_{i,i+1,k}^{n})={\begin{ytableau}\scriptstyle i&\scriptstyle i+1\\\scriptstyle k\end{ytableau}}^{\ \langle n-3\rangle} \hspace{4mm} \text{ and } \hspace{4mm} z_{i,i+1,k}^{n}=(i-2, i)(i-1,k).
\end{equation}

\hspace{2mm} \textbf{Type (4}$^{\ast}$\textbf{)}: For $i\in[2,n-3]$ and $n\geq 5$, we have that
\begin{equation}\label{Type4star}
\tilde{\mathtt{RS}}_{n}(z_{i,i+2,i+3}^{n})={\begin{ytableau}\scriptstyle i&\scriptstyle i+2\\\scriptstyle i+3\end{ytableau}}^{\ \langle n-3\rangle} \hspace{4mm} \text{ and } \hspace{4mm} z_{i,i+2,i+3}^{n}=(i-1, i)(i+1, i+3).
\end{equation}

\hspace{2mm} \textbf{Type (5}$^{\ast}$\textbf{)}: For $i\in[2,n-4]$, $i+3<k\leq n$, and $n\geq 6$, we have that
\begin{equation}\label{Type5star}
\tilde{\mathtt{RS}}_{n}(z_{i,i+2,k}^{n})={\begin{ytableau}\scriptstyle i&\scriptstyle i+2\\\scriptstyle k\end{ytableau}}^{\ \langle n-3\rangle} \hspace{4mm} \text{ and } \hspace{4mm} z_{i,i+2,k}^{n}=(i-1, i)(i+1, k).
\end{equation}

\hspace{2mm} \textbf{Type (6}$^{\ast}$\textbf{)}: For $i\in[2,n-4]$, $i+2<j<n$, and $n\geq 6$, we have that
\begin{equation}\label{Type6star}
\tilde{\mathtt{RS}}_{n}(z_{i,j,j+1}^{n})={\begin{ytableau}\scriptstyle i&\scriptstyle j\\\scriptstyle j+1\end{ytableau}}^{\ \langle n-3\rangle} \hspace{4mm} \text{ and } \hspace{4mm} z_{i,j,j+1}^{n}=(i-1, i)(j-1, j+1).
\end{equation}

\hspace{2mm} \textbf{Type (7}$^{\ast}$\textbf{)}: For $i\in[2,n-5]$, $i+2<j<n$, $j+1<k\leq n$, and $n\geq 7$, we have that
\begin{equation}\label{Type7star}
\tilde{\mathtt{RS}}_{n}(z_{i,j,k}^{n})={\begin{ytableau}\scriptstyle i&\scriptstyle j\\\scriptstyle k\end{ytableau}}^{\ \langle n-3\rangle} \hspace{4mm} \text{ and } \hspace{4mm} z_{i,j,k}^{n}=(i-1, i)(j-1, k).
\end{equation}

Again, we briefly discuss why, under the assumption that $i<j<k$, the above six types give a complete list of cases for the involution $z_{i,j,k}^{n}$. Firstly, for each of these types, one can check that the conditions imposed on $i$, $j$, $k$, and $n$, are such that the corresponding standard Young tableau is well-defined. Next, a pairwise comparison tells us that these types are disjoint. Lastly, to see why any $z_{i,j,k}^{n}$ must belong to one of these types, note that types (2$^{\ast}$) and (3$^{\ast}$) account for all standard Young tableau with $j=i+1$, types (4$^{\ast}$) and (5$^{\ast}$) account for all standard Young tableau with $j=i+2$, while types (6$^{\ast}$) and (7$^{\ast}$) account for the remaining cases for when $j>i+2$.

\begin{rmk}\label{Rmk:5:2}
A natural question to ask here is why have we picked these particular types to consider. Put simply, these types allow for the proofs to come to be presented in a more consistent and readable manner. In more detail, we will be interested in understanding what kind of standard Young tableau $\mathtt{P}_{x}$ can possibly occur given that $x\leq_{R}^{(n)}z_{i,j,k}^{n}$, and to describe such tableaux in terms of $i$, $j$, and $k$. For example, two possible tableaux that $\mathtt{P}_{x}$ could be are
\[ {\begin{ytableau}\scriptstyle i&\scriptstyle i+1&\scriptstyle i+2\end{ytableau}}^{\ \langle n-3\rangle} \hspace{2mm} \text{ and } \hspace{3mm} {\begin{ytableau}\scriptstyle i&\scriptstyle j&\scriptstyle j+1\end{ytableau}}^{\ \langle n-3\rangle}. \] 
However, these two tableaux are only distinct whenever $j>i+1$, since they coincide when $j=i+1$. For the arguments we will make, it is significantly easier to work in situations where the tableaux which occur are uniquely given in terms of $i$, $j$, and $k$, that is, we do not want tableaux to coincide in special cases like the above pair. The types we have described above mostly achieve this property (with only one, very manageable, exception), and this is the main reason we are considering them.
\end{rmk}

Examining the cycle notation for the above types, one can deduce that, for each $T\in\{2,3,4,5,6,7\}$, the involutions of Type $(T)$ are in bijection with those of Type $(T^{\ast})$ via conjugation by the longest element $w_{0,n}$. For example, we will now demonstrate this for Types (7) and (7$^{\ast}$): Let $z:=z_{i,j.k}^{n}$ be of Type (7), thus $n\geq 7$, $2\leq i\leq n-5$, $i+1<k<n$, $k+2<j\leq n$, and $z=(i-1,k)(j-1,j)$. Expressing $z$ as a \emph{one-line diagram}, then we have the following:
\[ z\hspace{2mm}=\hspace{2mm}\begin{matrix}\begin{tikzpicture}
\node[V,color=red] (1) at (0,0){};
\node[draw=none] (1.5) at (0.5,0){{\color{red}$\dots$}};	
\node[V,color=red] (i-2) at (1,0){};	
\node[V, label=below:{{\scriptsize$i-1$}}] (i-1) at (2,0){};	
\node[V,color=teal] (i) at (3,0){};	
\node[draw=none] (i.5) at (3.5,0){{\color{teal}$\dots$}};	
\node[V,color=teal] (k-1) at (4,0){};	
\node[V, label=below:{{\scriptsize$k$}}] (k) at (5,0){};	
\node[V,color=blue] (k+1) at (6,0){};	
\node[draw=none] (k+1.5) at (6.5,0){{\color{blue}$\dots$}};	
\node[V,color=blue] (j-2) at (7,0){};	
\node[V, label=below:{{\scriptsize$j-1$}}] (j-1) at (8,0){};	
\node[V, label=below:{{\scriptsize$j$}}] (j) at (9,0){};	
\node[V,color=purple] (j+1) at (10,0){};	
\node[draw=none] (j+1.5) at (10.5,0){{\color{purple}$\dots$}};	
\node[V,color=purple] (1) at (11,0){};
\node[draw=none] (1.5L) at (0.5,-0.5){{\color{red}{\scriptsize$\geq 0$}}};	
\node[draw=none] (4.5L) at (3.5,-0.5){{\color{teal}{\scriptsize$\geq 2$}}};
\node[draw=none] (7.5L) at (6.5,-0.5){{\color{blue}{\scriptsize$\geq 1$}}};
\node[draw=none] (11.5L) at (10.5,-0.5){{\color{purple}{\scriptsize$\geq 0$}}};
\draw (i-1) to [bend left] (k) (j-1) to [bend left] (j);
\end{tikzpicture}\end{matrix}. \]
Here the dots are labelled $1$ to $n$ reading left to right, isolated dots correspond to fixed points, and edges correspond to transpositions. As indicated in the diagram, there can be any non-negative number of {\color{red}red} and {\color{purple}purple} dots, while there needs to be at least $2$ {\color{teal}teal} dots and at least $1$ {\color{blue}blue} dot. These conditions on the number of different coloured dots corresponds to the various inequalities imposed on $i$, $j$, and $k$, defining the Type (7) involution $z$. The one-line diagram for the involution $w_{0,n}zw_{0,n}$ is obtained by simply reflecting the above diagram horizontally, and replacing each label $l\in[n]$ of a dot with the label $w_{0,n}(l)=n-l+1$. Therefore, we have that
\[ w_{0,n}zw_{0,n}\hspace{2mm}=\hspace{2mm}\begin{matrix}\begin{tikzpicture}
\node[V,color=red] (1) at (11,0){};
\node[draw=none] (1.5) at (10.5,0){{\color{red}$\dots$}};	
\node[V,color=red] (i-2) at (10,0){};	
\node[V, label=below:{{\scriptsize$k'$}}] (i-1) at (9,0){};	
\node[V,color=teal] (i) at (8,0){};	
\node[draw=none] (i.5) at (7.5,0){{\color{teal}$\dots$}};	
\node[V,color=teal] (k-1) at (7,0){};	
\node[V, label=below:{{\scriptsize$j'-1$}}] (k) at (6,0){};	
\node[V,color=blue] (k+1) at (5,0){};	
\node[draw=none] (k+1.5) at (4.5,0){{\color{blue}$\dots$}};	
\node[V,color=blue] (j-2) at (4,0){};	
\node[V, label=below:{{\scriptsize$i'+1$}}] (j-1) at (3,0){};	
\node[V, label=below:{{\scriptsize$i'$}}] (j) at (2,0){};	
\node[V,color=purple] (j+1) at (1,0){};	
\node[draw=none] (j+1.5) at (0.5,0){{\color{purple}$\dots$}};	
\node[V,color=purple] (1) at (0,0){};
\node[draw=none] (1.5L) at (0.5,-0.5){{\color{purple}{\scriptsize$\geq 0$}}};	
\node[draw=none] (4.5L) at (4.5,-0.5){{\color{blue}{\scriptsize$\geq 1$}}};
\node[draw=none] (7.5L) at (7.5,-0.5){{\color{teal}{\scriptsize$\geq 2$}}};
\node[draw=none] (10.5L) at (10.5,-0.5){{\color{red}{\scriptsize$\geq 0$}}};
\draw (i-1) to [bend right] (k) (j-1) to [bend right] (j);
\end{tikzpicture}\end{matrix}, \]
where $i':=n-j+1$, $j':=n-k$, and $k':=n-i+2$. Now, for this diagram, translating the conditions for the number of different coloured dots into inequalities in terms of $i'$, $j'$, and $k'$, one can check that such inequalities agree with those which define the Type (7$^{\ast}$) involutions (after adding a prime to $i$, $j$, and $k$). Similar considerations hold for types $T\in\{2,3,4,5,6\}$.

Hence, by \Cref{Eq:2.12:5,Eq:2.13:3}, answering Kostant's problem and the Indecomposability Conjecture for involutions of Type $(T)$ is equivalent to doing the same for Type $(T^{\ast})$. Thus, it suffices to deal with types (1), (2), (3), (4$^{\ast}$), (5$^{\ast}$), (6$^{\ast}$), and (7$^{\ast}$). We tackle these types one at a time in the following subsections. We will then end with a final subsection which collects/summarises all the results together and proves the Asymptotic Shape Conjecture \ref{Conj:2.13:5} for the shape $(2,1)$.

\subsection{Involutions of Type (1)}\label{Sec:5.1}

Fix $n\geq 5$ and $i\in[3,n-2]$, and let $z:=z_{i,i+2,i+1}^{n}$ be the involution in $\mathrm{Inv}_{n}$ of Type (1). Therefore it satisfies \Cref{Type1}, which we recall here:
\begin{equation*}
\tilde{\mathtt{RS}}_{n}(z)={\begin{ytableau}\scriptstyle i&\scriptstyle i+2\\\scriptstyle i+1\end{ytableau}}^{\ \langle n-3\rangle} \hspace{4mm} \text{ and } \hspace{4mm} z=(i-2, i+1)(i-1, i+2).
\end{equation*}
One can deduce $z=s_{i-1}s_{i-2}s_{i}s_{i-1}s_{i-2}s_{i+1}s_{i}s_{i-1}$. We seek to show that $z$ satisfies the Indecomposability Conjecture \ref{Conj:2.12:1} and is Kostant positive. To help with this, we first discuss which elements $x\in\mathrm{S}_{n}$ have a chance at satisfying the relation $x\leq_{R}^{(n)}z$. Assuming $x\leq_{R}^{(n)}z$, then by the implications in Equations~\eqref{Eq:2.9:4} and \eqref{Eq:2.9:4.5}, we have that $\mathrm{Des}_{L}(x)\subseteq\mathrm{Des}_{L}(z)=\{s_{i-1},s_{i}\}$ and $\mathtt{sh}(z)=(n-3,2,1)\preceq\mathtt{sh}(x)$, with equality of shapes further implying that $x\sim_{L}^{(n)}z$. Therefore, by \Cref{Lem:2.4:1}, this implies that $\mathtt{P}_{x}$ (the insertion tableau of $x$) belongs to the following list (the {\color{blue}blue} tableau only exists if $i\geq 4$ and $n\geq 6$, while the {\color{red}red} tableau only exists if $i\leq n-3$ and $n\geq 6$):
\begin{align}
\emptyset^{(n)} \hspace{6mm}
&{\begin{ytableau}\scriptstyle i\end{ytableau}}^{\ \langle n-1\rangle} \hspace{6mm}
{\begin{ytableau}\scriptstyle i+1\end{ytableau}}^{\ \langle n-1\rangle} \hspace{6mm}
{\begin{ytableau}\scriptstyle i&\scriptstyle i+1\end{ytableau}}^{\ \langle n-2\rangle} \hspace{6mm}
{\begin{ytableau}\scriptstyle i+1&\scriptstyle i+2\end{ytableau}}^{\ \langle n-2\rangle} \hspace{6mm}
{\begin{ytableau}\scriptstyle i\\\scriptstyle i+1\end{ytableau}}^{\ \langle n-2\rangle} \label{Type1RList} \\
&{\color{blue}{\begin{ytableau}\scriptstyle i&\scriptstyle i+1&\scriptstyle i+2\end{ytableau}}^{\ \langle n-3\rangle}} \hspace{6mm}
{\color{red}{\begin{ytableau}\scriptstyle i+1&\scriptstyle i+2&\scriptstyle i+3\end{ytableau}}^{\ \langle n-3\rangle}} \hspace{6mm}
{\begin{ytableau}\scriptstyle i&\scriptstyle i+2\\\scriptstyle i+1\end{ytableau}}^{\ \langle n-3\rangle}. \nonumber
\end{align}

\begin{lem}\label{Type1IC}
The Type (1) involution $z$ given above satisfies the Indecomposability Conjecture \ref{Conj:2.12:1}. That is to say, we have that $\mathbf{KM}_{n}(\star,z)=\mathtt{true}$.
\end{lem}

\begin{proof}
By Equations~\eqref{Eq:2.11:1} and \eqref{Eq:2.12:2}, it suffices to show $\mathbf{KM}_{n}(x,z)=\mathtt{true}$ for all $x\in\mathrm{Inv}_{n}$ such that $x\leq_{R}^{(n)}z$. This relation implies that $\mathtt{P}_{x}$ belongs to the list \eqref{Type1RList}, and one can check that any involution with such a $\mathtt{P}_{x}$ has support set $\mathrm{Sup}(x)$ of size no greater than $6$. For example, if $\mathtt{P}_{x}$ is the {\color{blue}blue} tableaux above (assuming $i\geq 4$ and $n\geq 6$), then we have that $x=s_{i-1}s_{i-2}s_{i-3}s_{i}s_{i-1}s_{i-2}s_{i+1}s_{i}s_{i-1}$, thus $|\mathrm{Sup}(x)|=5$. By the implication in Equation~\eqref{Eq:2.12:4}, the lemma holds.
\end{proof}

\begin{prop}\label{Type1K}
The Type (1) involution $z$ given above is Kostant positive, that is $\mathbf{K}_{n}(z)=\mathtt{true}$.
\end{prop}

\begin{proof}
By \Cref{Thm:2.13:1} and \Cref{Type1IC}, it suffices to prove that we have a non-isomorphism
\begin{equation}\label{Eq:Type1K-1}
\theta_{x}^{(n)}L_{z}^{(n)}\not\cong\theta_{y}^{(n)}L_{z}^{(n)} \hspace{1mm} \text{ in } \hspace{1mm} {^\mathbb{Z}}\mathcal{O}_{0}^{(n)},
\end{equation}
for all distinct pairs $x,y\in\mathrm{S}_{n}$ such that $x,y\leq_{R}^{(n)}z$. For the sake of contradiction, assume that there exists a pair $x,y\leq_{R}^{(n)}z$ such that $\theta_{x}^{(n)}L_{z}^{(n)}\cong\theta_{y}^{(n)}L_{z}^{(n)}$. Then since $x,y\leq_{R}^{(n)}z$, both $\mathtt{P}_{x}$ and $\mathtt{P}_{y}$ belong to the list in Equation~\eqref{Type1RList} above. Moreover, by the implication in Equation~\eqref{Eq:2.9:6} we have $x\sim_{L}^{(n)}y$, which implies $\mathtt{Q}_{x}=\mathtt{Q}_{y}$, and, in particular, $\mathtt{sh}(x)=\mathtt{sh}(y)$. Also, since $x$ and $y$ are distinct, this implies that $\mathtt{P}_{x}\neq\mathtt{P}_{y}$. Hence altogether, both $\mathtt{P}_{x}$ and $\mathtt{P}_{y}$ belong to the list  in Equation~\eqref{Type1RList}, share the same shape, but are distinct. Up to symmetry in $x$ and $y$, this implies that $(\mathtt{P}_{x},\mathtt{P}_{y})$ is one of the following: 
\[ \left({\begin{ytableau}\scriptstyle i\end{ytableau}}^{\ \langle n-1\rangle}, \hspace{2mm} {\begin{ytableau}\scriptstyle i+1\end{ytableau}}^{\ \langle n-1\rangle} \right), \hspace{6mm} \left({\begin{ytableau}\scriptstyle i&\scriptstyle i+1\end{ytableau}}^{\ \langle n-2\rangle}, \hspace{2mm} {\begin{ytableau}\scriptstyle i+1&\scriptstyle i+2\end{ytableau}}^{\ \langle n-2\rangle}\right), \\ \text{ or} \]
\[ \left({\color{blue}{\begin{ytableau}\scriptstyle i&\scriptstyle i+1&\scriptstyle i+2\end{ytableau}}^{\ \langle n-3\rangle}}, \hspace{2mm} {\color{red}{\begin{ytableau}\scriptstyle i+1&\scriptstyle i+2&\scriptstyle i+3\end{ytableau}}^{\ \langle n-3\rangle}}\right). \]
Moreover, by the equivalence  in Equation~\eqref{Eq:2.11:3}, it suffices to consider any particular pair of elements $(x',y')$ such that $(\mathtt{P}_{x'},\mathtt{P}_{y'})$ is one of the above three pairings of right cells, and $x'\sim_{L}^{(n)}y'$. Thus, we must have that the isomorphism $\theta_{x}^{(n)}L_{z}^{(n)}\cong\theta_{y}^{(n)}L_{z}^{(n)}$ holds for $(x,y)$ being one of the following three pairs:
\begin{equation}\label{Eq:Type1K-2}
(s_{i-1}, s_{i}s_{i-1}), \hspace{2mm} (s_{i-1}s_{i-2}s_{i}, s_{i}s_{i-1}s_{i-2}s_{i+1}s_{i}), \\ \text{ or} 
\end{equation}
\[ (s_{i-1}s_{i-2}s_{i-3}s_{i}s_{i-1}s_{i-2}s_{i+1}, s_{i}s_{i-1}s_{i-2}s_{i-3}s_{i+1}s_{i}s_{i-1}s_{i-2}s_{i+2}s_{i+1}). \]
These pairings of elements correspond, respectively, to the three pairings of right cells above. In particular, the last pair only exists when $4\leq i\leq n-3$. Therefore, by decategorifying the isomorphism $\theta_{x}^{(n)}L_{z}^{(n)}\cong\theta_{y}^{(n)}L_{z}^{(n)}$, we must have that one of the following equalities in $\mathrm{H}_{n}^{\mathbb{A}}$ holds:
\begin{align*}
D_{z}^{(n)}C_{s_{i-1}}^{(n)}=D_{z}^{(n)}C_{s_{i}s_{i-1}}^{(n)},& \hspace{5mm} D_{z}^{(n)}C_{s_{i-1}s_{i-2}s_{i}}^{(n)}=D_{z}^{(n)}C_{s_{i}s_{i-1}s_{i-2}s_{i+1}s_{i}}^{(n)}, \hspace{2mm} \text{ or} \\
D_{z}^{(n)}C_{s_{i-1}s_{i-2}s_{i-3}s_{i}s_{i-1}s_{i-2}s_{i+1}}^{(n)}&=D_{z}^{(n)}C_{s_{i}s_{i-1}s_{i-2}s_{i-3}s_{i+1}s_{i}s_{i-1}s_{i-2}s_{i+2}s_{i+1}}^{(n)}.
\end{align*}
Lastly, by \Cref{Lem:3:18}, for appropriate shifts (and using the notational short-hand $\mathsf{i}:=s_{i}$), we must have that one of the following equalities holds:
\begin{align}
D_{\mathsf{21321432}}^{(5)}C_{\mathsf{2}}^{(5)}=D_{\mathsf{21321432}}^{(5)}C_{\mathsf{32}}^{(5)},& \hspace{5mm} D_{\mathsf{21321432}}^{(5)}C_{\mathsf{213}}^{(5)}=D_{\mathsf{21321432}}^{(5)}C_{\mathsf{32143}}^{(5)}, \label{Type1K-3} \\
D_{\mathsf{32432543}}^{(7)}C_{\mathsf{3214325}}^{(7)}&=D_{\mathsf{32432543}}^{(7)}C_{\mathsf{4321543265}}^{(7)} \nonumber.
\end{align}
The first two equalities are taking place in $\mathrm{H}_{5}^{\mathbb{A}}$, where we have employed the shift $F_{5,n}^{i-2}$, in particular
\[ F_{5,n}^{i-2}(s_{2}s_{1}s_{3}s_{2}s_{1}s_{4}s_{3}s_{2})=F_{5,n}^{i-2}(\mathsf{21321432})=z, \]
while the latter equality is taking place in $\mathrm{H}_{7}^{\mathbb{A}}$, where  we have employed the shift $F_{7,n}^{i-3}$, in particular
\[ F_{7,n}^{i-3}(s_{3}s_{2}s_{4}s_{3}s_{2}s_{5}s_{4}s_{3})=F_{7,n}^{i-3}(\mathsf{32432543})=z. \]
One can now confirm computationally that neither of the three equalities in \eqref{Type1K-3} hold, which gives the desired contradiction, completing the proof of the proposition. It is worth remarking that the two equalities in $\mathrm{H}_{5}^{\mathbb{A}}$ are already known to not hold, since it is known that the involution $s_{2}s_{1}s_{3}s_{2}s_{1}s_{4}s_{3}s_{2}$ in $\mathrm{S}_{5}$ is Kostant positive and satisfies K{\aa}hrstr{\"o}m's conjecture.
\end{proof}

\subsection{Involutions of Type $(2)$}\label{Sec:5.2}

Fix $n\geq 5$ and $i\in[2,n-3]$, and let $z:=z_{i,i+3,i+2}^{n}$ be the involution in $\mathrm{Inv}_{n}$ of Type (2). Therefore, it satisfies \Cref{Type2}, which we recall here:
\begin{equation*}
\tilde{\mathtt{RS}}_{n}(z)={\begin{ytableau}\scriptstyle i&\scriptstyle i+3\\\scriptstyle i+2\end{ytableau}}^{\ \langle n-3\rangle} \hspace{4mm} \text{ and } \hspace{4mm} z=(i-1, i+2)(i+1,i+3).
\end{equation*}
One can deduce that $z=s_{i-1}s_{i+1}s_{i}s_{i-1}s_{i+2}s_{i+1}$, and in one-line notation we have
\begin{equation}\label{Eq:5.2:1}
z=1\cdots(i-2){\color{teal}(i+2)}i{\color{teal}(i+3)(i-1)(i+1)}(i+4)\cdots n,
\end{equation}
where we have coloured the non-fixed points in {\color{teal}teal}. We first seek to understand which $x\in\mathrm{S}_{n}$ have a chance at satisfying the relation $x\leq_{R}^{(n)}z$. Assuming $x\leq_{R}^{(n)}z$, then by the implications in Equations~\eqref{Eq:2.9:4} and \eqref{Eq:2.9:4.5}, we have that both $\mathrm{Des}_{L}(x)\subseteq\mathrm{Des}_{L}(z)=\{s_{i-1},s_{i+1}\}$ as well as $\mathtt{sh}(z)=(n-3,2,1)\preceq\mathtt{sh}(x)$, with equality of shapes further implying that $x\sim_{L}^{(n)}z$. Therefore, by \Cref{Lem:2.4:1}, this implies that $\mathtt{P}_{x}$ (the insertion tableau of $x$) belongs to the following list:
\begin{equation*}
\emptyset^{\langle n\rangle}
\end{equation*}
\[
{\begin{ytableau}\scriptstyle i\end{ytableau}}^{\ \langle n-1\rangle} \hspace{6mm}
{\begin{ytableau}\scriptstyle i+2\end{ytableau}}^{\ \langle n-1\rangle}
\]
\[
{\begin{ytableau}\scriptstyle i&\scriptstyle i+2\end{ytableau}}^{\ \langle n-2\rangle} \hspace{6mm}
{\color{red}{\begin{ytableau}\scriptstyle i&\scriptstyle i+1\end{ytableau}}^{\ \langle n-2\rangle}} \hspace{6mm}
{\begin{ytableau}\scriptstyle i+2&\scriptstyle i+3\end{ytableau}}^{\ \langle n-2\rangle}
\]
\[
{\begin{ytableau}\scriptstyle i\\\scriptstyle i+2\end{ytableau}}^{\ \langle n-2\rangle}
\]
\[
{\color{magenta}{\begin{ytableau}\scriptstyle i&\scriptstyle i+2&\scriptstyle i+3\end{ytableau}}^{\ \langle n-3\rangle}} \hspace{6mm}
{\color{red}{\begin{ytableau}\scriptstyle i&\scriptstyle i+1&\scriptstyle i+2\end{ytableau}}^{\ \langle n-3\rangle}} \hspace{6mm}
{\color{blue}{\begin{ytableau}\scriptstyle i+2&\scriptstyle i+3&\scriptstyle i+4\end{ytableau}}^{\ \langle n-3\rangle}} 
\]
\[
{\begin{ytableau}\scriptstyle i&\scriptstyle i+3\\\scriptstyle i+2\end{ytableau}}^{\ \langle n-4\rangle}
\]
Here the fifth row only exists when $n\geq 6$, with the {\color{blue}blue} tableau also requiring $i<n-3$, while the {\color{magenta}magenta} tableau requires $i>2$ instead. Furthermore, the {\color{red}red} tableau in the third row requires that $i>2$, while the {\color{red}red} tableau in the fifth row requires that $i>3$. However, we now show that the {\color{red}red} tableaux can be dropped from the above collection.

\begin{lem}\label{Lem:5.2:1}
Let $x\in\mathrm{S}_{n}$ with $\mathtt{P}_{x}$ a {\color{red}red} tableaux above, then $\theta_{x}^{(n)}L_{z}^{(n)}=0$, equivalently $x\not\leq_{R}^{(n)}z$.
\end{lem}

\begin{proof}
It suffices to prove $x\not\leq_{R}^{(n)}z$ with $x$ an involution with $\mathtt{P}_{x}$ either of the {\color{red}red} tableaux above (since each right cell contains a unique involution). Moreover, by the equivalence in Equation~\eqref{Eq:2.9:2}, this is equivalent to proving $x\not\leq_{L}^{(n)}z$. This non-relation can be confirmed by employing Claim~(b) of \Cref{Prop:3:7}. We will demonstrate this for the second {\color{red}red} tableau above, with the other following by similar arguments.

Therefore, assume that $x$ is the involution in $\mathrm{Inv}_{n}$ such that $\mathtt{P}_{x}={\color{red}{\begin{ytableau}\scriptstyle i&\scriptstyle i+1&\scriptstyle i+2\end{ytableau}}^{\ (n-3)}}$, which implies
\[ x=s_{i-1}s_{i-2}s_{i-3}s_{i}s_{i-1}s_{i-2}s_{i+1}s_{i}s_{i-1}=F_{6,n}^{i-3}(s_{3}s_{2}s_{1}s_{4}s_{3}s_{2}s_{5}s_{4}s_{3}). \]
Also, by \Cref{Eq:5.2:1}, one can deduce that $z$ consecutively contains the pattern
\[ p:=125463=s_{3}s_{4}s_{5}s_{3}\in\mathrm{S}_{6} \]
at position $i-3$. Hence, from Claim~(b) of \Cref{Prop:3:7}, we have that
\begin{align*}
x\not\leq_{L}^{(n)}z &\iff F_{6,n}^{i-3}(s_{3}s_{2}s_{1}s_{4}s_{3}s_{2}s_{5}s_{4}s_{3})\not\leq_{L}^{(n)}z \\
&\iff y:=s_{3}s_{2}s_{1}s_{4}s_{3}s_{2}s_{5}s_{4}s_{3}\not\leq_{L}^{(6)}p=s_{3}s_{4}s_{5}s_{3}.
\end{align*}
By Equations~\eqref{Eq:2.9:5} and \eqref{Eq:2.9:2}, the last non-relation is confirmed computationally by checking that $D_{p}^{(6)}C_{y}^{(6)}=0$. As mentioned, proving $x\not\leq_{R}^{(n)}z$ for $\mathtt{P}_{x}$ the other {\color{red}red} tableau follows from a similar argument.
\end{proof}

Therefore, if $x\leq_{R}^{(n)}z$, then the insertion tableau $\mathtt{P}_{x}$ must belong to the following list:
\begin{equation}\label{Type2RList}
\emptyset^{\langle n\rangle} \hspace{6mm}
{\begin{ytableau}\scriptstyle i\end{ytableau}}^{\ \langle n-1\rangle} \hspace{6mm}
{\begin{ytableau}\scriptstyle i+2\end{ytableau}}^{\ \langle n-1\rangle} \hspace{6mm}
{\begin{ytableau}\scriptstyle i&\scriptstyle i+2\end{ytableau}}^{\ \langle n-2\rangle} \hspace{6mm}
{\begin{ytableau}\scriptstyle i+2&\scriptstyle i+3\end{ytableau}}^{\ \langle n-2\rangle}
\end{equation}
\[
{\begin{ytableau}\scriptstyle i\\\scriptstyle i+2\end{ytableau}}^{\ \langle n-2\rangle} \hspace{6mm}
{\color{magenta}{\begin{ytableau}\scriptstyle i&\scriptstyle i+2&\scriptstyle i+3\end{ytableau}}^{\ \langle n-3\rangle}} \hspace{6mm}
{\color{blue}{\begin{ytableau}\scriptstyle i+2&\scriptstyle i+3&\scriptstyle i+4\end{ytableau}}^{\ \langle n-3\rangle}} \hspace{6mm}
{\begin{ytableau}\scriptstyle i&\scriptstyle i+3\\\scriptstyle i+2\end{ytableau}}^{\ \langle n-4\rangle}
\]

\begin{lem}\label{Type2IC}
The Type (2) involution $z$ given above satisfies the Indecomposability Conjecture \ref{Conj:2.12:1}. That is to say, we have that $\mathbf{KM}_{n}(\star,z)=\mathtt{true}$.
\end{lem}

\begin{proof}
By Equations~\eqref{Eq:2.11:1} and \eqref{Eq:2.12:2}, it suffices to show that $\mathbf{KM}_{n}(x,z)=\mathtt{true}$ for all $x\in\mathrm{Inv}_{n}$ where $x\leq_{R}^{(n)}z$. This relation implies that $\mathtt{P}_{x}$ belongs to the list \eqref{Type2RList}, and one can check that any involution with such a $\mathtt{P}_{x}$ has support set $\mathrm{Sup}(x)$ of size no greater than $6$. For example, if $\mathtt{P}_{x}$ is the {\color{magenta}magenta} tableaux above (assuming $i>2$ and $n\geq 6$), then $x=s_{i-1}s_{i+1}s_{i}s_{i+2}s_{i+1}$. So in this example we have $|\mathrm{Sup}(x)|=4$. Therefore, by the implication in Equation~\eqref{Eq:2.12:4}, this lemma holds.
\end{proof}

\begin{prop}\label{Type2K}
The answer to Kostant's problem for the Type (2) involution $z$ above is given by
\[ \mathbf{K}_{n}(z)=
\begin{cases} 
\mathtt{true}, &  i=2 \text{ and/or } i=n-3\\
\mathsf{false}, & \text{otherwise}
\end{cases}.
\]
\end{prop}

\begin{proof}
Let us first prove that $\mathbf{K}_{n}(z)=\mathtt{true}$ when $i=2$. In this case we have
\[ z=(1,4)(3,5)=s_{1}s_{3}s_{2}s_{1}s_{4}s_{3}\in\mathrm{S}_{n}. \]
Consider the parabolic subgroup $\mathrm{S}_{n-2}\subset\mathrm{S}_{n}$ via the natural inclusion. Then one can deduce that
\[ z\sim_{L}^{(n)}(s_{2}s_{1}s_{3}s_{2})w_{0,n-2}w_{0,n}. \]
Therefore, by \cite[Theorem 1.1]{K10}, we have that $\mathbf{K}_{n}(z)=\mathbf{K}_{n-2}(s_{2}s_{1}s_{3}s_{2})$. The involution $s_{2}s_{1}s_{3}s_{2}$ is fully commutative, and so by \cite[Theorem 5.1]{MMM24} we see that $\mathbf{K}_{n-2}(s_{2}s_{1}s_{3}s_{2})=\mathtt{true}$, which implies that $\mathbf{K}_{n}(z)=\mathtt{true}$, completing the $i=2$ case. 

Let us now assume $i=n-3$, then we also seek to show that $\mathbf{K}_{n}(z)=\mathtt{true}$. Firstly, we have that
\[ z=(n-4,n-1)(n-2,n)=s_{n-4}s_{n-2}s_{n-3}s_{n-4}s_{n-1}s_{n-2}\in\mathrm{S}_{n}. \]
By \Cref{Thm:2.13:1} and \Cref{Type2IC}, it suffices to prove that we have a non-isomorphism
\begin{equation}\label{Eq:Type2K-1}
\theta_{x}^{(n)}L_{z}^{(n)}\not\cong\theta_{y}^{(n)}L_{z}^{(n)} \hspace{1mm} \text{ in } \hspace{1mm} {^\mathbb{Z}}\mathcal{O}_{0}^{(n)},
\end{equation}
for all distinct pairs $x,y\in\mathrm{S}_{n}$ such that $x,y\leq_{R}^{(n)}z$. For the sake of contradiction, assume there exists a pair $x,y\leq_{R}^{(n)}z$ such that $\theta_{x}^{(n)}L_{z}^{(n)}\cong\theta_{y}^{(n)}L_{z}^{(n)}$. Then, since $x,y\leq_{R}^{(n)}z$, both $\mathtt{P}_{x}$ and $\mathtt{P}_{y}$ belong to the list in Equation~\eqref{Type2RList} above, in the case where $i=n-3$. Note, in this case, the {\color{blue}blue} tableau does not exist since it required $i<n-3$. For clarity, both $\mathtt{P}_{x}$ and $\mathtt{P}_{y}$ belong to the following list:
\begin{equation}\label{Eq:Type2k-2}
\emptyset^{(n)} \hspace{6mm}
{\begin{ytableau}\scriptstyle n-3\end{ytableau}}^{\ (n-1)} \hspace{6mm}
{\begin{ytableau}\scriptstyle n-1\end{ytableau}}^{\ (n-1)} \hspace{6mm}
{\begin{ytableau}\scriptstyle n-3&\scriptstyle n-1\end{ytableau}}^{\ (n-2)} \hspace{6mm}
{\begin{ytableau}\scriptstyle n-1&\scriptstyle n\end{ytableau}}^{\ (n-2)}
\end{equation}
\[
{\begin{ytableau}\scriptstyle n-3\\\scriptstyle n-1\end{ytableau}}^{\ (n-2)} \hspace{6mm}
{\color{magenta}{\begin{ytableau}\scriptstyle n-3&\scriptstyle n-1&\scriptstyle n\end{ytableau}}^{\ (n-3)}} \hspace{6mm}
{\begin{ytableau}\scriptstyle n-3&\scriptstyle n\\\scriptstyle n-1\end{ytableau}}^{\ (n-4)}
\]
Moreover, by the implication in Equation~\eqref{Eq:2.9:6}, $x\sim_{L}^{(n)}y$, which implies $\mathtt{Q}_{x}=\mathtt{Q}_{y}$, and, in particular, $\mathtt{sh}(x)=\mathtt{sh}(y)$. Also, since $x$ and $y$ are distinct, this implies that $\mathtt{P}_{x}\neq\mathtt{P}_{y}$. Hence altogether, both $\mathtt{P}_{x}$ and $\mathtt{P}_{y}$ belong to the list  in Equation~\eqref{Eq:Type2k-2}, share the same shape, and are distinct. Up to symmetry in $x$ and $y$, this implies
\[ (\mathtt{P}_{x},\mathtt{P}_{y})=\left({\begin{ytableau}\scriptstyle n-3\end{ytableau}}^{\ (n-1)}, \hspace{2mm} {\begin{ytableau}\scriptstyle n-1\end{ytableau}}^{\ (n-1)} \right) \hspace{2mm} \text{ or } \hspace{2mm}  (\mathtt{P}_{x},\mathtt{P}_{y})=\left({\begin{ytableau}\scriptstyle n-3&\scriptstyle n-1\end{ytableau}}^{\ (n-2)}, \hspace{2mm} {\begin{ytableau}\scriptstyle n-1&\scriptstyle n\end{ytableau}}^{\ (n-2)}\right). \]
Furthermore, by the equivalence  in Equation~\eqref{Eq:2.11:3}, it suffices to consider any particular pair of elements $(x',y')$ such that $(\mathtt{P}_{x'},\mathtt{P}_{y'})$ is one of the above two pairings of insertion tableau, and $x'\sim_{L}^{(n)}y'$. Thus, we must have that the isomorphism $\theta_{x}^{(n)}L_{z}^{(n)}\cong\theta_{y}^{(n)}L_{z}^{(n)}$ holds for $(x,y)$ being one of the following two pairs:
\begin{equation*}
(s_{n-4}, s_{n-2}s_{n-3}s_{n-4}) \hspace{2mm} \text{ or } \hspace{2mm} (s_{n-4}s_{n-2}, s_{n-2}s_{n-3}s_{n-4}s_{n-1}s_{n-2}).
\end{equation*}
Decategorifying the corresponding isomorphism $\theta_{x}^{(n)}L_{z}^{(n)}\cong\theta_{y}^{(n)}L_{z}^{(n)}$, we must have that either
\begin{equation*}
D_{z}^{(n)}C_{s_{n-4}}^{(n)}=D_{z}^{(n)}C_{s_{n-2}s_{n-3}s_{n-4}}^{(n)} \hspace{2mm} \text{ or } \hspace{2mm} D_{z}^{(n)}C_{s_{n-4}s_{n-2}}^{(n)}=D_{z}C_{s_{n-2}s_{n-3}s_{n-4}s_{n-1}s_{n-2}}^{(n)},
\end{equation*}
holds in $\mathrm{H}_{n}^{\mathbb{A}}$. Lastly, by employing \Cref{Lem:3:18} for an appropriate shift, we must have that either 
\begin{equation}\label{Type2K-3}
D_{s_{1}s_{3}s_{2}s_{1}s_{4}s_{3}}^{(5)}C_{s_{1}}^{(5)}=D_{s_{1}s_{3}s_{2}s_{1}s_{4}s_{3}}^{(5)}C_{s_{3}s_{2}s_{1}}^{(5)} \hspace{2mm} \text{ or } \hspace{2mm} D_{s_{1}s_{3}s_{2}s_{1}s_{4}s_{3}}^{(5)}C_{s_{1}s_{3}}^{(5)}=D_{s_{1}s_{3}s_{2}s_{1}s_{4}s_{3}}^{(5)}C_{s_{3}s_{2}s_{1}s_{4}s_{3}}^{(5)}. 
\end{equation}
These two equalities are taking place in $\mathrm{H}_{5}^{\mathbb{A}}$, where we have employed the shift $F_{5,n}^{n-4}$, in particular
\[ F_{5,n}^{n-4}(s_{1}s_{3}s_{2}s_{1}s_{4}s_{3})=s_{n-4}s_{n-2}s_{n-3}s_{n-4}s_{n-1}s_{n-2}=z. \]
One can now confirm computationally that neither of the two equalities in \Cref{Type2K-3} hold, which gives the desired contradiction, completing the proof of the case for $i=n-3$. It is worth remarking that these two equalities are already known not to hold, since it is known that the involution $s_{1}s_{3}s_{2}s_{1}s_{4}s_{3}$ in $\mathrm{S}_{5}$ is Kostant positive and satisfies K{\aa}hrstr{\"o}m's conjecture.

It remains is to prove $\mathbf{K}_{n}(z)=\mathtt{false}$ if $2<i<n-3$, so in particular $n\geq 7$. In this case we have
\[ z=(i-1, i+2)(i+1,i+3)=s_{i-1}s_{i+1}s_{i}s_{i-1}s_{i+2}s_{i+1}. \]
Consider the pair of permutations
\[ x=s_{i-1}s_{i-2}s_{i+1}s_{i}s_{i-1}s_{i+2}s_{i+1}s_{i+3} \hspace{2mm} \text{ and } \hspace{2mm} y=s_{i+1}s_{i}s_{i-1}s_{i+2}s_{i+1}s_{i+3}. \]
By \Cref{Thm:2.13:1}, it suffices to show that we have the isomorphism of non-zero modules
\begin{equation}\label{Eq:Type2K-4}
\theta_{x}^{(n)}L_{z}^{(n)}\cong\theta_{y}^{(n)}L_{z}^{(n)}\neq0.
\end{equation}
Firstly, one can deduce that $C_{x}^{(n)}=C_{s_{i-1}s_{i-2}}^{(n)}C_{y}^{(n)}=C_{s_{i-1}}^{(n)}C_{s_{i-2}}^{(n)}C_{y}^{(n)}$ by checking a corresponding relation in $\mathrm{H}_{7}^{\mathbb{A}}$ computationally, then shifting to $\mathrm{H}_{n}^{\mathbb{A}}$. This implies
\begin{equation}\label{Eq:Type2K-5}
\theta_{x}^{(n)}\cong\theta_{y}^{(n)}\theta_{s_{i-1}s_{i-2}}^{(n)}\cong\theta_{y}^{(n)}\theta_{s_{i-2}}^{(n)}\theta_{s_{i-1}}^{(n)}.
\end{equation}
Therefore, we will begin by focusing on understanding the following module:
\[ \theta_{s_{i-1}s_{i-2}}^{(n)}L_{z}^{(n)}\cong\theta_{s_{i-2}}^{(n)}\theta_{s_{i-1}}^{(n)}L_{z}^{(n)}. \]
Since $zs_{i-1}<z$, then, by \Cref{Prop:2.10:2}, the module $\theta_{s_{i-1}}^{(n)}L_{z}^{(n)}$ has simple top and socle isomorphic to $L_{z}^{(n)}\langle\pm1\rangle$, and semi-simple Jantzen middle $\mathrm{J}_{s_{i-1}}^{(n)}(z)$ at degree 0. Moreover, since $zs_{i-2}>z$, we have that $\theta_{s_{i-2}}^{(n)}L_{z}^{(n)}=0$. Therefore, we must have the isomorphisms
\begin{equation*}
\theta_{s_{i-1}s_{i-2}}^{(n)}L_{z}^{(n)}\cong\theta_{s_{i-2}}^{(n)}\theta_{s_{i-1}}^{(n)}L_{z}^{(n)}\cong\theta_{s_{i-2}}^{(n)}\mathrm{J}_{s_{i-1}}^{(n)}(z).
\end{equation*}
The module $L_{zs_{i-2}}^{(n)}$ appears within $\mathrm{J}_{s_{i-1}}^{(n)}(z)$ since $(zs_{i-2})s_{i-1}>zs_{i-2}$ and $\mu(z,zs_{i-2})=1$, and it is not killed by $\theta_{s_{i-2}}^{(n)}$ since $(zs_{i-2})s_{i-2}<zs_{i-2}$. Furthermore, the module $ \theta_{s_{i-1}s_{i-2}}^{(n)}L_{z}^{(n)}$, and thus also $\theta_{s_{i-2}}^{(n)}\mathrm{J}_{s_{i-1}}^{(n)}(z)$, is indecomposable by the implication  in Equation~\eqref{Eq:2.12:4}. This implies that $L_{zs_{i-2}}^{(n)}$ is the only submodule of $\mathrm{J}_{s_{i-1}}^{(n)}(z)$ not killed by $\theta_{s_{i-2}}^{(n)}$, which means that we have the isomorphism
\begin{equation}\label{Eq:Type2K-6}
\theta_{s_{i-1}s_{i-2}}^{(n)}L_{z}^{(n)}\cong\theta_{s_{i-2}}^{(n)}\mathrm{J}_{s_{i-1}}^{(n)}(z)\cong\theta_{s_{i-2}}^{(n)}L_{zs_{i-2}}^{(n)}.
\end{equation}
Since $(zs_{i-2})s_{i-2}<zs_{i-2}$, then, by \Cref{Prop:2.10:2}, the module $\theta_{s_{i-2}}^{(n)}L_{zs_{i-2}}^{(n)}$ has simple top and socle isomorphic to $L_{zs_{i-2}}^{(n)}\langle\pm1\rangle$, and semi-simple Jantzen middle $\mathrm{J}_{s_{i-2}}^{(n)}(zs_{i-2})$ at degree 0. Also, since $z=s_{i-1}s_{i+1}s_{i}s_{i-1}s_{i+2}s_{i+1}$ and $y^{-1}=s_{i+3}s_{i+1}s_{i+2}s_{i-1}s_{i}s_{i+1}$, we see that
\begin{align*}
\theta_{y}^{(n)}L_{z_{s_{i-2}}}^{(n)}=0 &\iff y^{-1}\not\leq_{L}^{(n)}zs_{i-2}, \\
&\iff F_{6,n}^{i-2}(s_{6}s_{4}s_{5}s_{2}s_{3}s_{4})\not\leq_{L}^{(n)}F_{6,n}^{i-2}(s_{2}s_{4}s_{3}s_{2}s_{5}s_{4}s_{1}), \\
&\iff s_{6}s_{4}s_{5}s_{2}s_{3}s_{4}\not\leq_{L}^{(7)}s_{2}s_{4}s_{3}s_{2}s_{5}s_{4}s_{1}, \\
&\iff D_{s_{2}s_{4}s_{3}s_{2}s_{5}s_{4}s_{1}}^{(7)}C_{s_{4}s_{3}s_{2}s_{5}s_{4}s_{6}}^{(7)}=0.
\end{align*}
The first equivalence follows from \Cref{Eq:2.11:1}, the second by the definition of the shift, the third from Claim~(b) of \Cref{Prop:3:7}, and the last from \Cref{Eq:2.9:5}. Moreover, the last equality can be confirmed computationally, and therefore we have that
\begin{equation}\label{Eq:Type2K-7} 
\theta_{y}^{(n)}L_{z_{s_{i-2}}}^{(n)}=0.
\end{equation}
Thus, altogether, we have the isomorphisms
\begin{equation}\label{Eq:Typle2K-8}
\theta_{x}^{(n)}L_{z}^{(n)}\cong\theta_{y}^{(n)}\theta_{s_{i-1}s_{i-2}}^{(n)}L_{z}^{(n)}\cong\theta_{y}^{(n)}\theta_{s_{i-2}}^{(n)}L_{zs_{i-2}}^{(n)}\cong\theta_{y}^{(n)}\mathrm{J}_{s_{i-2}}^{(n)}(zs_{i-2}).
\end{equation}
The first isomorphism follows from \Cref{Eq:Type2K-5}, the second from \Cref{Eq:Type2K-6}, and the third from \Cref{Eq:Type2K-7} and the sentence proceeding \Cref{Eq:Type2K-6}. Now, as $|\mathsf{Sup}(x)|=6$, we know that $\theta_{x}^{(n)}L_{z}^{(n)}$ is indecomposable. Thus, at most one summand in $\mathrm{J}_{s_{i-2}}^{(n)}(zs_{i-2})$ is not killed by $\theta_{y}^{(n)}$. Note that $L_{z}^{(n)}$ appears in $\mathrm{J}_{s_{i-2}}^{(n)}(zs_{i-2})$ since $zs_{i-2}>z$ and $\mu(zs_{i-2},z)=1$. Hence, if $\theta_{y}^{(n)}L_{z}^{(n)}\neq0$, then from \Cref{Eq:Typle2K-8} we must have that
\[  \theta_{x}^{(n)}L_{z}^{(n)}\cong\theta_{y}^{(n)}\mathrm{J}_{s_{i-2}}^{(n)}(zs_{i-2})\cong\theta_{y}^{(n)}L_{z}^{(n)}\neq0, \]
establishing \Cref{Eq:Type2K-4}, and completing the proof. Thus, it suffices to show that $\theta_{y}^{(n)}L_{z}^{(n)}\neq0$, equivalently $D_{z}^{(n)}C_{y}^{(n)}\neq0$ in $\mathrm{H}_{n}^{\mathbb{A}}$, or equivalently $y^{-1}\leq_{L}z$ in $\mathrm{S}_{n}$. Well, we have that 
\begin{align*}
y^{-1}\leq_{L}^{(n)}z &\iff F_{6,n}^{i-2}(s_{5}s_{3}s_{4}s_{1}s_{2}s_{3})\not\leq_{L}^{(n)}F_{6,n}^{i-2}(s_{1}s_{3}s_{2}s_{1}s_{4}s_{3}), \\
&\iff s_{5}s_{3}s_{4}s_{1}s_{2}s_{3}\leq_{L}^{(6)}s_{1}s_{3}s_{2}s_{1}s_{4}s_{3}, \\
&\iff D_{s_{1}s_{3}s_{2}s_{1}s_{4}s_{3}}^{(6)}C_{s_{3}s_{2}s_{1}s_{4}s_{3}s_{5}}^{(6)}\neq0,
\end{align*}
where the latter inequality can be confirmed computationally, completing the proof. 
\end{proof}

\subsection{Involutions of Type $(3)$}\label{Sec:5.3}

Fix $n\geq 6$, $i\in[2,n-4]$, and $i+2<k<n$, and let $z:=z_{i,k+1,k}^{n}$ be the involution in $\mathrm{Inv}_{n}$ of Type (3). Hence, it satisfies \Cref{Type3}, which we recall here:
\begin{equation*}
\tilde{\mathtt{RS}}_{n}(z)={\begin{ytableau}\scriptstyle i&\scriptstyle k+1\\\scriptstyle k\end{ytableau}}^{\ \langle n-3\rangle} \hspace{4mm} \text{ and } \hspace{4mm} z=(i-1, k)(k-1,k+1).
\end{equation*}
When written in one-line notation, the involution $z$ is given by
\begin{equation}\label{Eq:5.3:1}
z=12\cdots(i-2){\color{teal}k}i(i+1)\cdots(k-2){\color{teal}(k+1)(i-1)(k-1)}(k+2)\cdots (n-1)n,
\end{equation}
where we have coloured the non-fixed points in {\color{teal}teal}. As done with previous types, we first investigate what elements $x\in\mathrm{S}_{n}$ have a chance at satisfying the relation $x\leq_{R}^{(n)}z$. Assuming $x\leq_{R}^{(n)}z$, then by the implications in Equations~\eqref{Eq:2.9:4} and \eqref{Eq:2.9:4.5}, we have $\mathrm{Des}_{L}(x)\subseteq\mathrm{Des}_{L}(z)=\{s_{i-1},s_{k-1}\}$ and $\mathtt{sh}(z)=(n-3,2,1)\preceq\mathtt{sh}(x)$, with equality of shapes further implying that $x\sim_{L}^{(n)}z$. Therefore, by \Cref{Lem:2.4:1}, this implies that $\mathtt{P}_{x}$ (the insertion tableau of $x$) belongs to the following list:
\[ 
\emptyset^{\langle n\rangle}
\]
\[
{\begin{ytableau}\scriptstyle i\end{ytableau}}^{\ \langle n-1\rangle} \hspace{6mm}
{\begin{ytableau}\scriptstyle k\end{ytableau}}^{\ \langle n-1\rangle}
\]
\[
{\begin{ytableau}\scriptstyle i&\scriptstyle k\end{ytableau}}^{\ \langle n-2\rangle} \hspace{6mm}
{\color{red}{\begin{ytableau}\scriptstyle i&\scriptstyle i+1\end{ytableau}}^{\ \langle n-2\rangle}} \hspace{6mm}
{\begin{ytableau}\scriptstyle k&\scriptstyle k+1\end{ytableau}}^{\ \langle n-2\rangle}
\]
\[
{\begin{ytableau}\scriptstyle i\\\scriptstyle k\end{ytableau}}^{\ \langle n-2\rangle}
\]
\[
{\color{red}{\begin{ytableau}\scriptstyle i&\scriptstyle i+1&\scriptstyle k\end{ytableau}}^{\ \langle n-3\rangle}} \hspace{6mm}
{\begin{ytableau}\scriptstyle i&\scriptstyle k&\scriptstyle k+1\end{ytableau}}^{\ \langle n-3\rangle} \hspace{6mm}
{\color{red}{\begin{ytableau}\scriptstyle i&\scriptstyle i+1&\scriptstyle i+2\end{ytableau}}^{\ \langle n-3\rangle}} \hspace{6mm}
{\color{blue}{\begin{ytableau}\scriptstyle k&\scriptstyle k+1&\scriptstyle k+2\end{ytableau}}^{\ \langle n-3\rangle}} 
\]
\[
{\begin{ytableau}\scriptstyle i&\scriptstyle k+1\\\scriptstyle k\end{ytableau}}^{\ \langle n-3\rangle}
\]
The {\color{blue}blue} tableau requires that $k<n-2$. Also, reading from top to bottom, left to right, the first and second {\color{red}red} tableaux require that $i>2$, while the third {\color{red}red} tableau requires that $i>3$.

\begin{lem}\label{Lem:5.3:2}
Let $x\in\mathbf{S}_{n}$ be such that $\mathtt{P}_{x}$ is either the {\color{blue}blue} tableau, or one of the {\color{red}red} tableaux above. Then we have that $x\not\leq_{R}^{(n)}z$, equivalently $\theta_{x}^{(n)}L_{z}^{(n)}=0$.
\end{lem}

\begin{proof}
It suffices to prove $x\not\leq_{R}^{(n)}z$ where $x$ is an involution with $\mathtt{P}_{x}$ either of the {\color{red}red} tableaux above, or the {\color{blue}blue} tableau above (since each right cell contains a unique involution). By the equivalence in Equation~\eqref{Eq:2.9:2}, this is equivalent to proving $x\not\leq_{L}^{(n)}z$. For the first and last {\color{red}red} tableaux, this non-relation can be confirmed by employing Claim~(b) of \Cref{Prop:3:7}, as done in \Cref{Lem:5.2:1}, while for the second {\color{red}red} tableau, this can be confirmed by \Cref{Cor:3:15}. We will demonstrate this argument for the second {\color{red}red} tableau above, with the other two following by similar arguments to that in \Cref{Lem:5.2:1}.

Therefore, assume that $x$ is the involution in $\mathrm{Inv}_{n}$ such that $\mathtt{P}_{x}={\color{red}{\begin{ytableau}\scriptstyle i&\scriptstyle i+1&\scriptstyle k\end{ytableau}}^{\ (n-3)}}$, which implies
\[ x=s_{i-1}s_{i-2}s_{i}s_{i-1}s_{k-1}=F_{(4,2),n}^{(i-2,k-1)}(s_{2}s_{1}s_{3}s_{2},s_{1}) \]
(where one might wish to recall the definition of a multi-shift map from \Cref{Defn:3:1}). Moreover, by \Cref{Eq:5.3:1}, one can deduce that $z$ consecutively contains the pattern
\[ \underline{p}:=(1423,21)=(s_{3}s_{2},s_{1})\in\mathrm{S}_{4}\times\mathrm{S}_{2} \]
at position $\underline{i}:=(i-2,k-1)$. Therefore, by \Cref{Cor:3:15} we have that
\[ x\not\leq_{L}^{(n)}z \iff F_{(4,2),n}^{(i-2,k-1)}(s_{2}s_{1}s_{3}s_{2},s_{1})\not\leq_{L}^{(n)}z \iff \underline{y}:=(s_{2}s_{1}s_{3}s_{2},s_{1})\not\leq_{L}^{(4,2)}\underline{p}:=(s_{3}s_{2},s_{1}). \] 
This latter non-relation in $\mathrm{S}_{4}\times\mathrm{S}_{2}$ is satisfied since 
\[ y_{1}=s_{2}s_{1}s_{3}s_{2}=3412\not\leq_{L}^{(4)}1423=s_{3}s_{2}=p_{1}, \]
which holds from \Cref{Fig:3:9-1}, or by noting that $\mathrm{sh}(y_{1})=(2,2)$ does not dominate $\mathrm{sh}(p_{1})=(3,1)$.

Now assume that $\mathtt{P}_{x}$ is the {\color{blue}blue} tableau above. Then this implies that
\[ x=s_{k-1}s_{k-2}s_{k-3}s_{k}s_{k-1}s_{k-2}s_{k+1}s_{k}s_{k-1}=F_{6,n}^{k-3}(s_{3}s_{2}s_{1}s_{4}s_{3}s_{2}s_{5}s_{4}s_{3}). \]
Recall that $i+2<k<n-2$ (with the latter inequality required for the {\color{blue}blue} tableau to exist). By Equation~\eqref{Eq:5.3:1}, one can see that $z$ consecutively contains the pattern $235146$ at position $k-3\in[n-5]$. Therefore, by Claim~(b) of \Cref{Prop:3:7}, we have that 
\begin{align*}
x\not\leq_{L}^{(n)}z &\iff F_{6,n}^{k-3}(s_{3}s_{2}s_{1}s_{4}s_{3}s_{2}s_{5}s_{4}s_{3})\not\leq_{L}^{(n)}z, \\
&\iff s_{3}s_{2}s_{1}s_{4}s_{3}s_{2}s_{5}s_{4}s_{3}\not\leq_{L}^{(6)}235146,
\end{align*}
where the latter non-relation can be confirmed computationally, completing the proof.
\end{proof}

By \Cref{Lem:5.3:2} and the preceding discussion, if $x\leq_{R}^{(n)}z$, then $\mathtt{P}_{x}$ belongs to the following list:
\begin{equation}\label{Type3RList}
\emptyset^{\langle n\rangle} \hspace{6mm}
{\begin{ytableau}\scriptstyle i\end{ytableau}}^{\ \langle n-1\rangle} \hspace{6mm}
{\begin{ytableau}\scriptstyle k\end{ytableau}}^{\ \langle n-1\rangle} \hspace{6mm}
{\begin{ytableau}\scriptstyle i&\scriptstyle k\end{ytableau}}^{\ \langle n-2\rangle} \hspace{6mm}
{\begin{ytableau}\scriptstyle k&\scriptstyle k+1\end{ytableau}}^{\ \langle n-2\rangle}
\end{equation}
\[
{\begin{ytableau}\scriptstyle i\\\scriptstyle k\end{ytableau}}^{\ \langle n-2\rangle} \hspace{6mm}
{\begin{ytableau}\scriptstyle i&\scriptstyle k&\scriptstyle k+1\end{ytableau}}^{\ \langle n-3\rangle} \hspace{6mm}
{\begin{ytableau}\scriptstyle i&\scriptstyle k+1\\\scriptstyle k\end{ytableau}}^{\ \langle n-3\rangle}
\]
\begin{lem}\label{Type3IC}
The Type (3) involution $z$ given above satisfies the Indecomposability Conjecture \ref{Conj:2.12:1}. That is to say, we have that $\mathbf{KM}_{n}(\star,z)=\mathtt{true}$.
\end{lem}

\begin{proof}
By Equations~\eqref{Eq:2.11:1} and \eqref{Eq:2.12:2}, it suffices to show $\mathbf{KM}_{n}(x,z)=\mathtt{true}$ for all $x\in\mathrm{Inv}_{n}$ where $x\leq_{R}^{(n)}z$. This relation implies that $\mathtt{P}_{x}$ belongs to the list in Equation~\eqref{Type3RList}. One can check that all involutions $x$ where $\mathtt{P}_{x}$ is one of the tableau from \Cref{Type3RList}, except the tableaux of shape $(1,1)^{\langle n-2\rangle}$ and $(2,1)^{\langle n-3\rangle}$, have support set $\mathrm{Sup}(x)$ of size no greater than $6$. For example, if $\mathtt{P}_{x}$ is the tableau above with shape $(3)^{\langle n-3\rangle}$, then $x=s_{i-1}s_{k-1}s_{k-2}s_{k}s_{k-1}$, and so $|\mathrm{Sup}(x)|=4\leq 6$. Therefore, by the implication in Equation~\eqref{Eq:2.12:4}, we have $\mathbf{KM}_{n}(x,z)=\mathtt{true}$ for all such involutions. Also, when $\mathtt{P}_{x}$ is the tableau of shape $(2,1)^{\langle n-3\rangle}$, then it has the same shape as $z$, so it follows that $\mathbf{KM}_{n}(x,z)=\mathtt{true}$ by \cite[Section 5.2]{KiM16}.

What remains is to prove that $\mathbf{KM}_{n}(x,z)=\mathtt{true}$ where $x\in\mathrm{Inv}_{n}$ is the involution such that
\begin{equation}\label{Eq:Type3IC-1}
\mathtt{P}_{x}= {\begin{ytableau}\scriptstyle i\\\scriptstyle k\end{ytableau}}^{\ \langle n-2\rangle}.
\end{equation}
We seek to show this by induction, and we begin by setting up a little notation. Firstly, we will think of both $k$ and $n$ as fixed, while we think of $2\leq i<k-2$ as something which varies, with these inequalities being those for $i$ defining the Type (3) involution. Then consider both
\[ z_{i}:=z=(i-1,k)(k-1,k+1), \hspace{2mm} \text{ and } \hspace{2mm} x_{i}:=s_{i-1}s_{i}\cdots s_{k-3}(s_{k-2}s_{k-1}s_{k-2}). \]
Hence $z_{i}$ is just $z$, but we are stressing the dependence on $i$, and one can deduce that $x_{i}\sim_{R}^{(n)}x$, that is to say, the insertion tableau of $x_{i}$ equals that of $x$ in \Cref{Eq:Type3IC-1}. Therefore, by the implication in Equation~\eqref{Eq:2.12:2}, it suffices to prove that $\mathbf{KM}_{n}(x_{i},z_{i})=\mathtt{true}$, and we do this by downward induction on $2\leq i<k-2$. For the base case of $i=k-3$, we have that $|\mathrm{Sup}(x_{k-3})|=4$, and thus $\mathbf{KM}_{n}(x_{k-3},z_{k-3})=\mathtt{true}$ by the implication in Equation~\eqref{Eq:2.12:4}. Now let $2\leq i<k-3$, and assume that $\mathbf{KM}_{n}(x_{j},z_{j})=\mathtt{true}$ for all $j$ such that $i<j\leq k-3$, then we seek to show that $\mathbf{KM}_{n}(x_{i},z_{i})=\mathtt{true}$.

By \Cref{Lem:2.6:2}, we have that $C_{x_{i}}^{(n)}=C_{s_{i-1}}^{(n)}C_{s_{i}}^{(n)}\cdots C_{s_{k-3}}^{(n)}(C_{s_{k-2}s_{k-1}s_{k-2}}^{(n)})$, and hence
\begin{equation}\label{Eq:Type3IC-2}
\theta_{x_{i}}^{(n)}\cong\theta_{s_{k-2}s_{k-1}s_{k-2}}^{(n)}\theta_{s_{k-3}}^{(n)}\cdots\theta_{s_{i}}^{(n)}\theta_{s_{i-1}}^{(n)}.
\end{equation}
Now consider the strong right Bruhat walk $\mathrm{w}:=(z_{i},z_{i}s_{i-1})$ and the reduced expression $\underline{r}:=s_{i-1}s_{i}$. Then since $z_{i}=(i-1,k)(k-1,k+1)$ (see also \Cref{Eq:5.3:1}), one can check that $\mathrm{w}$ is $\underline{r}$-compatible, as discussed in \Cref{Sec:2.3}. That is, one can check that the following four Bruhat relations hold:
\[
z_{i}s_{i-1}<z_{i}, \hspace{2mm} z_{i}s_{i}>z_{i}, \hspace{2mm} (z_{i}s_{i-1})s_{i}<z_{i}s_{i-1}, \hspace{2mm} \text{ and } \hspace{2mm} (z_{i}s_{i-1})s_{i-1}>z_{i}s_{i-1}
\]
(with the former and latter equivalent). Thus, we have the following isomorphisms of modules:
\begin{align}
\theta_{x_{i}}^{(n)}L_{z_{i}}^{(n)}&\cong\theta_{s_{k-2}s_{k-1}s_{k-2}}^{(n)}\theta_{s_{k-3}}^{(n)}\cdots\theta_{s_{i+1}}^{(n)}(\theta_{s_{i}}^{(n)}\theta_{s_{i-1}}^{(n)}L_{z_{i}}^{(n)}) \label{Eq:Type3IC-3} \\
&\cong\theta_{s_{k-2}s_{k-1}s_{k-2}}^{(n)}\theta_{s_{k-3}}^{(n)}\cdots\theta_{s_{i+1}}^{(n)}\theta_{s_{i}}^{(n)}L_{z_{i}s_{i-1}}^{(n)} \nonumber \\
&\cong\theta_{x_{i+1}}^{(n)}L_{z_{i}s_{i-1}}^{(n)}. \nonumber
\end{align}
The first isomorphism follows by \Cref{Eq:Type3IC-2}, the second by Claim~(b) of \Cref{Lem:2.11:5}, and the latter by the definition of $x_{i+1}$ and \Cref{Lem:2.6:2}. Lastly, one can check $z_{i}s_{i-1}\sim_{L}^{(n)}z_{i+1}$, i.e. $\mathtt{Q}_{z_{i}s_{i-1}}=\mathtt{Q}_{z_{i+1}}$. Therefore, since $\theta_{x_{i}}^{(n)}L_{z_{i}}^{(n)}\cong\theta_{x_{i+1}}^{(n)}L_{z_{i}s_{i-1}}^{(n)}$ by \Cref{Eq:Type3IC-3}, and by the implication in Equation~\eqref{Eq:2.12:3}, we have that
\[ \mathbf{KM}_{n}(x_{i},z_{i})=\mathbf{KM}_{n}(x_{i+1},z_{i}s_{i-1})=\mathbf{KM}_{n}(x_{i+1},z_{i+1})=\mathtt{true}, \]
where the latter equality $\mathbf{KM}_{n}(x_{i+1},z_{i+1})=\mathtt{true}$ follows from the induction hypothesis, completing the proof by induction, and, in turn, the proof of the lemma.
\end{proof}

\begin{prop}\label{Type3K}
The Type (3) involution $z$ given above is Kostant positive, that is $\mathbf{K}_{n}(z)=\mathtt{true}$.
\end{prop}

\begin{proof}
Firstly, recall that we have $n\geq6$, $i\in[2,n-4]$, $i+2<k<n$, and
\[ z=z_{i,k+1,k}^{n}=(i-1,k)(k-1,k+1)\in\mathrm{S}_{n}. \]
By \Cref{Thm:2.13:1} and \Cref{Type3IC}, it suffices to prove that we have a non-isomorphism
\begin{equation}\label{Eq:Type3K-1}
\theta_{x}^{(n)}L_{z}^{(n)}\not\cong\theta_{y}^{(n)}L_{z}^{(n)} \hspace{1mm} \text{ in } \hspace{1mm} {^\mathbb{Z}}\mathcal{O}_{0}^{(n)},
\end{equation}
for all distinct pairs $x,y\in\mathrm{S}_{n}$ where $x,y\leq_{R}^{(n)}z$. For contradiction, assume there exists $x,y\leq_{R}^{(n)}z$ such that $\theta_{x}^{(n)}L_{z}^{(n)}\cong\theta_{y}^{(n)}L_{z}^{(n)}$. As $x,y\leq_{R}^{(n)}z$, both $\mathtt{P}_{x}$ and $\mathtt{P}_{y}$ belong to the list in Equation~\eqref{Type3RList} above. Moreover, by \Cref{Eq:2.9:6}, we have $x\sim_{L}^{(n)}y$, which implies $\mathtt{Q}_{x}=\mathtt{Q}_{y}$, and in particular $\mathtt{sh}(x)=\mathtt{sh}(y)$. Also, since $x$ and $y$ are distinct, this implies that $\mathtt{P}_{x}\neq\mathtt{P}_{y}$. Hence altogether, both $\mathtt{P}_{x}$ and $\mathtt{P}_{y}$ belong to \eqref{Type3RList}, share the same shape, but are distinct. Up to symmetry in $x$ and $y$, this implies that
\[ (\mathtt{P}_{x},\mathtt{P}_{y})=\left({\begin{ytableau}\scriptstyle i\end{ytableau}}^{\ \langle n-1\rangle}, \hspace{2mm} {\begin{ytableau}\scriptstyle k\end{ytableau}}^{\ \langle n-1\rangle} \right) \hspace{2mm} \text{ or } \hspace{2mm}  (\mathtt{P}_{x},\mathtt{P}_{y})=\left({\begin{ytableau}\scriptstyle i&\scriptstyle k\end{ytableau}}^{\ \langle n-2\rangle}, \hspace{2mm} {\begin{ytableau}\scriptstyle k&\scriptstyle k+1\end{ytableau}}^{\ \langle n-2\rangle}\right). \]
Furthermore, by the equivalence in Equation~\eqref{Eq:2.11:3}, it suffices to consider any particular pair $(x',y')$ such that $(\mathtt{P}_{x'},\mathtt{P}_{y'})$ is one of the above two pairings of insertion tableau, and $x'\sim_{L}^{(n)}y'$. Thus, by assumption, we have that $\theta_{x}^{(n)}L_{z}^{(n)}\cong\theta_{y}^{(n)}L_{z}^{(n)}$ holds for $x,y\leq_{R}^{(n)}z$ being (at least) one of the following two pairs:
\begin{center}
\begin{itemize}
\item[(i)] $(x,y)=(s_{i-1}s_{i}\cdots s_{k-3}, s_{k-1}s_{k-2}s_{k-3})$, or \\
\item[(ii)] $(x,y)=(s_{i-1}s_{i}\cdots s_{k-3}s_{k-1}s_{k-2}s_{k}s_{k-1}, s_{k-1}s_{k-2}s_{k}s_{k-1})$.
\end{itemize}
\end{center}
We examine these two cases separately, showing that both lead to a contradiction:

\textbf{Case (i)}: By assumption, we have an isomorphism $\theta_{x}^{(n)}L_{z}^{(n)}\cong\theta_{y}^{(n)}L_{z}^{(n)}$ where
\[ x:=s_{i-1}s_{i}\cdots s_{k-3} \hspace{2mm} \text{ and } \hspace{2mm} y:=s_{k-1}s_{k-2}s_{k-3}. \]
We seek a contradiction. Firstly, by \Cref{Lem:2.6:2}, we have the decompositions
\[ \theta_{x}^{(n)}\cong\theta_{s_{k-3}}^{(n)}\cdots\theta_{s_{i}}^{(n)}\theta_{s_{i-1}}^{(n)} \hspace{2mm} \text{ and } \hspace{2mm} \theta_{y}^{(n)}\cong\theta_{s_{k-3}}^{(n)}\theta_{s_{k-2}}^{(n)}\theta_{s_{k-1}}^{(n)}. \]
Now consider the reduced expression $\underline{x}=s_{i-1}s_{i}\cdots s_{k-3}$ of $x$, and the strong right Bruhat walk
\begin{equation}\label{Eq:Type3K-2}
\mathrm{w}=(\mathrm{w}_{i-1},\mathrm{w}_{i},\dots,\mathrm{w}_{k-3}):=(z,zs_{i-1},zs_{i-1}s_{i},\dots,zs_{i-1}s_{i}\cdots s_{k-4}).
\end{equation}
The length of $\mathrm{w}$ is $m=k-i-1$, and we have indexed the entries in $\mathrm{w}$ with the set $[i-1,k-3]:=\{i-1,i,\dots,k-3\}$ to allow such entries to be easily compared to the simple transpositions appearing in the reduced expression $\underline{x}$. For $j\in[i-1,k-3]$, the $j$-th entry of $\mathrm{w}$ is given by
\begin{equation*}
\mathrm{w}_{j}=zs_{i-1}s_{i}\cdots s_{j-1},
\end{equation*}
where the product $s_{i-1}s_{i}\cdots s_{j-1}$ equals the identity $1_{n}$ when $j=i-1$. By consulting \Cref{Eq:5.3:1}, one can deduce that the one-line notation for the $j$-th entry of $\mathrm{w}$ is given by
\begin{equation}\label{Eq:Type3K-2.5}
\mathrm{w}_{j}=1\cdots(i-2){\color{teal}i(i+1)\cdots jk}(j+1)\cdots(k-2){\color{teal}(k+1)(i-1)(k-1)}(k+2)\cdots n.
\end{equation}
We have coloured the non-fixed points of $\mathrm{w}_{j}$ in {\color{teal}teal}. One can see that, to obtain this description of $\mathrm{w}_{j}$ from that of $z$ in \Cref{Eq:5.3:1}, one has repeatedly swapped the letter $k$ with its right hand neighbour until $k$ is placed in the $j$-th position. With this description, one can confirm that the strong right Bruhat walk $\mathrm{w}$ is $\underline{x}$-compatible, as described in \Cref{Sec:2.3}. That is, one can confirm that the following Bruhat relations hold for each $j\in[i-1,k-3]$:
\[ \mathrm{w}_{j}s_{j}<\mathrm{w}_{j}, \hspace{2mm} \mathrm{w}_{j}s_{j-1}>\mathrm{w}_{j} \hspace{1mm} \text{(for $j>i-1$)}, \hspace{1mm} \text{ and } \hspace{1mm} \mathrm{w}_{j}s_{j+1}>\mathrm{w}_{j} \hspace{1mm} \text{(for $j<k-3$)}. \]
Therefore, by Claim~(b) of \Cref{Lem:2.11:5}, we have that
\begin{equation}\label{Eq:Type3K-3}
\theta_{x}^{(n)}L_{z}^{(n)}\cong\theta_{s_{k-3}}^{(n)}\cdots\theta_{s_{i}}^{(n)}\theta_{s_{i-1}}^{(n)}L_{z}^{(n)}\cong\theta_{s_{k-3}}^{(n)}L_{\mathrm{w}_{k-3}}^{(n)}=\theta_{s_{k-3}}^{(n)}L_{zs_{i-1}s_{i}\cdots s_{k-4}}^{(n)}.
\end{equation}
Now consider the reduced expression $\underline{y}:=s_{k-1}s_{k-2}s_{k-3}$. Then one can check that the strong right Bruhat walk $(z,zs_{k-1},zs_{k-1}s_{k-2})$ is $\underline{y}$-compatible (see also \Cref{Ex:2.3:1} since the considerations are analogous). Hence, by Claim~(b) of \Cref{Lem:2.11:5}, we also have that
\begin{equation}\label{Eq:Type3K-4}
\theta_{y}^{(n)}L_{z}^{(n)}\cong\theta_{s_{k-3}}^{(n)}\theta_{s_{k-2}}^{(n)}\theta_{s_{k-1}}^{(n)}L_{z}^{(n)}\cong\theta_{s_{k-3}}^{(n)}L_{zs_{k-1}s_{k-2}}^{(n)}.
\end{equation}
By assumption $\theta_{x}^{(n)}L_{z}^{(n)}\cong\theta_{y}^{(n)}L_{z}^{(n)}$, and thus from \Cref{Eq:Type3K-3,Eq:Type3K-4}, we have that
\begin{equation}\label{Eq:Type3K-5}
\theta_{s_{k-3}}^{(n)}L_{zs_{i-1}s_{i}\cdots s_{k-4}}^{(n)}\cong\theta_{s_{k-3}}^{(n)}L_{zs_{k-1}s_{k-2}}^{(n)}.
\end{equation}
Note that both sides of this isomorphism are non-zero, since it can be checked that $s_{k-3}$ belongs to the right descent sets of both $zs_{i-1}s_{i}\cdots s_{k-4}$ and $zs_{k-1}s_{k-2}$. Lastly, by \Cref{Prop:2.10:2}, we see that the module on the left hand side of \Cref{Eq:Type3K-5} has simple top $L_{zs_{i-1}s_{i}\cdots s_{k-4}}\langle-1\rangle$, while the top of the module on the right hand side is $L_{zs_{k-1}s_{k-2}}\langle-1\rangle$. However, regardless of the value of $k$ (which satisfies $i+2<k<n$), we have that $zs_{i-1}s_{i}\cdots s_{k-4}\neq zs_{k-1}s_{k-2}$, which means that these two tops are non-isomorphic, which gives the desired contradiction, completing Case (i).

\textbf{Case (ii)}: By assumption, we have an isomorphism $\theta_{x}^{(n)}L_{z}^{(n)}\cong\theta_{y}^{(n)}L_{z}^{(n)}$ where
\[ x:=s_{i-1}s_{i}\cdots s_{k-3}(s_{k-1}s_{k-2}s_{k}s_{k-1}) \hspace{2mm} \text{ and } \hspace{2mm} y:=s_{k-1}s_{k-2}s_{k}s_{k-1}. \]
We seek a contradiction. Firstly, from \Cref{Lem:2.6:2}, and explicit computation, we have that
\[ \theta_{x}^{(n)}\cong\theta_{y}^{(n)}\theta_{s_{k-3}}^{(n)}\cdots\theta_{s_{i}}^{(n)}\theta_{s_{i-1}}^{(n)} \hspace{1mm} \text{ and } \hspace{1mm} \theta_{y}^{(n)}\cong\theta_{s_{k-1}}^{(n)}\theta_{s_{k-2}}^{(n)}\theta_{s_{k}}^{(n)}\theta_{s_{k-1}}^{(n)}. \]
Now consider the reduced expression $\underline{r}=s_{i-1}s_{i}\cdots s_{k-3}$ and the strong right Bruhat walk
\[ \mathrm{w}:=(z,zs_{i-1},zs_{i-1}s_{i},\dots,zs_{i-1}s_{i}\cdots s_{k-4}). \]
This is the same reduced expression and strong right Bruhat walk which was considered in Case~(i) above. Therefore, from \Cref{Eq:Type3K-3} above, we have that
\[ \theta_{x}^{(n)}L_{z}^{(n)}\cong\theta_{y}^{(n)}(\theta_{s_{k-3}}^{(n)}\cdots\theta_{s_{i}}^{(n)}\theta_{s_{i-1}}^{(n)}L_{z}^{(n)})\cong\theta_{y}^{(n)}\theta_{s_{k-3}}^{(n)}L_{zs_{i-1}s_{i}\cdots s_{k-4}}^{(n)}. \] 
From this isomorphism, and the assumption that $\theta_{x}^{(n)}L_{z}^{(n)}\cong\theta_{y}^{(n)}L_{z}^{(n)}$, we must have the isomorphism $\theta_{y}^{(n)}\theta_{s_{k-3}}^{(n)}L_{zs_{i-1}s_{i}\cdots s_{k-4}}^{(n)}\cong\theta_{y}^{(n)}L_{z}^{(n)}$, and decategorifying this gives the following equality in $\mathrm{H}_{n}^{\mathbb{A}}$:
\begin{equation}\label{Eq:Type3K-6}
D_{zs_{i-1}s_{i}\cdots s_{k-4}}^{(n)}C_{s_{k-3}}^{(n)}C_{y}^{(n)}=D_{z}^{(n)}C_{y}^{(n)}.
\end{equation}
We want to show that this equality cannot be. We will do this by proving that, when expressed in terms of the dual Kazhdan-Lusztig basis $\mathcal{D}_{n}$, the right hand side contains no terms in degree 3 or higher, while the left hand side does contain a term in degree 3. 

Starting with the right hand side of \Cref{Eq:Type3K-6}, we have that
\begin{align}
D_{z}^{(n)}C_{y}^{(n)}&=D_{z}^{(n)}C_{s_{k-1}}^{(n)}C_{s_{k-2}}^{(n)}C_{s_{k}}^{(n)}C_{s_{k-1}}^{(n)} \nonumber \\
&=\left((v+v^{-1})D_{z}^{(n)}+D_{zs_{k-1}}^{(n)}+\sum_{z<w<ws_{k-1}}\mu^{(n)}(z,w)D_{w}^{(n)}\right)C_{s_{k-2}}^{(n)}C_{s_{k}}^{(n)}C_{s_{k-1}}^{(n)} \nonumber \\
&= \left(D_{zs_{k-1}}^{(n)}+\sum_{z<w<ws_{k-1}}\mu^{(n)}(z,w)D_{w}^{(n)}\right)C_{s_{k-2}}^{(n)}C_{s_{k}}^{(n)}C_{s_{k-1}}^{(n)}. \nonumber
\end{align}
The second equality follows by \Cref{Eq:2.8:1}, while the last follows by the fact that $D_{z}^{(n)}C_{s_{k-2}}^{(n)}=0$ (since $s_{k-2}$ does not belong to $\mathrm{Des}_{R}(z)$). We can see from the last equality above that if a degree 3 or higher term appears in $D_{z}^{(n)}C_{y}^{(n)}$, then it must appear in either
\[ D_{zs_{k-1}}^{(n)}C_{s_{k-2}}^{(n)}C_{s_{k}}^{(n)}C_{s_{k-1}}^{(n)} \hspace{1mm} \text{ and/or } \hspace{1mm} D_{w}^{(n)}C_{s_{k-2}}^{(n)}C_{s_{k}}^{(n)}C_{s_{k-1}}^{(n)}, \]
for some $w$ such that $z<w<ws_{k-1}$. By repeatedly applying \Cref{Eq:2.8:1} to either of these, we see that no terms of degree 4 or higher can appear, and for a degree 3 term to appear, it can only be $D_{zs_{k-1}}^{(n)}$ and/or $D_{w}^{(n)}$, and only if $s_{k-2}$, $s_{k}$, and $s_{k-1}$ belong to the respective right descent sets. However, we know that $z<zs_{k-1}$ and $w<ws_{k-1}$, hence no degree 3 term appears in $D_{z}^{(n)}C_{y}^{(n)}$.

Now for the left hand side of \Cref{Eq:Type3K-6}. First let $z':=zs_{i-1}s_{i}\cdots s_{k-4}$, and note that $z'=\mathrm{w}_{k-3}$, the last term in the strong right Bruhat walk from Case (i), and also above. Thus, by \Cref{Eq:Type3K-2.5} with $j=k-3$, we see that $s_{k-3}\in\mathrm{Des}_{R}(z')$. Hence, by \Cref{Eq:2.8:1} we have that
\begin{equation}\label{Eq:Type3K-7}
D_{z'}^{(n)}C_{s_{k-3}}^{(n)}C_{y}^{(n)}=\left((v+v^{-1})D_{z'}^{(n)}+D_{z's_{k-3}}^{(n)}+\sum_{z<w<ws_{k-3}}\mu^{(n)}(z,w)D_{w}^{(n)}\right)C_{y}^{(n)}
\end{equation}
Note, the coefficient of $[D_{a}^{(n)}]D_{b}^{(n)}C_{c}^{(n)}$ always belongs to $\mathbb{A}_{\geq 0}=\mathbb{Z}_{\geq0}[v,v^{-1}]$, i.e. the coefficients of the Laurent polynomial are non-negative, since $[D_{a}^{(n)}]D_{b}^{(n)}C_{c}^{(n)}$ corresponds to the graded multiplicity of $L_{a}^{(n)}$ in $\theta_{c}^{(n)}L_{b}^{(n)}$. Therefore, to show that a degree 3 term appears in $D_{z'}^{(n)}C_{s_{k-3}}^{(n)}C_{y}^{(n)}$, it suffices to show that a degree 3 term appears in the term $(v+v^{-1})D_{z'}^{(n)}C_{y}^{(n)}$, since no cancellation can occur from the other terms in the right hand side of \Cref{Eq:Type3K-7}. Well, first note that
\[ F_{4,n}^{k-2}(s_{2}s_{1}s_{3}s_{2})=y, \]
and the pattern of size $4$ consecutively appearing in $z'$ at position $k-2$ is $2413$. Since
\[ s_{2}s_{1}s_{3}s_{2}=3412\sim_{L}2413, \]
then by Claim~(b) of \Cref{Prop:3:7}, we have that $y\leq_{L}^{(n)}z'$, which is equivalent to $y\leq_{R}^{(n)}(z')^{-1}$ and $D_{z'}^{(n)}C_{y}^{(n)}\neq0$. Therefore, the term $(v+v^{-1})D_{z'}^{(n)}C_{y}^{(n)}$, appearing in \Cref{Eq:Type3K-7}, is non-zero. 

Lastly, as $y\in\mathrm{Inv}_{n}$ and $\mathbf{a}_{n}(y)=2$, then by Claim~(1) of \Cref{Lem:2.11:4}, 
\[ [v^{2}][D_{z'}^{(n)}]\left(D_{z'}^{(n)}C_{y}^{(n)}\right)\neq0. \]
This implies $D_{z'}^{(n)}$ appears in $(v+v^{-1})D_{z'}^{(n)}C_{y}^{(n)}$ at degree 3, and so $D_{z'}^{(n)}$ also appears in $D_{z'}^{(n)}C_{s_{k-3}}^{(n)}C_{y}^{(n)}$ at degree 3. Therefore, \Cref{Eq:Type3K-6} cannot be, giving the desired contradiction. 
\end{proof}

\subsection{Involutions of Type $(4^{\ast})$}\label{Sec:5.4}

Fix $n\geq 5$, $i\in[2,n-3]$, and let $z:=z_{i,i+2,i+3}^{n}$ be the involution in $\mathrm{Inv}_{n}$ of Type $(4^{\ast})$. Hence, it satisfies \Cref{Type4star}, which we recall here:
\begin{equation*}
\tilde{\mathtt{RS}}_{n}(z)={\begin{ytableau}\scriptstyle i&\scriptstyle i+2\\\scriptstyle i+3\end{ytableau}}^{\ \langle n-3\rangle} \hspace{4mm} \text{ and } \hspace{4mm} z=(i-1, i)(i+1,i+3).
\end{equation*}
When written in one-line notation, the involution $z$ is given by
\begin{equation}\label{Eq:5.4:1}
z=1\cdots(i-2){\color{teal}i(i-1)(i+3)}(i+2){\color{teal}(i+1)}(i+4)\cdots n,
\end{equation}
where we have coloured the fixed points in {\color{teal}teal}. We first investigate what elements $x\in\mathrm{S}_{n}$ have a chance at satisfying the relation $x\leq_{R}^{(n)}z$. Assume $x\leq_{R}^{(n)}z$, then by the implications in Equations~\eqref{Eq:2.9:4} and \eqref{Eq:2.9:4.5}, $\mathrm{Des}_{L}(x)\subseteq\mathrm{Des}_{L}(z)=\{s_{i-1},s_{i+1},s_{i+2}\}$ and $\mathtt{sh}(z)=(n-3,2,1)\preceq\mathtt{sh}(x)$, with equality of shapes further implying $x\sim_{L}^{(n)}z$. Therefore, by \Cref{Lem:2.4:1}, $\mathtt{P}_{x}$ must belong to the following list:
\begin{equation}\label{Type4RList}
\emptyset^{\langle n\rangle}
\end{equation}
\[
{\begin{ytableau}\scriptstyle i\end{ytableau}}^{\ \langle n-1\rangle} \hspace{6mm}
{\begin{ytableau}\scriptstyle i+2\end{ytableau}}^{\ \langle n-1\rangle} \hspace{6mm}
{\begin{ytableau}\scriptstyle i+3\end{ytableau}}^{\ \langle n-1\rangle}
\]
\[
{\color{red}{\begin{ytableau}\scriptstyle i&\scriptstyle i+1\end{ytableau}}^{\ \langle n-2\rangle}} \hspace{6mm}
{\begin{ytableau}\scriptstyle i&\scriptstyle i+2\end{ytableau}}^{\ \langle n-2\rangle} \hspace{6mm}
{\begin{ytableau}\scriptstyle i&\scriptstyle i+3\end{ytableau}}^{\ \langle n-2\rangle} \hspace{6mm}
{\begin{ytableau}\scriptstyle i+2&\scriptstyle i+3\end{ytableau}}^{\ \langle n-2\rangle} \hspace{6mm}
{\color{blue}{\begin{ytableau}\scriptstyle i+3&\scriptstyle i+4\end{ytableau}}^{\ \langle n-2\rangle}}
\]
\[
{\begin{ytableau}\scriptstyle i\\\scriptstyle i+2\end{ytableau}}^{\ \langle n-2\rangle} \hspace{6mm}
{\begin{ytableau}\scriptstyle i\\\scriptstyle i+3\end{ytableau}}^{\ \langle n-2\rangle} \hspace{6mm}
{\begin{ytableau}\scriptstyle i+2\\\scriptstyle i+3\end{ytableau}}^{\ \langle n-2\rangle}
\]
\[
{\color{purple}{\begin{ytableau}\scriptstyle i&\scriptstyle i+1&\scriptstyle i+2\end{ytableau}}^{\ \langle n-3\rangle}} \hspace{6mm}
{\color{red}{\begin{ytableau}\scriptstyle i&\scriptstyle i+1&\scriptstyle i+3\end{ytableau}}^{\ \langle n-3\rangle}} \hspace{6mm}
{\color{red}{\begin{ytableau}\scriptstyle i&\scriptstyle i+2&\scriptstyle i+3\end{ytableau}}^{\ \langle n-3\rangle}} \hspace{6mm}
{\color{blue}{\begin{ytableau}\scriptstyle i&\scriptstyle i+3&\scriptstyle i+4\end{ytableau}}^{\ \langle n-3\rangle}}
\]
\[
{\color{blue}{\begin{ytableau}\scriptstyle i+2&\scriptstyle i+3&\scriptstyle i+4\end{ytableau}}^{\ \langle n-3\rangle}} \hspace{6mm}
{\color{cyan}{\begin{ytableau}\scriptstyle i+3&\scriptstyle i+4&\scriptstyle i+5\end{ytableau}}^{\ \langle n-3\rangle}}
\]
\[
{\begin{ytableau}\scriptstyle i&\scriptstyle i+2\\\scriptstyle i+3\end{ytableau}}^{\ \langle n-3\rangle}
\]
Reading top to bottom, left to right, we have the following existence requirements:
\begin{itemize}
\item The three {\color{red}red} tableaux require that $i>2$, with the latter two also requiring that $n>5$.
\item The {\color{purple}purple} tableau requires that $i>3$ and $n>5$.
\item The three {\color{blue}blue} tableaux require that $i<n-3$ and $n>5$.
\item The {\color{cyan}cyan} tableau requires that $i<n-4$ and $n>6$.
\end{itemize}

\begin{lem}\label{Type4IC}
The Type (4$^{\ast}$) involution $z$ above satisfies the Indecomposability Conjecture \ref{Conj:2.12:1}. That is to say, we have that $\mathbf{KM}_{n}(\star,z)=\mathtt{true}$.
\end{lem}

\begin{proof}
By Equations~\eqref{Eq:2.11:1} and \eqref{Eq:2.12:2}, it suffices to show $\mathbf{KM}_{n}(x,z)=\mathtt{true}$ for all $x\in\mathrm{Inv}_{n}$ such that $x\leq_{R}^{(n)}z$. This relation implies that $\mathtt{P}_{x}$ belongs to the list in Equation~\eqref{Type4RList}, and one can check that any involution with such a $\mathtt{P}_{x}$ has support set $\mathrm{Sup}(x)$ of size no greater than $6$. For example, if $\mathtt{P}_{x}$ is the {\color{cyan}cyan} tableaux above (thus assuming $i<n-4$ and $n>6$), then we have that
\[ x=s_{i+2}s_{i+1}s_{i}s_{i+3}s_{i+2}s_{i+1}s_{i+4}s_{i+3}s_{i+2}, \]
hence we have that $|\mathrm{Sup}(x)|=5$. Therefore, by the implication in Equation~\eqref{Eq:2.12:4}, the lemma holds.
\end{proof}

\begin{prop}\label{Type4K}
The Type (4$^{\ast}$) involution $z$ above is Kostant negative, that is $\mathbf{K}_{n}(z)=\mathtt{false}$.
\end{prop}

\begin{proof}
From \Cref{Eq:5.4:1}, we see that $z$ consecutively contains the pattern $2143$ at position $i-1$. Therefore, by \cite[Theorem 3.6]{CM25-1}, the involution $z$ is Kostant negative.
\end{proof}

\subsection{Involutions of Type $(5^{\ast})$}\label{Sec:5.5}

Fix $n\geq 6$, $i\in[2,n-4]$, $i+3<k\leq n$, and let $z:=z_{i,i+2,k}^{n}$ be the involution in $\mathrm{Inv}_{n}$ of Type $(5^{\ast})$. Hence, it satisfies \Cref{Type5star}, which we recall here:
\begin{equation*}
\tilde{\mathtt{RS}}_{n}(z)={\begin{ytableau}\scriptstyle i&\scriptstyle i+2\\\scriptstyle k\end{ytableau}}^{\ \langle n-3\rangle} \hspace{4mm} \text{ and } \hspace{4mm} z=(i-1, i)(i+1,k).
\end{equation*}
When written in one-line notation, the involution $z$ is given by
\begin{equation}\label{Eq:5.5:1}
z=1\cdots(i-2){\color{teal}i(i-1)k}(i+2)\cdots(k-1){\color{teal}(i+1)}(k+1)\cdots n,
\end{equation}
where we have coloured the non-fixed points of $z$ in {\color{teal}teal}. We first investigate what elements $x\in\mathrm{S}_{n}$ have a chance at satisfying the relation $x\leq_{R}^{(n)}z$. Assume $x\leq_{R}^{(n)}z$, then by the implications in Equations~\eqref{Eq:2.9:4} and \eqref{Eq:2.9:4.5}, $\mathrm{Des}_{L}(x)\subseteq\mathrm{Des}_{L}(z)=\{s_{i-1},s_{i+1},s_{k-1}\}$ and $\mathtt{sh}(z)=(n-3,2,1)\preceq\mathtt{sh}(x)$, with equality of shapes further implying $x\sim_{L}^{(n)}z$. Therefore, by \Cref{Lem:2.4:1}, $\mathtt{P}_{x}$ belongs to the following list:
\begin{equation}\label{Type5RList}
\emptyset^{\langle n\rangle}
\end{equation}
\[
{\begin{ytableau}\scriptstyle i\end{ytableau}}^{\ \langle n-1\rangle} \hspace{6mm}
{\begin{ytableau}\scriptstyle i+2\end{ytableau}}^{\ \langle n-1\rangle} \hspace{6mm}
{\begin{ytableau}\scriptstyle k\end{ytableau}}^{\ \langle n-1\rangle}
\]
\[
{\color{red}{\begin{ytableau}\scriptstyle i&\scriptstyle i+1\end{ytableau}}^{\ \langle n-2\rangle}} \hspace{6mm}
{\begin{ytableau}\scriptstyle i&\scriptstyle i+2\end{ytableau}}^{\ \langle n-2\rangle} \hspace{6mm}
{\begin{ytableau}\scriptstyle i&\scriptstyle k\end{ytableau}}^{\ \langle n-2\rangle} \hspace{6mm}
{\begin{ytableau}\scriptstyle i+2&\scriptstyle i+3\end{ytableau}}^{\ \langle n-2\rangle} \hspace{6mm}
{\begin{ytableau}\scriptstyle i+2&\scriptstyle k\end{ytableau}}^{\ \langle n-2\rangle} \hspace{6mm}
{\color{blue}{\begin{ytableau}\scriptstyle k&\scriptstyle k+1\end{ytableau}}^{\ \langle n-2\rangle}}
\]
\[
{\begin{ytableau}\scriptstyle i\\\scriptstyle i+2\end{ytableau}}^{\ \langle n-2\rangle} \hspace{6mm}
{\begin{ytableau}\scriptstyle i\\\scriptstyle k\end{ytableau}}^{\ \langle n-2\rangle} \hspace{6mm}
{\begin{ytableau}\scriptstyle i+2\\\scriptstyle k\end{ytableau}}^{\ \langle n-2\rangle}
\]
\[
{\color{purple}{\begin{ytableau}\scriptstyle i&\scriptstyle i+1&\scriptstyle i+2\end{ytableau}}^{\ \langle n-3\rangle}} \hspace{6mm}
{\color{red}{\begin{ytableau}\scriptstyle i&\scriptstyle i+1&\scriptstyle k\end{ytableau}}^{\ \langle n-3\rangle}} \hspace{6mm}
{\color{red}{\begin{ytableau}\scriptstyle i&\scriptstyle i+2&\scriptstyle i+3\end{ytableau}}^{\ \langle n-3\rangle}} \hspace{6mm}
{\begin{ytableau}\scriptstyle i&\scriptstyle i+2&\scriptstyle k\end{ytableau}}^{\ \langle n-3\rangle} \hspace{6mm}
{\color{blue}{\begin{ytableau}\scriptstyle i&\scriptstyle k&\scriptstyle k+1\end{ytableau}}^{\ \langle n-3\rangle}}
\]
\[
{\begin{ytableau}\scriptstyle i+2&\scriptstyle i+3&\scriptstyle i+4\end{ytableau}}^{\ \langle n-3\rangle} \hspace{6mm}
{\begin{ytableau}\scriptstyle i+2&\scriptstyle i+3&\scriptstyle k\end{ytableau}}^{\ \langle n-3\rangle} \hspace{6mm}
{\color{blue}{\begin{ytableau}\scriptstyle i+2&\scriptstyle k&\scriptstyle k+1\end{ytableau}}^{\ \langle n-3\rangle}} \hspace{6mm}
{\color{cyan}{\begin{ytableau}\scriptstyle k&\scriptstyle k+1&\scriptstyle k+2\end{ytableau}}^{\ \langle n-3\rangle}}
\]
\[
{\begin{ytableau}\scriptstyle i&\scriptstyle i+2\\\scriptstyle k\end{ytableau}}^{\ \langle n-3\rangle}
\]
We have the following existence requirements for the coloured tableaux:
\begin{itemize}
\item The three {\color{red}red} tableaux above require that $i>2$.
\item The {\color{purple}purple} tableau above requires that $i>3$.
\item The three {\color{blue}blue} tableaux above requires that $k<n$.
\item The {\color{cyan}cyan} tableau above requires that $k<n-1$.
\end{itemize}

\begin{lem}\label{Type5IC}
The Type (5$^{\ast}$) involution $z$ above satisfies the Indecomposability Conjecture \ref{Conj:2.12:1}. That is to say, we have that $\mathbf{KM}_{n}(\star,z)=\mathtt{true}$.
\end{lem}

\begin{proof}
By Equations~\eqref{Eq:2.11:1} and \eqref{Eq:2.12:2}, it suffices to show $\mathbf{KM}_{n}(x,z)=\mathtt{true}$, for all $x\in\mathrm{Inv}_{n}$ such that $x\leq_{R}^{(n)}z$. This relation implies that $\mathtt{P}_{x}$ belongs to the list in Equation~\eqref{Type5RList}. Firstly, let $x$ be any involution where $\mathtt{P}_{x}$ is one of the standard Young tableau in \eqref{Type5RList}, except for one of the following three:
\[ {\begin{ytableau}\scriptstyle i\\\scriptstyle k\end{ytableau}}^{\ \langle n-2\rangle}, \hspace{6mm}
{\begin{ytableau}\scriptstyle i+2\\\scriptstyle k\end{ytableau}}^{\ \langle n-2\rangle}, \hspace{3mm} \text{ or } \hspace{3mm}
{\begin{ytableau}\scriptstyle i&\scriptstyle i+2\\\scriptstyle k\end{ytableau}}^{\ \langle n-3\rangle}.
 \]
Then it can be checked that the support set $\mathrm{Sup}(x)$ is of size no greater than $6$. For example, if $\mathtt{P}_{x}$ is the last {\color{blue}blue} tableau above (thus assuming $k<n$), then we have that
\[ x=s_{i+1}s_{k-1}s_{k-2}s_{k}s_{k-1}, \]
and so $|\mathrm{Sup}(x)|=4$. Thus, by the implication in Equation~\eqref{Eq:2.12:4}, $\mathbf{KM}_{n}(x,z)=\mathtt{true}$ for such an involution $x$. Also, if $\mathtt{P}_{x}$ is the tableau of shape $(n-3,2,1)$ given in the exceptions above, then it has the same shape as $z$ itself, and thus $\mathbf{KM}_{n}(x,z)=\mathtt{true}$ by \cite[Section 5.2]{KiM16}. Hence, what remains to be shown, is that $\mathbf{KM}_{n}(x,z)=\mathtt{true}$ when $\mathtt{P}_{x}$ is either 
\[ \text{(i)} \hspace{3mm} {\begin{ytableau}\scriptstyle i\\\scriptstyle k\end{ytableau}}^{\ \langle n-2\rangle}  \text{ or } \hspace{3mm}
\text{(ii)} \hspace{3mm} {\begin{ytableau}\scriptstyle i+2\\\scriptstyle k\end{ytableau}}^{\ \langle n-2\rangle}. \]

{\bf Case (i)}: Thus assume $x\in\mathrm{Inv}_{n}$ is the involution such that 
\[ \mathtt{P}_{x}={\begin{ytableau}\scriptstyle i\\\scriptstyle k\end{ytableau}}^{\ \langle n-2\rangle}. \]
We seek to show that $\mathbf{KM}_{n}(x,z)=\mathtt{true}$ by induction, and we begin by setting up a little notation. Firstly, we will think of both $i$ and $n$ as fixed, while we think of $i+3<k\leq n$ as something which varies, with these inequalities being those which define the Type (5$^{\ast}$) involution. Then consider both
\[ z_{k}:=z=(i-1,i)(i+1,k), \hspace{2mm} \text{ and } \hspace{2mm} x_{k}:=s_{k-1}s_{k-2}\cdots s_{i+1}(s_{i}s_{i-1}s_{i}). \]
Hence $z_{k}$ is just $z$, but we are stressing the dependence on $k$, and one can deduce that $x_{k}\sim_{R}^{(n)}x$, that is to say, the insertion tableau of $x_{k}$ equals that of $x$, being that of Case (i) displayed above. Thus, by the implication in Equation~\eqref{Eq:2.12:2}, it suffices to prove that $\mathbf{KM}_{n}(x_{k},z_{k})=\mathtt{true}$, and we do this by induction on $k$ in the range $i+3<k\leq n$. For the base case of $k=i+4$, we see that $|\mathrm{Sup}(x_{i+4})|=5$, and thus $\mathbf{KM}_{n}(x_{k},z_{k})=\mathtt{true}$ by the implication in Equation~\eqref{Eq:2.12:4}. Now let $i+4<k\leq n$, and assume $\mathbf{KM}_{n}(x_{j},z_{j})=\mathtt{true}$ for all $j$ such that $i+3<j<k$, then we seek to show that $\mathbf{KM}_{n}(x_{k},z_{k})=\mathtt{true}$.

By \Cref{Lem:2.6:2}, we have that $C_{x_{k}}^{(n)}=C_{s_{k-1}}^{(n)}C_{s_{k-2}}^{(n)}\cdots C_{s_{i+1}}^{(n)}(C_{s_{i}s_{i-1}s_{i}}^{(n)})$, and hence
\begin{equation}\label{Eq:Type5IC-1}
\theta_{x_{k}}^{(n)}\cong\theta_{s_{i}s_{i-1}s_{i}}^{(n)}\theta_{s_{i+1}}^{(n)}\cdots\theta_{s_{k-2}}^{(n)}\theta_{s_{k-1}}^{(n)}.
\end{equation}
Now consider the strong right Bruhat walk $\mathrm{w}:=(z_{k},z_{k}s_{k-1})$ and reduced expression $\underline{r}:=s_{k-1}s_{k-2}$. Then since $z_{i}=(i-1,i)(i+1,k)$ (see also \Cref{Eq:5.5:1}), one can check that $\mathrm{w}$ is $\underline{r}$-compatible, as discussed in \Cref{Sec:2.3}. That is, one can check that the following four Bruhat relations hold:
\[
z_{k}s_{k-1}<z_{k}, \hspace{2mm} z_{k}s_{k-2}>z_{k}, \hspace{2mm} (z_{k}s_{k-1})s_{k-2}<z_{k}s_{k-1}, \hspace{2mm} \text{ and } \hspace{2mm} (z_{k}s_{k-1})s_{k-1}>z_{k}s_{k-1}
\]
(with the former and latter equivalent). Thus, we have the following isomorphisms of modules:
\begin{align}
\theta_{x_{k}}^{(n)}L_{z_{k}}^{(n)}&\cong\theta_{s_{i}s_{i-1}s_{i}}^{(n)}\theta_{s_{i+1}}^{(n)}\cdots\theta_{s_{k-3}}^{(n)}(\theta_{s_{k-2}}^{(n)}\theta_{s_{k-1}}^{(n)}L_{z_{k}}^{(n)}) \label{Eq:Type5IC-2} \\
&\cong\theta_{s_{i}s_{i-1}s_{i}}^{(n)}\theta_{s_{i+1}}^{(n)}\cdots\theta_{s_{k-3}}^{(n)}\theta_{s_{k-2}}^{(n)}L_{z_{k}s_{k-1}}^{(n)} \nonumber \\
&\cong\theta_{x_{k-1}}^{(n)}L_{z_{k}s_{k-1}}^{(n)}. \nonumber
\end{align}
The first isomorphism follows from Equation~\eqref{Eq:Type5IC-1}, the second from Claim~(b) of \Cref{Lem:2.11:5}, and the last by the definition of $x_{k-1}$ and Equation~\eqref{Eq:Type5IC-1}. Lastly, one can check that $z_{k}s_{k-1}\sim_{L}^{(n)}z_{k-1}$, i.e. $\mathtt{Q}_{z_{k}s_{k-1}}=\mathtt{Q}_{z_{k-1}}$. Therefore, since $\theta_{x_{k}}^{(n)}L_{z_{k}}^{(n)}\cong\theta_{x_{k-1}}^{(n)}L_{z_{k}s_{k-1}}^{(n)}$ by  Equation~\eqref{Eq:Type5IC-2}, and by the implication in Equation~\eqref{Eq:2.12:3}, we have that
\[ \mathbf{KM}_{n}(x_{k},z_{k})=\mathbf{KM}_{n}(x_{k-1},z_{k}s_{k-1})=\mathbf{KM}_{n}(x_{k-1},z_{k-1})=\mathtt{true}, \]
where the latter equality $\mathbf{KM}_{n}(x_{k-1},z_{k-1})=\mathtt{true}$ follows from the induction hypothesis, completing the proof by induction, and in turn the proof of the Case (i).

{\bf Case (ii)}: This case follows in a completely analogous manner to Case (i) above, where one proves $\mathbf{KM}_{n}(x_{k},z_{k})=\mathtt{true}$ by induction on $k$ in the range $i+3<k\leq n$, but this time we have
\[ x_{k}:=s_{k-1}s_{k-2}\cdots s_{i+3}(s_{i+2}s_{i+1}s_{i+2}). \]
Beside this, one can check that everything else in the above argument holds in this setting too, so we skip the details. This completes the proof of the lemma.
\end{proof}

\begin{prop}\label{Type5K}
The Type (5$^{\ast}$) involution $z$ above is Kostant negative, that is $\mathbf{K}_{n}(z)=\mathtt{false}$.
\end{prop}

\begin{proof}
From \Cref{Eq:5.5:1}, we see that $z$ consecutively contains the pattern $2143$ at position $i-1$. Therefore, by \cite[Theorem 3.6]{CM25-1}, the involution $z$ is Kostant negative.
\end{proof}

\subsection{Involutions of Type $(6^{\ast})$}\label{Sec:5.6}

Fix $n\geq 6$, $i\in[2,n-4]$, $i+2<j<n$, and let $z:=z_{i,j,j+1}^{n}$ be the involution in $\mathrm{Inv}_{n}$ of Type $(6^{\ast})$. Hence, it satisfies \Cref{Type6star}, which we recall here:
\begin{equation*}
\tilde{\mathtt{RS}}_{n}(z)={\begin{ytableau}\scriptstyle i&\scriptstyle j\\\scriptstyle j+1\end{ytableau}}^{\ \langle n-3\rangle} \hspace{4mm} \text{ and } \hspace{4mm} z=(i-1, i)(j-1,j+1).
\end{equation*}
When written in one-line notation, the involution $z$ is given by
\begin{equation}\label{Eq:5.6:1}
z=1\cdots(i-2){\color{teal}i(i-1)}(i+1)\cdots(j-2){\color{teal}(j+1)}j{\color{teal}(j-1)}(j+2)\cdots n,
\end{equation}
where we have coloured the non-fixed points of $z$ in {\color{teal}teal}. We first investigate what elements $x\in\mathrm{S}_{n}$ have a chance at satisfying the relation $x\leq_{R}^{(n)}z$. Assume $x\leq_{R}^{(n)}z$, then by the implication in Equations~\eqref{Eq:2.9:4} and \eqref{Eq:2.9:4.5}, $\mathrm{Des}_{L}(x)\subseteq\mathrm{Des}_{L}(z)=\{s_{i-1},s_{j-1},s_{j}\}$ and $\mathtt{sh}(z)=(n-3,2,1)\preceq\mathtt{sh}(x)$, with equality of shapes further implying $x\sim_{L}^{(n)}z$. Therefore, by \Cref{Lem:2.4:1}, $\mathtt{P}_{x}$ belongs to the following list:
\begin{equation}\label{Type6RList}
\emptyset^{\langle n\rangle}
\end{equation}
\[
{\begin{ytableau}\scriptstyle i\end{ytableau}}^{\ \langle n-1\rangle} \hspace{6mm}
{\begin{ytableau}\scriptstyle j\end{ytableau}}^{\ \langle n-1\rangle} \hspace{6mm}
{\begin{ytableau}\scriptstyle j+1\end{ytableau}}^{\ \langle n-1\rangle}
\]
\[
{\color{red}{\begin{ytableau}\scriptstyle i&\scriptstyle i+1\end{ytableau}}^{\ \langle n-2\rangle}} \hspace{6mm}
{\begin{ytableau}\scriptstyle i&\scriptstyle j\end{ytableau}}^{\ \langle n-2\rangle} \hspace{6mm}
{\begin{ytableau}\scriptstyle i&\scriptstyle j+1\end{ytableau}}^{\ \langle n-2\rangle} \hspace{6mm}
{\begin{ytableau}\scriptstyle j&\scriptstyle j+1\end{ytableau}}^{\ \langle n-2\rangle} \hspace{6mm}
{\color{blue}{\begin{ytableau}\scriptstyle j+1&\scriptstyle j+2\end{ytableau}}^{\ \langle n-2\rangle}}
\]
\[
{\begin{ytableau}\scriptstyle i\\\scriptstyle j\end{ytableau}}^{\ \langle n-2\rangle} \hspace{6mm}
{\begin{ytableau}\scriptstyle i\\\scriptstyle j+1\end{ytableau}}^{\ \langle n-2\rangle} \hspace{6mm}
{\begin{ytableau}\scriptstyle j\\\scriptstyle j+1\end{ytableau}}^{\ \langle n-2\rangle}
\]
\[
{\color{purple}{\begin{ytableau}\scriptstyle i&\scriptstyle i+1&\scriptstyle i+2\end{ytableau}}^{\ \langle n-3\rangle}} \hspace{6mm}
{\color{red}{\begin{ytableau}\scriptstyle i&\scriptstyle i+1&\scriptstyle j\end{ytableau}}^{\ \langle n-3\rangle}} \hspace{6mm}
{\color{red}{\begin{ytableau}\scriptstyle i&\scriptstyle i+1&\scriptstyle j+1\end{ytableau}}^{\ \langle n-3\rangle}} \hspace{6mm}
{\begin{ytableau}\scriptstyle i&\scriptstyle j&\scriptstyle j+1\end{ytableau}}^{\ \langle n-3\rangle}
\]
\[
{\color{blue}{\begin{ytableau}\scriptstyle j&\scriptstyle j+1&\scriptstyle j+2\end{ytableau}}^{\ \langle n-3\rangle}} \hspace{6mm}
{\color{cyan}{\begin{ytableau}\scriptstyle j+1&\scriptstyle j+2&\scriptstyle j+3\end{ytableau}}^{\ \langle n-3\rangle}}
\]
\[
{\begin{ytableau}\scriptstyle i&\scriptstyle j\\\scriptstyle j+1\end{ytableau}}^{\ \langle n-3\rangle}
\]
We have the following existence requirements for the coloured tableaux:
\begin{itemize}
\item The three {\color{red}red} tableaux above require that $i>2$.
\item The {\color{purple}purple} tableau above requires that $i>3$.
\item The two {\color{blue}blue} tableaux above require that $j<n-1$.
\item The {\color{cyan}cyan} tableau above requires that $j<n-2$.
\end{itemize}

\begin{prop}\label{Type6K}
The answer to Kostant's problem for the Type (6$^{\ast}$) involution $z$ is given by
\[ \mathbf{K}_{n}(z)=
\begin{cases} 
\mathtt{true}, &  j+1=n\\
\mathsf{false}, & \text{otherwise}
\end{cases}.
\]
\end{prop}

\begin{proof}
If $j+1\neq n$, then by \Cref{Eq:5.6:1}, we see that $z$ consecutively contains the pattern $14325$ at position $j-2$. Therefore, by \cite[Theorem 3.6]{CM25-1}, we have that $\mathbf{K}_{n}(z)=\mathtt{false}$.

Now let $j+1=n$. We seek to show $\mathbf{K}_{n}(z)=\mathtt{true}$. By \Cref{Eq:2.13:3}, it suffices to prove that
\[ z':=w_{0,n}zw_{0,n}=(1,3)(i',i'+1)\in\mathrm{S}_{n}, \]
is Kostant positive, with $i':=n-i+1$ and $4<i'<n$. Via one-line notation, one can deduce that
\[ z' \sim_{L}^{(n)}(s_{1}s_{n-2})w_{0,n-1}w_{0,n}, \]
where $s_{1}s_{n-2}$ belongs to the parabolic subgroup $\mathrm{S}_{n-1}\subset\mathrm{S}_{n}$. Therefore, by \cite[Theorem 1.1]{K10}, we have that $\mathbf{K}_{n}(z')=\mathbf{K}_{n-1}(s_{1}s_{n-2})$. Lastly, the fully commutative element $s_{1}s_{n-2}$ is Kostant positive by \cite[Theorem 5.1]{MMM24}, completing the proof.
\end{proof}

\begin{lem}\label{Type6IC}
The Type (6$^{\ast}$) involution $z$ above satisfies the Indecomposability Conjecture \ref{Conj:2.12:1}. That is to say, we have that $\mathbf{KM}_{n}(\star,z)=\mathtt{true}$.
\end{lem}

\begin{proof}
For the sake of this proof, it would be convenient to add back the subscripts and superscripts for our element $z=z_{i,j,j+1}^{n}$.
By \Cref{Thm:2.13:1}, we know that when $z_{i,j,j+1}^{n}$ is Kostant positive, then $\mathbf{KM}_{n}(\star,z_{i,j,j+1}^{n})=\mathtt{true}$. Thus by \Cref{Type6K}, we have that $\mathbf{KM}_{n}(\star,z_{i,n-1,n}^{n})=\mathtt{true}$ (i.e. when $j+1=n$).

It remains to show $\mathbf{KM}_{n}(\star,z_{i,j,j+1}^{n})=\mathtt{true}$ when $j+1\neq n$. By Equations~\eqref{Eq:2.11:1} and \eqref{Eq:2.12:2}, it suffices to show
\[ \mathbf{KM}_{n}(x,z_{i,j,j+1}^{n})=\mathtt{true}, \]
for all $x\in\mathrm{Inv}_{n}$ where $x\leq_{R}^{(n)}z_{i,j,j+1}^{n}$. This relation implies that $\mathtt{P}_{x}$ belongs to \eqref{Type6RList}. Let $x$ be an involution with $\mathtt{P}_{x}$ one of the standard Young tableau in \eqref{Type6RList}, except for one of the following three:
\[ {\begin{ytableau}\scriptstyle i\\\scriptstyle j\end{ytableau}}^{\ \langle n-2\rangle}, \hspace{6mm}
{\begin{ytableau}\scriptstyle i\\\scriptstyle j+1\end{ytableau}}^{\ \langle n-2\rangle}, \hspace{3mm} \text{ or } \hspace{3mm}
{\begin{ytableau}\scriptstyle i&\scriptstyle j\\\scriptstyle j+1\end{ytableau}}^{\ \langle n-3\rangle}.
 \]
Then it can be checked that the support set $\mathrm{Sup}(x)$ is of size no greater than $6$. For example, if $\mathtt{P}_{x}$ is the last {\color{blue}blue} tableau above (thus assuming $j<n-1$), then we have that
\[ x=s_{j-1}s_{j-2}s_{j-3}s_{j}s_{j-1}s_{j-2}s_{j+1}s_{j}s_{j-1}, \]
and so $|\mathrm{Sup}(x)|=5$. Thus, by Implication \eqref{Eq:2.12:4}, $\mathbf{KM}_{n}(x,z_{i,j,j+1}^{n})=\mathtt{true}$ for such an involution $x$. Also, if $\mathtt{P}_{x}$ is the tableau of shape $(n-3,2,1)$ given in the exceptions above, then it has the same shape as $z_{i,j,j+1}^{n}$ itself, and thus $\mathbf{KM}_{n}(x,z_{i,j,j+1}^{n})=\mathtt{true}$ by \cite[Section 5.2]{KiM16}. Hence, we just need to show that $\mathbf{KM}_{n}(x,z_{i,j,j+1}^{n})=\mathtt{true}$ when $\mathtt{P}_{x}$ is either 
\begin{equation*}
\hspace{3mm} {\begin{ytableau}\scriptstyle i\\\scriptstyle j\end{ytableau}}^{\ \langle n-2\rangle}  \text{ or }  \hspace{6mm} {\begin{ytableau}\scriptstyle i\\\scriptstyle j+1\end{ytableau}}^{\ \langle n-2\rangle}.
\end{equation*}
Let $x_{1}$ and $x_{2}$ be the involutions whose insertion tableau are these two above, respectively. Thus
\[ x_{1}=(i-1,j) \hspace{2mm} \text{ and } \hspace{2mm} x_{2}=(i-1,j+1). \]
Now consider the Kostant positive Type (6$^{\ast}$) involution $z_{i,j,j+1}^{j+1}$ belonging to $\mathrm{S}_{j+1}$ (which is Kostant positive by \Cref{Type6K}). Then note that both $x_{1}$ and $x_{2}$ also belong to the subgroup $\mathrm{S}_{j+1}\subset\mathrm{S}_{n}$. Therefore, via the natural embedding, we have the shifts
\[ F_{j+1,n}^{1}(z_{i,j,j+1}^{j+1})=z_{i,j,j+1}^{n}, \hspace{2mm} F_{j+1,n}^{1}(x_{1})=x_{1}, \hspace{2mm} \text{ and } \hspace{2mm} F_{j+1,n}^{1}(x_{2})=x_{2}. \]
Hence, by \Cref{Lem:3:19}, and since $z_{i,j,j+1}^{j+1}$ is Kostant positive (as a permutation in $\mathrm{S}_{j+1}$),
\[ \mathtt{true}=\mathbf{KM}_{j+1}\left(x_{a},z_{i,j,j+1}^{j+1}\right)=\mathbf{KM}_{n}\left(F_{j+1,n}^{1}(x_{a}),F_{j+1,n}^{1}(z_{i,j,j+1}^{j+1})\right)=\mathbf{KM}_{n}\left(x_{a},z_{i,j,j+1}^{n}\right), \]
for each $a\in[2]$, which completes the proof of the lemma.
\end{proof}

\subsection{Involutions of Type $(7^{\ast})$}\label{Sec:5.7}

Fix $n\geq 7$, $i\in[2,n-5]$, $i+2<j<n$, $j+1<k\leq n$, and let $z:=z_{i,j,k}^{n}$ be the involution in $\mathrm{Inv}_{n}$ of Type $(7^{\ast})$. Recalling from \Cref{Type7star}, this means 
\begin{equation*}
\tilde{\mathtt{RS}}_{n}(z)={\begin{ytableau}\scriptstyle i&\scriptstyle j\\\scriptstyle k\end{ytableau}}^{\ \langle n-3\rangle} \hspace{4mm} \text{ and } \hspace{4mm} z=(i-1, i)(j-1,k).
\end{equation*}
When written in one-line notation, the involution $z$ is given by
\begin{equation}\label{Eq:5.7:1}
z=1\cdots(i-2){\color{teal}i(i-1)}(i+1)\cdots(j-2){\color{teal}k}j(j+1)\cdots(k-1){\color{teal}(j-1)}(k+1)(k+2)\cdots n,
\end{equation}
where we have coloured the non-fixed points of $z$ in {\color{teal}teal}. We first investigate what elements $x\in\mathrm{S}_{n}$ have a chance at satisfying the relation $x\leq_{R}^{(n)}z$. Assume $x\leq_{R}^{(n)}z$, then by the implications in Equations~\eqref{Eq:2.9:4} and \eqref{Eq:2.9:4.5}, we have the inclusion $\mathrm{Des}_{L}(x)\subseteq\mathrm{Des}_{L}(z)=\{s_{i-1},s_{j-1},s_{k-1}\}$ as well as $\mathtt{sh}(z)=(n-3,2,1)\preceq\mathtt{sh}(x)$, with equality of shapes further implying $x\sim_{L}^{(n)}z$. Therefore, by \Cref{Lem:2.4:1}, $\mathtt{P}_{x}$ belongs to the following list:
\begin{equation}\label{Type7RList1}
\emptyset^{\langle n\rangle}
\end{equation}
\[
{\begin{ytableau}\scriptstyle i\end{ytableau}}^{\ \langle n-1\rangle} \hspace{6mm}
{\begin{ytableau}\scriptstyle j\end{ytableau}}^{\ \langle n-1\rangle} \hspace{6mm}
{\begin{ytableau}\scriptstyle k\end{ytableau}}^{\ \langle n-1\rangle}
\]
\[
{\begin{ytableau}\scriptstyle i&\scriptstyle j\end{ytableau}}^{\ \langle n-2\rangle} \hspace{6mm}
{\begin{ytableau}\scriptstyle i&\scriptstyle k\end{ytableau}}^{\ \langle n-2\rangle} \hspace{6mm}
{\begin{ytableau}\scriptstyle j&\scriptstyle k\end{ytableau}}^{\ \langle n-2\rangle} \hspace{6mm}
{\color{red}{\begin{ytableau}\scriptstyle i&\scriptstyle i+1\end{ytableau}}^{\ \langle n-2\rangle}} \hspace{6mm}
{\begin{ytableau}\scriptstyle j&\scriptstyle j+1\end{ytableau}}^{\ \langle n-2\rangle} \hspace{6mm}
{\color{blue}{\begin{ytableau}\scriptstyle k&\scriptstyle k+1\end{ytableau}}^{\ \langle n-2\rangle}}
\]
\[
{\begin{ytableau}\scriptstyle i\\\scriptstyle j\end{ytableau}}^{\ \langle n-2\rangle} \hspace{6mm}
{\begin{ytableau}\scriptstyle i\\\scriptstyle k\end{ytableau}}^{\ \langle n-2\rangle} \hspace{6mm}
{\begin{ytableau}\scriptstyle j\\\scriptstyle k\end{ytableau}}^{\ \langle n-2\rangle}
\]
\[
{\begin{ytableau}\scriptstyle i&\scriptstyle j&\scriptstyle k\end{ytableau}}^{\ \langle n-3\rangle} \hspace{6mm}
{\color{red}{\begin{ytableau}\scriptstyle i&\scriptstyle i+1&\scriptstyle j\end{ytableau}}^{\ \langle n-3\rangle}} \hspace{6mm}
{\color{red}{\begin{ytableau}\scriptstyle i&\scriptstyle i+1&\scriptstyle k\end{ytableau}}^{\ \langle n-3\rangle}} \hspace{6mm}
{\begin{ytableau}\scriptstyle i&\scriptstyle j&\scriptstyle j+1\end{ytableau}}^{\ \langle n-3\rangle} \hspace{6mm}
{\color{blue}{\begin{ytableau}\scriptstyle i&\scriptstyle k&\scriptstyle k+1\end{ytableau}}^{\ \langle n-3\rangle}}
\]
\[
{\begin{ytableau}\scriptstyle j&\scriptstyle j+1&\scriptstyle k\end{ytableau}}^{\ \langle n-3\rangle} \hspace{6mm}
{\color{blue}{\begin{ytableau}\scriptstyle j&\scriptstyle k&\scriptstyle k+1\end{ytableau}}^{\ \langle n-3\rangle}} \hspace{6mm}
{\color{purple}{\begin{ytableau}\scriptstyle i&\scriptstyle i+1&\scriptstyle i+2\end{ytableau}}^{\ \langle n-3\rangle}} \hspace{6mm}
{\color{cyan}{\begin{ytableau}\scriptstyle k&\scriptstyle k+1&\scriptstyle k+2\end{ytableau}}^{\ \langle n-3\rangle}} 
\]
\[
{\begin{ytableau}\scriptstyle i&\scriptstyle j\\\scriptstyle k\end{ytableau}}^{\ \langle n-3\rangle}
\]

\begin{rmk}\label{Rmk:Type6RList}
We have not included ${\begin{ytableau}\scriptstyle j&\scriptstyle j+1&\scriptstyle j+2\end{ytableau}}^{\ \langle n-3\rangle}$ above as this is accounted for by ${\begin{ytableau}\scriptstyle j&\scriptstyle j+1&\scriptstyle k\end{ytableau}}^{\ \langle n-3\rangle}$ with $k=j+2$. This is precisely the exception mentioned in \Cref{Rmk:5:2}. We will not need to explicitly distinguish between this overlap due to \Cref{Lem:5.7:1} below.
\end{rmk}

Reading top to bottom, left to right, we have the following existence requirements:
\begin{itemize}
\item The three {\color{red}red} tableaux above require that $i>2$.
\item The {\color{purple}purple} tableau above requires that $i>4$.
\item The three {\color{blue}blue} tableaux above require that $k<n$.
\item The {\color{cyan}cyan} tableau above requires that $k<n-1$.
\end{itemize}

We now show that the above list can be significantly reduced.

\begin{lem}\label{Lem:5.7:1}
We have that $x\not\leq_{R}^{(n)}z$ whenever $\mathtt{P}_{x}$ is one of the following:
\[
{\color{red}{\begin{ytableau}\scriptstyle i&\scriptstyle i+1\end{ytableau}}^{\ \langle n-2\rangle}} \hspace{6mm}
{\begin{ytableau}\scriptstyle j&\scriptstyle j+1\end{ytableau}}^{\ \langle n-2\rangle} \hspace{6mm}
{\color{blue}{\begin{ytableau}\scriptstyle k&\scriptstyle k+1\end{ytableau}}^{\ \langle n-2\rangle}} \hspace{6mm} {\begin{ytableau}\scriptstyle i\\\scriptstyle j\end{ytableau}}^{\ \langle n-2\rangle}
\]
\[
{\color{red}{\begin{ytableau}\scriptstyle i&\scriptstyle i+1&\scriptstyle j\end{ytableau}}^{\ \langle n-3\rangle}} \hspace{6mm}
{\color{red}{\begin{ytableau}\scriptstyle i&\scriptstyle i+1&\scriptstyle k\end{ytableau}}^{\ \langle n-3\rangle}} \hspace{6mm}
{\begin{ytableau}\scriptstyle i&\scriptstyle j&\scriptstyle j+1\end{ytableau}}^{\ \langle n-3\rangle} \hspace{6mm}
{\color{blue}{\begin{ytableau}\scriptstyle i&\scriptstyle k&\scriptstyle k+1\end{ytableau}}^{\ \langle n-3\rangle}}
\]
\[
{\begin{ytableau}\scriptstyle j&\scriptstyle j+1&\scriptstyle k\end{ytableau}}^{\ \langle n-3\rangle} \hspace{6mm}
{\color{blue}{\begin{ytableau}\scriptstyle j&\scriptstyle k&\scriptstyle k+1\end{ytableau}}^{\ \langle n-3\rangle}} \hspace{6mm}
{\color{purple}{\begin{ytableau}\scriptstyle i&\scriptstyle i+1&\scriptstyle i+2\end{ytableau}}^{\ \langle n-3\rangle}} \hspace{6mm}
{\color{cyan}{\begin{ytableau}\scriptstyle k&\scriptstyle k+1&\scriptstyle k+2\end{ytableau}}^{\ \langle n-3\rangle}} 
\]
\end{lem}

\begin{proof}
It suffices to prove $x\not\leq_{R}^{(n)}z$ with $x$ an involution, equivalently $x\not\leq_{L}^{(n)}z$ (by \Cref{Eq:2.9:2}). Such a non-relation can be established, for each of the above tableaux in turn, by employing either Claim~(b) of \Cref{Prop:3:7} or \Cref{Cor:3:15} (similarly to the proofs of \Cref{Lem:5.2:1} and \Cref{Lem:5.3:2}).

We will demonstrate this for three of the above tableaux. From these three cases, it will be clear how to establish the remaining cases via analogous considerations. Firstly, let us consider the case
\[ \mathtt{P}_{x}={\begin{ytableau}\scriptstyle j&\scriptstyle j+1\end{ytableau}}^{\ \langle n-2\rangle}, \hspace{2mm} \text{ thus } \hspace{2mm} x=s_{j-1}s_{j-2}s_{j}s_{j-1}=F_{4,n}^{j-2}(s_{2}s_{1}s_{3}s_{2}). \]
Note from \Cref{Eq:5.7:1}, $z$ consecutively contains the pattern $p:=1423\in\mathrm{S}_{4}$ at position $j-2$. Therefore, by Claim~(b) of \Cref{Prop:3:7}, we have that
\begin{equation*}
x\not\leq_{L}^{(n)}z \iff F_{4,n}^{j-2}(s_{2}s_{1}s_{3}s_{2})\not\leq_{L}^{(n)}z \iff s_{2}s_{1}s_{3}s_{2}\not\leq_{L}^{(4)}1423,
\end{equation*}
and the latter holds from the implication in \Cref{Eq:2.9:4}, since the shape $\mathtt{sh}(s_{2}s_{1}s_{3}s_{2})=(2,2)$ does not dominate the shape $\mathtt{sh}(1423)=(3,1)$ (or by consulting \Cref{Fig:3:9-1}). Now, suppose that we have 
\[ \mathtt{P}_{x}={\begin{ytableau}\scriptstyle i\\\scriptstyle j\end{ytableau}}^{\ \langle n-2\rangle}, \hspace{2mm} \text{ thus } \hspace{2mm} x=(i-1,j)=F_{j',n}^{i-1}((1,j')), \]
where $j':=j-i+2$. From \Cref{Eq:5.7:1}, we see that $z$ consecutively contains the pattern
\[ p:=2134\cdots(j'-2)j'(j'-1)=s_{1}s_{j'-1}\in\mathrm{S}_{j'} \]
at position $i-1$. Therefore, by Claim~(b) of \Cref{Prop:3:7}, we have that
\begin{equation*}
x\not\leq_{L}^{(n)}z \iff F_{j',n}^{i-1}((1,j'))\not\leq_{L}^{(n)}z \iff (1,j')\not\leq_{L}^{(j')}s_{1}s_{j'},
\end{equation*}
and the latter non-relation holds because of the implication in \Cref{Eq:2.9:4}, since the shape $\mathtt{sh}((1,j'))=(j'-2,1,1)$ does not dominate the shape $\mathtt{sh}(s_{1}s_{j'})=(j'-2,2)$. Lastly, suppose that we have 
\[ \mathtt{P}_{x}={\begin{ytableau}\scriptstyle j&\scriptstyle j+1&\scriptstyle k\end{ytableau}}^{\ \langle n-3\rangle}, \hspace{2mm} \text{ thus } \hspace{2mm} x=
\begin{cases}
s_{j-1}s_{j-2}s_{j}s_{j-1}s_{k-1} & k>j+2 \\
s_{j-1}s_{j-2}s_{j-3}s_{j}s_{j-1}s_{j-2}s_{j+1}s_{j}s_{j-1} & k=j+2
\end{cases}. \]
Assume that $k>j+2$. Then we have that
\[ x=s_{j-1}s_{j-2}s_{j}s_{j-1}s_{k-1}=F_{(4,2),n}^{(j-2,k-1)}(s_{2}s_{1}s_{3}s_{2},s_{1}). \]
Moreover, by \Cref{Eq:5.7:1}, $z$ consecutively contains the pattern $\underline{p}:=(1423,21)\in\mathrm{S}_{4}\times\mathrm{S}_{2}$ at position $\underline{i}:=(j-2,k-1)$. Therefore, by \Cref{Cor:3:15}, we have that
\begin{equation*}
x\not\leq_{L}^{(n)}z \iff F_{(4,2),n}^{(j-2,k-1)}(s_{2}s_{1}s_{3}s_{2},s_{1})\not\leq_{L}^{(n)}z \iff (s_{2}s_{1}s_{3}s_{2},s_{1})\not\leq_{L}^{(4,2)}(1423,21),
\end{equation*}
and the latter non-relation holds since $s_{2}s_{1}s_{3}s_{2}\not\leq_{L}^{(4)}1423$ as discussed in the first case above. Now assume that $k=j+2$. Then we have that
\[ x=s_{j-1}s_{j-2}s_{j-3}s_{j}s_{j-1}s_{j-2}s_{j+1}s_{j}s_{j-1}=F_{6,n}^{(j-3)}(s_{3}s_{2}s_{1}s_{4}s_{3}s_{2}s_{5}s_{4}s_{3}). \]
From \Cref{Eq:5.7:1}, $z$ consecutively contains the pattern $p:=126345\in\mathrm{S}_{6}$ at position $j-3$.
Therefore, by Claim~(b) of \Cref{Prop:3:7}, we have that
\begin{equation*}
x\not\leq_{L}^{(n)}z \iff F_{6,n}^{(j-3)}(s_{3}s_{2}s_{1}s_{4}s_{3}s_{2}s_{5}s_{4}s_{3})\not\leq_{L}^{(n)}z \iff s_{3}s_{2}s_{1}s_{4}s_{3}s_{2}s_{5}s_{4}s_{3}\not\leq_{L}^{(6)}126345,
\end{equation*}
and the latter non-relation holds because of the implication in \Cref{Eq:2.9:4}, since the shape $\mathtt{sh}(s_{3}s_{2}s_{1}s_{4}s_{3}s_{2}s_{5}s_{4}s_{3})=(3,3)$ does not dominate the shape $\mathtt{sh}(126345)=(5,1)$. As discussed, the remaining cases follow by similar arguments.
\end{proof}

By \Cref{Lem:5.7:1}, and consulting the list in \Cref{Type7RList1}, if $x\leq_{R}^{(n)}z$, then $\mathtt{P}_{x}$ belongs to the following list:
\begin{equation}\label{Type7RList2}
\emptyset^{\langle n\rangle} \hspace{6mm}
{\begin{ytableau}\scriptstyle i\end{ytableau}}^{\ \langle n-1\rangle} \hspace{6mm}
{\begin{ytableau}\scriptstyle j\end{ytableau}}^{\ \langle n-1\rangle} \hspace{6mm}
{\begin{ytableau}\scriptstyle k\end{ytableau}}^{\ \langle n-1\rangle} \hspace{6mm}
{\begin{ytableau}\scriptstyle i&\scriptstyle j\end{ytableau}}^{\ \langle n-2\rangle}
\end{equation}
\begin{equation*}
{\begin{ytableau}\scriptstyle i&\scriptstyle k\end{ytableau}}^{\ \langle n-2\rangle} \hspace{6mm}
{\begin{ytableau}\scriptstyle j&\scriptstyle k\end{ytableau}}^{\ \langle n-2\rangle} \hspace{6mm}
{\begin{ytableau}\scriptstyle i\\\scriptstyle k\end{ytableau}}^{\ \langle n-2\rangle} \hspace{6mm}
{\begin{ytableau}\scriptstyle j\\\scriptstyle k\end{ytableau}}^{\ \langle n-2\rangle} \hspace{6mm}
{\begin{ytableau}\scriptstyle i&\scriptstyle j\\\scriptstyle k\end{ytableau}}^{\ \langle n-3\rangle}
\end{equation*}

We now seek to prove the Indecomposability conjecture for Type (7$^{\ast}$) involutions $z$. To help with this, we first prove Kostant positivity for a special case of $z$: 

\begin{lem}\label{Lem:5.7:2}
If $k=n$, the involution $z=z_{i,j,n}^{n}$ of Type (7$^{\ast}$) is Kostant positive, i.e. $\mathbf{K}_{n}(z)=\mathtt{true}$.  
\end{lem}

\begin{proof}
By \Cref{Eq:2.13:3}, it suffices to prove that
\[ z':=w_{0,n}z_{i,j,n}^{n}w_{0,n}=(1,k')(j'-1,j')\in\mathrm{S}_{n} \]
is Kostant positive, where $k':=n-j+2$ and $j'=n-i+2$, and, in particular, we have $k'<j'-2$. Now, via one-line notation, one can deduce that
\[ z'\sim_{L}^{(n)}(s_{k'-2}s_{j'-2})w_{0,n-1}w_{0,n} \]
where $s_{k'-2}s_{j'-2}$ belongs to the parabolic subgroup $\mathrm{S}_{n-1}\subset\mathrm{S}_{n}$. Therefore, by \cite[Theorem 1.1]{K10}, we have that $\mathbf{K}_{n}(z')=\mathbf{K}_{n-1}(s_{k'-2}s_{j'-2})$. Lastly, the fully commutative permutation $s_{k'-2}s_{j'-2}$ is Kostant positive by \cite[Theorem 5.1]{MMM24}.
\end{proof}

\begin{lem}\label{Type7IC}
The Type (7$^{\ast}$) involution $z$ above satisfies the Indecomposability Conjecture \ref{Conj:2.12:1}. That is to say, we have that $\mathbf{KM}_{n}(\star,z)=\mathtt{true}$.
\end{lem}

\begin{proof}
For the sake of this proof, it would be convenient to add back the subscripts and superscripts for our involution $z=z_{i,j,k}^{n}$.
By \Cref{Thm:2.13:1}, we know that when $z_{i,j,k}^{n}$ is Kostant positive, then $\mathbf{KM}_{n}(\star,z_{i,j,k}^{n})=\mathtt{true}$. Thus by \Cref{Lem:5.7:2}, we have that $\mathbf{KM}_{n}(\star,z_{i,j,n}^{n})=\mathtt{true}$ (i.e. when $k=n$).

It remains to show $\mathbf{KM}_{n}(\star,z_{i,j,k}^{n})=\mathtt{true}$ when $k\neq n$. By Equations~\eqref{Eq:2.11:1} and \eqref{Eq:2.12:2}, it suffices to show that
\[ \mathbf{KM}_{n}(x,z_{i,j,k}^{n})=\mathtt{true}, \]
for all $x\in\mathrm{Inv}_{n}$ where $x\leq_{R}^{(n)}z_{i,j,k}^{n}$. This relation implies that $\mathtt{P}_{x}$ belongs to the list in \Cref{Type7RList2}. Let $x$ be an involution with $\mathtt{P}_{x}$ one of the standard Young tableau in the list in \Cref{Type7RList2}, except for one of the following three:
\[ {\begin{ytableau}\scriptstyle i\\\scriptstyle k\end{ytableau}}^{\ \langle n-2\rangle}, \hspace{6mm}
{\begin{ytableau}\scriptstyle j\\\scriptstyle k\end{ytableau}}^{\ \langle n-2\rangle}, \hspace{3mm} \text{ or } \hspace{3mm}
{\begin{ytableau}\scriptstyle i&\scriptstyle j\\\scriptstyle k\end{ytableau}}^{\ \langle n-3\rangle}.
 \]
Then it can be checked that the support set $\mathrm{Sup}(x)$ is of size no greater than $6$. For example, if
\[ \mathtt{P}_{x}={\begin{ytableau}\scriptstyle j&\scriptstyle k\end{ytableau}}^{\ \langle n-2\rangle} \] 
then $x=s_{j-1}s_{k-1}$, and so $|\mathrm{Sup}(x)|=2$. Therefore, by the implication in \Cref{Eq:2.12:4}, $\mathbf{KM}_{n}(x,z_{i,j,k}^{n})=\mathtt{true}$ for such involutions $x$. Moreover, if $\mathtt{P}_{x}$ is the tableau of shape $(n-3,2,1)$ given in the above exceptions, then it has the same shape as $z_{i,j,k}^{n}$ itself. As a consequence, we have  $\mathbf{KM}_{n}(x,z_{i,j,k}^{n})=\mathtt{true}$ by \cite[Section 5.2]{KiM16}. Hence, we just need to show that $\mathbf{KM}_{n}(x,z_{i,j,k}^{n})=\mathtt{true}$ when $\mathtt{P}_{x}$ is either 
\begin{equation*}
\hspace{3mm} {\begin{ytableau}\scriptstyle i\\\scriptstyle k\end{ytableau}}^{\ \langle n-2\rangle}  \text{ or }  \hspace{6mm} {\begin{ytableau}\scriptstyle j\\\scriptstyle k\end{ytableau}}^{\ \langle n-2\rangle}.
\end{equation*}
Let $x_{1}$ and $x_{2}$ be the involutions whose insertion tableau are these two above, respectively. Thus
\[ x_{1}=(i-1,k) \hspace{2mm} \text{ and } \hspace{2mm} x_{2}=(j-1,k). \]
Now consider the Kostant positive Type (7$^{\ast}$) involution $z_{i,j,k}^{k}$ belonging to $\mathrm{S}_{k}$ (which is Kostant positive by \Cref{Lem:5.7:2}). Then note that both $x_{1}$ and $x_{2}$ also belong to the subgroup $\mathrm{S}_{k}\subset\mathrm{S}_{n}$. Therefore, via the natural embedding, we have the shifts
\[ F_{k,n}^{1}(z_{i,j,k}^{k})=z_{i,j,k}^{n}, \hspace{2mm} F_{k,n}^{1}(x_{1})=x_{1}, \hspace{2mm} \text{ and } \hspace{2mm} F_{k,n}^{1}(x_{2})=x_{2}. \]
Hence, by \Cref{Lem:3:19}, and since $z_{i,j,k}^{k}$ is Kostant positive (as a permutation in $\mathrm{S}_{k}$),
\[ \mathtt{true}=\mathbf{KM}_{k}\left(x_{a},z_{i,j,k}^{k}\right)=\mathbf{KM}_{n}\left(F_{k,n}^{1}(x_{a}),F_{k,n}^{1}(z_{i,j,k}^{k})\right)=\mathbf{KM}_{n}\left(x_{a},z_{i,j,k}^{n}\right), \]
for each $a\in[2]$, which completes the proof of the lemma.
\end{proof}

We now seek to show that the Type (7$^{\ast}$) involutions $z$ are Kostant positive. To do this, we present the more technical aspects of the proof separately in the following four lemmas:

\begin{lem}\label{Lem:Type7KTech0}
The following hold:
\begin{itemize}
\item[(a)] Consider the strong right Bruhat walk
\[ \mathtt{w}:=(z,zs_{k-1},zs_{k-1}s_{k-2},\dots,zs_{k-1}s_{k-2}\cdots s_{j+1}) \]
and reduced expression $\underline{x}:=s_{k-1}s_{k-2}\dots s_{j}$. Then $\mathtt{w}$ is $\underline{x}$-compatible.
\item[(b)] Consider the strong right Bruhat walk
\[ \mathtt{w}:=(z,zs_{i},zs_{i}s_{i+1},\dots,zs_{i}s_{i+1}\cdots s_{j-2},zs_{i}s_{i+1}\dots s_{j-3}) \]
and reduced expression $\underline{x}:=s_{i-1}s_{i}\dots s_{j-1}$. Then $\mathtt{w}$ is $\underline{x}$-compatible.
\item[(c)] Consider the strong right Bruhat walk
\[ \mathtt{w}:=(z,zs_{j-1},zs_{j-1}s_{j},\dots,zs_{j-1}s_{j}\cdots s_{k-3}) \]
and reduced expression $\underline{x}:=s_{j-1}s_{j}\dots s_{k-2}$. Then $\mathtt{w}$ is $\underline{x}$-compatible.
\item[(d)] Consider the strong right Bruhat walk
\[ \mathtt{w}:=(z,zs_{j-2},zs_{j-2}s_{j-3},\dots,zs_{j-2}s_{j-3}\cdots s_{i}) \]
and reduced expression $\underline{x}:=s_{j-1}s_{j-2}\dots s_{i}$. Then $\mathtt{w}$ is $\underline{x}$-compatible.
\end{itemize}
\end{lem}

\begin{proof}
We prove Claims~(a) and (b) with the other two claims following by similar considerations.

{\bf Claim~(a)}: For $j+1\leq a\leq k$, let $\mathtt{w}_{a}:=zs_{k-1}s_{k-2}\cdots s_{a}$, where, in particular, $\mathtt{w}_{k}=z$. Then
\[ \mathtt{w}=(\mathtt{w}_{k},\mathtt{w}_{k-1},\dots,\mathtt{w}_{j+1}). \]
Moreover, from \Cref{Eq:5.7:1}, one can easily deduce that we have the one-line notation
\[ \mathtt{w}_{a}=1\cdots(i-2){\color{teal}i(i-1)}(i+1)\cdots(j-2){\color{teal}k}j(j+1)\cdots(a-1){\color{teal}(j-1)a\cdots(k-1)}(k+1)\cdots n, \]   
where we have coloured the non-fixed points in {\color{teal}teal}. This is obtained from the one-line notation of $z$ by swapping $(j-1)$ with its left neighbour until it is placed in position $a$. Now, by definition (see \Cref{Sec:2.3}), to prove that $\mathtt{w}$ is $\underline{x}$-compatible for reduced expression $\underline{x}:=s_{k-1}s_{k-2}\dots s_{j}$, it suffices to prove that the following Bruhat relations hold:
\begin{itemize}
\item[(i)] $\mathtt{w}_{k}s_{k-1}<\mathtt{w}_{k}$ and $\mathtt{w}_{k}s_{k-2}>\mathtt{w}_{k}$,
\item[(ii)] $\mathtt{w}_{a}s_{a-1}<\mathtt{w}_{a}$ and $\mathtt{w}_{a}s_{(a-1)\pm1}>\mathtt{w}_{a}$ for each $j+2\leq a\leq k-1$,
\item[(iii)] $\mathtt{w}_{j+1}s_{j}<\mathtt{w}_{j+1}$ and $\mathtt{w}_{j+1}s_{j+1}>\mathtt{w}_{j+1}$.
\end{itemize}
In the above list, (i) accounts for the first entry in $\mathtt{w}$, (ii) for the second to penultimate entries, and (iii) for the last entry. Since $\mathtt{w}_{a}s_{b}<\mathtt{w}_{a}$ if and only if $\mathtt{w}_{a}(b)>\mathtt{w}_{a}(b+1)$, the above Bruhat relations can be confirmed by consulting the one-line notation of $\mathtt{w}_{a}$ given above.

{\bf Claim~(b)}: For $i-1\leq a\leq j-3$, let $\mathtt{w}_{a}:=zs_{i}s_{i+1}\cdots s_{a}$, where in particular $\mathtt{w}_{i-1}=z$. Then
\[ \mathtt{w}=(\mathtt{w}_{i-1},\mathtt{w}_{i},\dots,\mathtt{w}_{j-3},\mathtt{w}_{j-3}s_{j-2},\mathtt{w}_{j-3}). \]
From \Cref{Eq:5.7:1}, one can easily deduce that we have the one-line notation
\[ \mathtt{w}_{a}=1\cdots(i-2){\color{teal}i(i+1)\cdots (a+1)(i-1)}(a+2)\cdots(j-2){\color{teal}k}j\cdots(k-1){\color{teal}(j-1)}(k+1)\cdots n, \]   
where we have coloured the non-fixed points in {\color{teal}teal}. This is obtained from the one-line notation of $z$ by swapping $(i-1)$ with its right neighbour until it is placed in position $a+1$. In particular,
\[ \mathtt{w}_{j-3}=1\cdots(i-2){\color{teal}i(i+1)\cdots (j-2)(i-1)}{\color{teal}k}j\cdots(k-1){\color{teal}(j-1)}(k+1)\cdots n. \]  
This accounts for all entries of $\mathtt{w}$ except $\mathtt{w}_{j-3}s_{j-2}$, whose one-line notation is given by
\[ \mathtt{w}_{j-3}s_{j-2}=1\cdots(i-2){\color{teal}i(i+1)\cdots (j-2)k}{\color{teal}(i-1)}j\cdots(k-1){\color{teal}(j-1)}(k+1)\cdots n. \]  
By definition, to prove that $\mathtt{w}$ is $\underline{x}$-compatible with the reduced expression $\underline{x}:=s_{i-1}s_{i}\dots s_{j-1}$, it suffices to prove that the following Bruhat relations hold:
\begin{itemize}
\item[(i)] $\mathtt{w}_{i-1}s_{i-1}<\mathtt{w}_{i-1}$ and $\mathtt{w}_{i-1}s_{i}>\mathtt{w}_{i-1}$,
\item[(ii)] $\mathtt{w}_{a}s_{a}<\mathtt{w}_{a}$ and $\mathtt{w}_{a}s_{a\pm1}>\mathtt{w}_{a}$ for each $i\leq a\leq j-3$,
\item[(iii)] $(\mathtt{w}_{j-3}s_{j-2})s_{j-2}<\mathtt{w}_{j-2}s_{j-2} \hspace{2mm} \text{ and } \hspace{2mm} (\mathtt{w}_{j-3}s_{j-2})s_{(j-2)\pm1}>\mathtt{w}_{j-3}s_{j-2}$,
\item[(iv)] $\mathtt{w}_{j-3}s_{j-1}<\mathtt{w}_{j-3} \hspace{2mm} \text{ and } \hspace{2mm} \mathtt{w}_{j-3}s_{j-2}>\mathtt{w}_{j-3}$,
\end{itemize}
In the above list, (i) accounts for the first entry in $\mathtt{w}$, (ii) for the second to third last entries, (iii) for the penultimate entry, and (iv) for the last entry. Since $\mathtt{w}_{a}s_{b}<\mathtt{w}_{a}$ if and only if $\mathtt{w}_{a}(b)>\mathtt{w}_{a}(b+1)$, all these Bruhat relations can be confirmed by consulting the one-line notations displayed above. 
\end{proof}

\begin{lem}\label{Lem:Type7KTech1}
Let $x,y\leq_{R}^{(n)}z$ be distinct such that $\mathtt{sh}(x)=\mathtt{sh}(y)=(n-1,1)$, then 
\[ \theta_{x}^{(n)}L_{z}^{(n)}\not\cong\theta_{y}^{(n)}L_{z}^{(n)}. \]
\end{lem}

\begin{proof}
For contradiction, let us assume that we have an isomorphism 
\begin{equation}\label{Eq:7Tech1-1}
\theta_{x}^{(n)}L_{z}^{(n)}\cong\theta_{y}^{(n)}L_{z}^{(n)}.
\end{equation}
As $x,y\leq_{R}^{(n)}z$, both $\mathtt{P}_{x}$ and $\mathtt{P}_{y}$ belong to the list in \Cref{Type7RList2}. By the implication in \Cref{Eq:2.9:6}, $x\sim_{L}^{(n)}y$, which implies that $\mathtt{Q}_{x}=\mathtt{Q}_{y}$, and as $x$ and $y$ are distinct, this implies that $\mathtt{P}_{x}\neq\mathtt{P}_{y}$. Up to symmetry, this implies that 
\begin{align*}
(i) \hspace{2mm} (\mathtt{P}_{x},\mathtt{P}_{y})=&\left({\begin{ytableau}\scriptstyle j\end{ytableau}}^{\ \langle n-1\rangle}, \hspace{2mm} {\begin{ytableau}\scriptstyle k\end{ytableau}}^{\ \langle n-1\rangle} \right), \hspace{2mm}
(ii) \hspace{2mm} (\mathtt{P}_{x},\mathtt{P}_{y})=\left({\begin{ytableau}\scriptstyle i\end{ytableau}}^{\ \langle n-1\rangle}, \hspace{2mm} {\begin{ytableau}\scriptstyle j\end{ytableau}}^{\ \langle n-1\rangle} \right), \\
&\hspace{5mm} (iii) \hspace{2mm} (\mathtt{P}_{x},\mathtt{P}_{y})=\left({\begin{ytableau}\scriptstyle i\end{ytableau}}^{\ \langle n-1\rangle}, \hspace{2mm} {\begin{ytableau}\scriptstyle k\end{ytableau}}^{\ \langle n-1\rangle} \right).
\end{align*}
Furthermore, by the equivalence in \Cref{Eq:2.11:3}, it suffices to consider any particular pair $(x',y')$ such that $(\mathtt{P}_{x'},\mathtt{P}_{y'})$ is one of the above pairings such that $x'\sim_{L}^{(n)}y'$. Thus, by assumption, we have that $\theta_{x}^{(n)}L_{z}^{(n)}\cong\theta_{y}^{(n)}L_{z}^{(n)}$ holds for $x,y\leq_{R}^{(n)}z$ being (at least) one of the following three pairs:
\begin{align*}
(i) \hspace{2mm} (x,y)=(&s_{j-1}, \hspace{1mm} s_{k-1}s_{k-2}\cdots s_{j-1}), \hspace{2mm}
(ii) \hspace{2mm} (x,y)=(s_{i-1}s_{i}\cdots s_{j-1}, \hspace{1mm} s_{j-1}), \\
&(iii) \hspace{2mm} (x,y)=(s_{i-1}s_{i}\cdots s_{j-1}, \hspace{1mm} s_{k-1}s_{k-2}\cdots s_{j-1}).
\end{align*}
We examine these three cases one at a time, showing that each leads to a contradiction.

{\bf Case (i)}: We have the following isomorphisms:
\[ \theta_{y}^{(n)}L_{z}^{(n)}\cong\theta_{s_{j-1}}^{(n)}\left(\theta_{s_{j}}^{(n)}\cdots\theta_{s_{k-2}}^{(n)}\theta_{s_{k-1}}^{(n)}L_{z}^{(n)}\right)\cong\theta_{s_{j-1}}^{(n)}\theta_{s_{j}}^{(n)}L_{zs_{k-1}s_{k-2}\cdots s_{j+1}}^{(n)}. \]
The first isomorphism above follows from \Cref{Lem:2.6:2}, while the last isomorphism follows from Claim~(b) of \Cref{Lem:2.11:5} and Claim~(a) of \Cref{Lem:Type7KTech0}. By \Cref{Eq:7Tech1-1}, this implies that
\begin{equation*}
\theta_{x}^{(n)}L_{z}^{(n)}\cong\theta_{s_{j-1}}^{(n)}\theta_{s_{j}}^{(n)}L_{zs_{k-1}s_{k-2}\cdots s_{j+1}}^{(n)}.
\end{equation*}
Decategorifying this isomorphism yields the equality
\begin{equation}\label{Eq:7Tech1-2nn}
D_{z}^{(n)}C_{x}^{(n)}=D_{zs_{k-1}s_{k-2}\cdots s_{j+1}}^{(n)}C_{s_{j}}^{(n)}C_{s_{j-1}}^{(n)}.
\end{equation}
However, from \Cref{Eq:2.8:1}, the right hand side of \Cref{Eq:7Tech1-2nn} contains the unique degree two term $D_{zs_{k-1}s_{k-2}\cdots s_{j+1}}$ (when written in terms of the dual Kazhdan-Lusztig basis), while the left hand side has no terms of degree two or higher. This gives the desired contradiction for Case (i).

{\bf Case (ii)}: We have the following isomorphisms:
\[ \theta_{x}^{(n)}L_{z}^{(n)}\cong\theta_{s_{j-1}}^{(n)}\theta_{s_{j-2}}^{(n)}\cdots\theta_{s_{i-1}}^{(n)}L_{z}^{(n)}\cong\theta_{s_{j-1}}^{(n)}L_{zs_{i}s_{i+1}\cdots s_{j-3}}^{(n)}. \]
The first isomorphism above follows from \Cref{Lem:2.6:2}, while the last isomorphism follows from Claim~(b) of \Cref{Lem:2.11:5} and Claim~(b) of \Cref{Lem:Type7KTech0}. By \Cref{Eq:7Tech1-1}, this implies that
\begin{equation*}
\theta_{s_{j-1}}^{(n)}L_{zs_{i}s_{i+1}\cdots s_{j-3}}^{(n)}\cong\theta_{s_{j-1}}^{(n)}L_{z}^{(n)}.
\end{equation*}
Decategorifying this isomorphism yields the equality
\begin{equation*}
D_{zs_{i}s_{i+1}\cdots s_{j-3}}^{(n)}C_{s_{j-1}}^{(n)}=D_{z}^{(n)}C_{s_{j-1}}^{(n)}.
\end{equation*}
By \Cref{Eq:2.8:1}, this equality cannot hold as $zs_{i}s_{i+1}\cdots s_{j-3}\neq z$, recalling that $i+2<j$. This gives the desired contradiction for Case (ii).

{\bf Case (iii)}: From the previous two cases, we have the equivalence of isomorphisms
\[ \theta_{x}^{(n)}L_{z}^{(n)}\cong\theta_{y}^{(n)}L_{z}^{(n)} \iff \theta_{s_{j-1}}^{(n)}L_{zs_{i}s_{i+1}\cdots s_{j-3}}^{(n)}\cong\theta_{s_{j-1}}^{(n)}\theta_{s_{j}}^{(n)}L_{zs_{k-1}s_{k-2}\cdots s_{j+1}}^{(n)}. \]
However, as mentioned in Case (i), the right hand side of the latter isomorphism has a degree two term, while the left hand side does not, which gives the desired contradiction for Case (iii).
\end{proof}

\begin{lem}\label{Lem:Type7KTech2}
Let $x,y\leq_{R}^{(n)}z$  be distinct such that $\mathtt{sh}(x)=\mathtt{sh}(y)=(n-2,2)$, then 
\[ \theta_{x}^{(n)}L_{z}^{(n)}\not\cong\theta_{y}^{(n)}L_{z}^{(n)}. \]
\end{lem}

\begin{proof}
For contradiction, let us assume that we have an isomorphism 
\begin{equation}\label{Eq:7Tech2-1}
\theta_{x}^{(n)}L_{z}^{(n)}\cong\theta_{y}^{(n)}L_{z}^{(n)}.
\end{equation}
As $x,y\leq_{R}^{(n)}z$, both $\mathtt{P}_{x}$ and $\mathtt{P}_{y}$ belong to the list in \Cref{Type7RList2}. By the implication in \Cref{Eq:2.9:6}, $x\sim_{L}^{(n)}y$, which implies that $\mathtt{Q}_{x}=\mathtt{Q}_{y}$, and as $x$ and $y$ are distinct, this implies that $\mathtt{P}_{x}\neq\mathtt{P}_{y}$. Up to symmetry, this implies that 
\begin{align*}
(i) \hspace{2mm} (\mathtt{P}_{x},\mathtt{P}_{y})=&\left({\begin{ytableau}\scriptstyle i&\scriptstyle j\end{ytableau}}^{\ \langle n-2\rangle}, \hspace{2mm} {\begin{ytableau}\scriptstyle i&\scriptstyle k\end{ytableau}}^{\ \langle n-2\rangle} \right), \hspace{2mm}
(ii) \hspace{2mm} (\mathtt{P}_{x},\mathtt{P}_{y})=\left({\begin{ytableau}\scriptstyle i&\scriptstyle k\end{ytableau}}^{\ \langle n-2\rangle}, \hspace{2mm} {\begin{ytableau}\scriptstyle j&\scriptstyle k\end{ytableau}}^{\ \langle n-2\rangle} \right), \\
&\hspace{5mm} (iii) \hspace{2mm} (\mathtt{P}_{x},\mathtt{P}_{y})=\left({\begin{ytableau}\scriptstyle i&\scriptstyle j\end{ytableau}}^{\ \langle n-2\rangle}, \hspace{2mm} {\begin{ytableau}\scriptstyle j&\scriptstyle k\end{ytableau}}^{\ \langle n-2\rangle} \right).
\end{align*}
Furthermore, by the equivalence in \Cref{Eq:2.11:3}, it suffices to consider any particular pair $(x',y')$ such that $(\mathtt{P}_{x'},\mathtt{P}_{y'})$ is one of the above pairings such that $x'\sim_{L}^{(n)}y'$. Thus, by assumption, we have that $\theta_{x}^{(n)}L_{z}^{(n)}\cong\theta_{y}^{(n)}L_{z}^{(n)}$ holds for $x,y\leq_{R}^{(n)}z$ being (at least) one of the following three pairs:
\begin{align*}
(i) \hspace{2mm} (x,y)=(&s_{i-1}s_{j-1}, \hspace{1mm} s_{i-1}s_{k-1}s_{k-2}\cdots s_{j-1}), \hspace{2mm}
(ii) \hspace{2mm} (x,y)=(s_{k-1}s_{i-1}s_{i}\cdots s_{j-1}, \hspace{1mm} s_{k-1}s_{j-1}), \\
&(iii) \hspace{2mm} (x,y)=(s_{i-1}s_{j-1}s_{j}\cdots s_{k-1}, \hspace{1mm} s_{k-1}s_{j-1}s_{j-2}\cdots s_{i-1}).
\end{align*}
We examine these three cases one at a time, showing that each leads to a contradiction.

{\bf Case (i)}: We have the following isomorphisms:
\[ \theta_{y}^{(n)}L_{z}^{(n)}\cong\theta_{s_{i-1}}^{(n)}\theta_{s_{j-1}}^{(n)}\left(\theta_{s_{j}}^{(n)}\cdots\theta_{s_{k-2}}^{(n)}\theta_{s_{k-1}}^{(n)}L_{z}^{(n)}\right)\cong\theta_{s_{i-1}}^{(n)}\theta_{s_{j-1}}^{(n)}\theta_{s_{j}}^{(n)}L_{zs_{k-1}s_{k-2}\cdots s_{j+1}}^{(n)}. \]
The first isomorphism above follows from \Cref{Lem:2.6:2}, while the last isomorphism follows from Claim~(b) of \Cref{Lem:2.11:5} and Claim~(a) of \Cref{Lem:Type7KTech0}. By the isomorphism in \Cref{Eq:7Tech2-1}, this implies that
\begin{equation*}
\theta_{s_{i-1}}^{(n)}\theta_{s_{j-1}}^{(n)}L_{z}^{(n)}\cong\theta_{s_{i-1}}^{(n)}\theta_{s_{j-1}}^{(n)}\theta_{s_{j}}^{(n)}L_{zs_{k-1}s_{k-2}\cdots s_{j+1}}^{(n)}.
\end{equation*}
Decategorifying this isomorphism yields the equality
\begin{equation}\label{Eq:7Tech1-2}
D_{z}^{(n)}C_{s_{j-1}}^{(n)}C_{s_{i-1}}^{(n)}=D_{zs_{k-1}s_{k-2}\cdots s_{j+1}}^{(n)}C_{s_{j}}^{(n)}C_{s_{j-1}}^{(n)}C_{s_{i-1}}^{(n)}.
\end{equation}
However, from \Cref{Eq:2.8:1}, the right hand side of \Cref{Eq:7Tech1-2} contains the unique degree three term $D_{zs_{k-1}s_{k-2}\cdots s_{j+1}}$ (when written in terms of the dual Kazhdan-Lusztig basis), while the left hand side has no terms of degree three or higher. This gives the desired contradiction for Case (i).

{\bf Case (ii)}: We have the following isomorphisms:
\[ \theta_{x}^{(n)}L_{z}^{(n)}\cong\theta_{s_{k-1}}^{(n)}\theta_{s_{j-1}}^{(n)}\theta_{s_{j-2}}^{(n)}\cdots\theta_{s_{i-1}}^{(n)}L_{z}^{(n)}\cong\theta_{s_{k-1}}^{(n)}\theta_{s_{j-1}}^{(n)}L_{zs_{i}s_{i+1}\cdots s_{j-3}}^{(n)}. \]
The first isomorphism above follows from \Cref{Lem:2.6:2}, while the last isomorphism follows from Claim~(b) of \Cref{Lem:2.11:5} and Claim~(b) of \Cref{Lem:Type7KTech0}. By the isomorphism in \Cref{Eq:7Tech2-1}, this implies that
\begin{equation*}
\theta_{s_{k-1}}^{(n)}\theta_{s_{j-1}}^{(n)}L_{zs_{i}s_{i+1}\cdots s_{j-3}}^{(n)}\cong\theta_{s_{k-1}}^{(n)}\theta_{s_{j-1}}^{(n)}L_{z}^{(n)}.
\end{equation*}
Decategorifying this isomorphism yields the equality
\begin{equation*}
D_{zs_{i}s_{i+1}\cdots s_{j-3}}^{(n)}C_{s_{j-1}}^{(n)}C_{s_{k-1}}^{(n)}=D_{z}^{(n)}C_{s_{j-1}}^{(n)}C_{s_{k-1}}^{(n)}.
\end{equation*}
By \Cref{Eq:2.8:1}, this equality cannot hold as $zs_{i}s_{i+1}\cdots s_{j-3}\neq z$, recalling that $i+2<j$. This gives the desired contradiction for Case (ii).

{\bf Case (iii)}: We have the following two chains of isomorphisms:
\begin{itemize}
\item $\theta_{x}^{(n)}L_{z}^{(n)}\cong\theta_{s_{i-1}}^{(n)}\theta_{s_{k-1}}^{(n)}\left(\theta_{s_{k-2}}^{(n)}\theta_{s_{k-3}}^{(n)}\cdots\theta_{s_{j-1}}^{(n)}L_{z}^{(n)}\right)\cong\theta_{s_{i-1}}^{(n)}\theta_{s_{k-1}}^{(n)}\theta_{s_{k-2}}^{(n)}L_{zs_{j-1}s_{j}\cdots s_{k-3}}^{(n)}$,
\item $\theta_{y}^{(n)}L_{z}^{(n)}\cong\theta_{s_{k-1}}^{(n)}\theta_{s_{i-1}}^{(n)}\left(\theta_{s_{i}}^{(n)}\theta_{s_{i+1}}^{(n)}\cdots\theta_{s_{j-1}}^{(n)}L_{z}^{(n)}\right)\cong\theta_{s_{k-1}}^{(n)}\theta_{s_{i-1}}^{(n)}\theta_{s_{i}}^{(n)}L_{zs_{j-2}s_{j-3}\cdots s_{i}}^{(n)}$.
\end{itemize}
The first isomorphism for the above two chains follows from \Cref{Lem:2.6:2}, while, for the former chain, the second isomorphism follows from Claim~(b) of \Cref{Lem:2.11:5} and Claim~(c) of \Cref{Lem:Type7KTech0}, and, for the latter chain, the second isomorphism follows from Claim~(b) of \Cref{Lem:2.11:5} and Claim~(d) of \Cref{Lem:Type7KTech0}. Now, from our assumed isomorphism in \Cref{Eq:7Tech2-1}, this implies that
\begin{equation*}
\theta_{s_{i-1}}^{(n)}\theta_{s_{k-1}}^{(n)}\theta_{s_{k-2}}^{(n)}L_{zs_{j-1}s_{j}\cdots s_{k-3}}^{(n)}\cong\theta_{s_{k-1}}^{(n)}\theta_{s_{i-1}}^{(n)}\theta_{s_{i}}^{(n)}L_{zs_{j-2}s_{j-3}\cdots s_{i}}^{(n)}.
\end{equation*}
Decategorifying this isomorphism yields the equality
\begin{equation*}
D_{zs_{j-1}s_{j}\cdots s_{k-3}}^{(n)}C_{s_{k-2}}^{(n)}C_{s_{k-1}}^{(n)}C_{s_{i-1}}^{(n)}=D_{zs_{j-2}s_{j-3}\cdots s_{i}}^{(n)}C_{s_{i}}^{(n)}C_{s_{i-1}}^{(n)}C_{s_{k-1}}^{(n)}.
\end{equation*}
By applying \Cref{Eq:2.8:1} (three times for each side), the left hand side of the above equality contains the unique degree three term $D_{zs_{j-1}s_{j}\cdots s_{k-3}}^{(n)}$, while the right hand side contains the unique degree three term $D_{zs_{j-2}s_{j-3}\cdots s_{i}}^{(n)}$. Thus this equality cannot be since $zs_{j-1}s_{j}\cdots s_{k-3}\neq zs_{j-2}s_{j-3}\cdots s_{i}$. This gives the desired contradiction for Case (iii), and completes the proof of the lemma.
\end{proof}

\begin{lem}\label{Lem:Type7KTech3}
Let $x,y\leq_{R}^{(n)}z$  be distinct such that $\mathtt{sh}(x)=\mathtt{sh}(y)=(n-2,1,1)$, then 
\[ \theta_{x}^{(n)}L_{z}^{(n)}\not\cong\theta_{y}^{(n)}L_{z}^{(n)}. \]
\end{lem}

\begin{proof}
For contradiction, let us assume that we have an isomorphism 
\begin{equation}\label{Eq:7Tech3-1}
\theta_{x}^{(n)}L_{z}^{(n)}\cong\theta_{y}^{(n)}L_{z}^{(n)}.
\end{equation}
As $x,y\leq_{R}^{(n)}z$, both $\mathtt{P}_{x}$ and $\mathtt{P}_{y}$ belong to the lis in \Cref{Type7RList2}. By the implication in \Cref{Eq:2.9:6}, $x\sim_{L}^{(n)}y$, which implies that $\mathtt{Q}_{x}=\mathtt{Q}_{y}$, and as $x$ and $y$ are distinct, this implies that $\mathtt{P}_{x}\neq\mathtt{P}_{y}$. Up to symmetry, this implies that 
\begin{equation*}
\mathtt{P}_{x}={\begin{ytableau}\scriptstyle i\\\scriptstyle k\end{ytableau}}^{\ \langle n-2\rangle} \text{ and } \hspace{2mm} \mathtt{P}_{y}={\begin{ytableau}\scriptstyle j\\\scriptstyle k\end{ytableau}}^{\ \langle n-2\rangle}.
\end{equation*}
Furthermore, by the equivalence in \Cref{Eq:2.11:3}, it suffices to consider any particular pair $(x',y')$ such that $(\mathtt{P}_{x'},\mathtt{P}_{y'})$ is one of the above pairings such that $x'\sim_{L}^{(n)}y'$. Thus, by assumption, we have that $\theta_{x}^{(n)}L_{z}^{(n)}\cong\theta_{y}^{(n)}L_{z}^{(n)}$ holds for $x,y\leq_{R}^{(n)}z$ being the pair
\begin{equation*}
(x,y)=(s_{i-1}s_{i}\cdots s_{k-2}s_{k-1}s_{k-2}\cdots s_{j}s_{j-1}, \hspace{1mm} s_{j-1}s_{j}\cdots s_{k-2}s_{k-1}s_{k-2}\cdots s_{j}s_{j-1}).
\end{equation*}
With this pair, we have the chain of isomorphisms
\[ \theta_{x}^{(n)}L_{z}^{(n)}\cong\theta_{y}^{(n)}\left(\theta_{s_{j-2}}^{(n)}\theta_{s_{j-3}}^{(n)}\cdots\theta_{s_{i-1}}^{(n)}L_{z}^{(n)}\right)\cong\theta_{y}^{(n)}\theta_{s_{j-2}}^{(n)}L_{zs_{i}s_{i+1}\cdots s_{j-2}}^{(n)}\cong\theta_{ys_{j-2}}^{(n)}L_{zs_{i}s_{i+1}\cdots s_{j-2}}^{(n)}. \]
The first and third isomorphism follows from \Cref{Lem:2.6:2}, while the second follows from Claim~(b) of \Cref{Lem:2.11:5} and Claim~(b) of \Cref{Lem:Type7KTech0}. By the isomorphism in \Cref{Eq:7Tech3-1}, this implies that
\[ \theta_{ys_{j-2}}^{(n)}L_{zs_{i}s_{i+1}\cdots s_{j-2}}^{(n)}\cong\theta_{y}^{(n)}L_{z}^{(n)}. \]
Decategorifying this isomorphism yields the equality
\begin{equation}\label{Eq:7Tech3-1n1}
D_{zs_{i}s_{i+1}\cdots s_{j-2}}^{(n)}C_{ys_{j-2}}^{(n)}=D_{z}^{(n)}C_{y}^{(n)}.
\end{equation}
We compare coefficients of the term $D_{s_{i-1}}^{(n)}$ in both sides of \Cref{Eq:7Tech3-1n1}. For the right hand side we have
\[ [D_{s_{i-1}}^{(n)}](D_{z}^{(n)}C_{y}^{(n)})=[C_{z}^{(n)}](C_{s_{i-1}}^{(n)}C_{y}^{(n)})=1, \]
where the first equality follows by \Cref{Eq:2.8:2}, and the last since $C_{s_{i-1}}^{(n)}C_{y}^{(n)}=C_{s_{i-1}y}^{(n)}$ (by \Cref{Lem:2.6:2}) and $s_{i-1}y=z$. As for the left hand side of \Cref{Eq:7Tech3-1n1}, we have
\[ [D_{s_{i-1}}^{(n)}](D_{zs_{i}s_{i+1}\cdots s_{j-2}}^{(n)}C_{ys_{j-2}}^{(n)})=[C_{zs_{i}s_{i+1}\cdots s_{j-2}}^{(n)}](C_{s_{i-1}}^{(n)}C_{ys_{j-2}}^{(n)})=0, \]
where again the first equality follows by \Cref{Eq:2.8:2}, and the last since $C_{s_{i-1}}^{(n)}C_{ys_{j-2}}^{(n)}=C_{s_{i-1}ys_{j-2}}^{(n)}$ (by \Cref{Lem:2.6:2} and recalling that $i+2<j$) and $s_{i-1}ys_{j-2}\neq zs_{i}s_{i+1}\cdots s_{j-2}$. Since these coefficients of the term $D_{s_{i-1}}^{(n)}$ disagree, the equality in \Cref{Eq:7Tech3-1n1} cannot be, giving the desired contradiction.
\end{proof}

\begin{prop}\label{Type7K}
The Type (7$^{\ast}$) involution $z$ is Kostant positive, that is $\mathbf{K}_{n}(z)=\mathtt{true}$.
\end{prop}

\begin{proof}
By \Cref{Thm:2.13:1} and \Cref{Type7IC}, it suffices to prove that we have a non-isomorphism
\begin{equation}\label{Eq:Type3K-1nn}
\theta_{x}^{(n)}L_{z}^{(n)}\not\cong\theta_{y}^{(n)}L_{z}^{(n)},
\end{equation}
for all distinct pairs $x,y\in\mathrm{S}_{n}$ where $x,y\leq_{R}^{(n)}z$. For contradiction, assume there exists $x,y\leq_{R}^{(n)}z$ such that $\theta_{x}^{(n)}L_{z}^{(n)}\cong\theta_{y}^{(n)}L_{z}^{(n)}$. As $x,y\leq_{R}^{(n)}z$, both $\mathtt{P}_{x}$ and $\mathtt{P}_{y}$ belong to the list in \Cref{Type7RList2} above. Moreover, by \Cref{Eq:2.9:6}, we have $x\sim_{L}^{(n)}y$, which implies $\mathtt{Q}_{x}=\mathtt{Q}_{y}$, and, in particular, $\mathtt{sh}(x)=\mathtt{sh}(y)$. Also, since $x$ and $y$ are distinct, this implies that $\mathtt{P}_{x}\neq\mathtt{P}_{y}$. Hence altogether, both $\mathtt{P}_{x}$ and $\mathtt{P}_{y}$ belong to the list in \Cref{Type7RList2}, share the same shape, but are distinct. This implies that the shape $\mathtt{sh}(x)=\mathtt{sh}(y)$ is either $(n-1,1)$, $(n-2,2)$, or $(n-2,1,1)$. Therefore, the previous three lemmas give us the desired contradiction, completing the proof.
\end{proof}

\subsection{Involutions of Shape $(n-3,2,1)$}\label{Sec:5.8}

In this section, we present the final theorems that we obtain by collecting together all the results of the previous sections. To begin, for each of the seven Types (up to conjugation by $w_{0,n}$) of involutions of shape $(n-3,2,1)$, we proved that the Indecomposability Conjecture \ref{Conj:2.12:1} holds. Hence we have the following theorem:

\begin{thm}\label{Thm:5.8:1}
Let $n\geq5$ and $z\in\mathrm{Inv}_{n}$ be such that $\mathtt{sh}(z)=(n-3,2,1)$, then $z$ satisfies the Indecomposability Conjecture \ref{Conj:2.12:1}. That is to say, we have that $\mathbf{KM}_{n}(\star,z)=\mathtt{true}$.
\end{thm}

Next, for each of the seven Types of involutions of shape $(n-3,2,1)$, we provided an answer to Kostant's problem. We summarise these results into the following theorem:

\begin{thm}\label{Thm:5.8:2}
Let $n\geq5$ and $z\in\mathrm{Inv}_{n}$ be such that $\mathtt{sh}(z)=(n-3,2,1)$, then $z$ if Kostant negative if and only if it consecutively contains any of the following four patterns:
\[ 2143, \hspace{1mm} 14325, \hspace{1mm} 1536247, \hspace{1mm} 1462537. \]
\end{thm}

\begin{proof}
Firstly, note that this list of patterns is closed under conjugation by $w_{0,n}$. Therefore, we just need to confirm the statement of the theorem for each of the seven Types of involutions that we considered in the previous sections (being (1), (2), (3), (4), (5$^{\ast}$), (6$^{\ast}$), and (5$^{\ast}$)) and not their duals. 

By going through each of the types, one can easily confirm the theorem by consulting the one-line notation of the given involution, and the corresponding proposition which provides the answer to Kostant's problem for that type. We discuss this for Type (6$^{\ast}$), with others following analogously.

Let $z=z_{i,j,j+1}^{n}$ be of Type (6$^{\ast}$), so $i+2<j<n$. Recall, its one-line notation is given by
\[ z=1\cdots(i-2){\color{teal}i(i-1)}(i+1)\cdots(j-2){\color{teal}(j+1)}j{\color{teal}(j-1)}(j+2)\cdots n, \]
where we have coloured the non-fixed points in {\color{teal}teal}. By focusing on the positions of the {\color{teal}teal} points, one can observe that the patterns 2143, 1536247, and 1462537, can never consecutively appear in $z$ (regardless of the values of $i$ and $j$ which satisfy $i+2<j<n$). As for the pattern 1{\color{teal}4}3{\color{teal}2}5, this can only appear at position $j-2$, and only if $j+2\leq n$. Since $j<n$, the condition of $j+2\leq n$ is equivalent to $j+1\neq n$. Thus, we have shown that $z$ consecutively contains any of the four patterns presented in the statement of the theorem if and only if $j+1\neq n$. By \Cref{Type6K}, this is precisely the same condition for Kostant negatively of $z$, and thus the theorem holds for Type (6$^{\ast}$). As mentioned, one can shown that it holds for the other Types by the same considerations.
\end{proof}

\begin{rmk}\label{Rmk:5.8:3}
By the implications in Equations~\eqref{Eq:2.12:3} and \eqref{Eq:2.13:2}, both the Indecomposability Conjecture and the answer to Kostant's problem are left cell invariants. Therefore, \Cref{Thm:5.8:1} confirms the Indecomposability Conjecture for all permutations with shape $(n-3,2,1)$, and \Cref{Thm:5.8:2} provides an answer to Kostant's problem for all permutations with shape $(n-3,2,1)$.
\end{rmk}

\begin{rmk}\label{Rmk:5.8:456}
We also note that the patterns $1536247$ and $1462537$ are cuspidal 
in the terminology of \cite{CM26} (this means that they are Kostant negative
but any proper consecutive subpattern is Kostant positive).
\end{rmk}

Lastly, by knowing the answer to Kostant's problem for each type, we can prove the following:

\begin{prop}\label{Prop:5.8:4}
The Asymptotic Shape Conjecture \ref{Conj:2.13:5} holds for the shape $\lambda=(2,1)$.
\end{prop}

\begin{proof}
It suffices to prove that the proportion of permutations of shape $(2,1)^{\langle n\rangle}$ which are Kostant negative is asymptotically equal to 0, that is to say,
\[ \frac{\mathbf{k}_{3+n}^{-}((2,1)^{\langle n\rangle})}{|\mathtt{SYT}_{3+n}((2,1)^{\langle n\rangle})|^{2}}\rightarrow 0 \hspace{2mm} \text{(as $n\geq2$ tends towards $\infty$)}. \] 
Let $\mathbf{ki}_{3+n}^{-}((2,1)^{\langle n\rangle})$ be the number of Kostant negative involutions of shape $(n,2,1)$. Then the quantity $\mathbf{k}_{3+n}^{-}((2,1)^{\langle n\rangle})$ is the product of $\mathbf{ki}_{3+n}^{-}((2,1)^{\langle n\rangle})$ and the number of standard Young tableau of shape $(2,1)^{\langle n\rangle}$. Therefore, we have that
\[ \frac{\mathbf{k}_{3+n}^{-}((2,1)^{\langle n\rangle})}{|\mathtt{SYT}_{3+n}((2,1)^{\langle n\rangle})|^{2}}=\frac{\mathbf{ki}_{3+n}^{-}((2,1)^{\langle n\rangle})|\mathtt{SYT}_{3+n}((2,1)^{\langle n\rangle})|}{|\mathtt{SYT}_{3+n}((2,1)^{\langle n\rangle})|^{2}}=\frac{\mathbf{ki}_{3+n}^{-}((2,1)^{\langle n\rangle})}{|\mathtt{SYT}_{3+n}((2,1)^{\langle n\rangle})|}. \]
Let $\mathbf{n}_{3+n}^{(T)}$ be the number of Kostant negative involutions of shape $(n,2,1)$ and Type ($T$) for
\[ T\in\{1,2,2^{\ast},3,3^{\ast},4,4^{\ast},5,5^{\ast},6,6^{\ast},7,7^{\ast}\}. \]
By \Cref{Eq:2.12:5}, we have that $\mathbf{n}_{3+n}^{(T)}=\mathbf{n}_{3+n}^{(T^{\ast})}$, for all $T\in[2,7]$, and thus we have that
\begin{equation}\label{Eq:Prop:5.8:4-1}
\mathbf{ki}_{3+n}^{-}((2,1)^{\langle n\rangle})=\mathbf{n}_{3+n}^{(1)}+2\mathbf{n}_{3+n}^{(2)}+2\mathbf{n}_{3+n}^{(3)}+2\mathbf{n}_{3+n}^{(4^{\ast})}+2\mathbf{n}_{3+n}^{(5^{\ast})}+2\mathbf{n}_{3+n}^{(6^{\ast})}+2\mathbf{n}_{3+n}^{(7^{\ast})}.
\end{equation}
In the previous sections, we answered Kostant's problem for involutions of all these types, thus we can work out all of the terms of the right hand side of Equation \eqref{Eq:Prop:5.8:4-1}. 

For example, let us demonstrate how to work out $\mathbf{n}_{3+n}^{(6^{\ast})}$. Firstly, let $N\geq 6$. Now, expressed as a one-line diagram, a Type (6$^{\ast}$) involution has the following form:
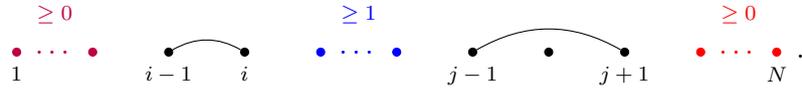
\[
\begin{matrix}\begin{tikzpicture}

\node[V,color=purple, label=below:{{\scriptsize$1$}}] (1) at (1,0){};
\node[draw=none] (1.5) at (1.5,0){{\color{purple}$\dots$}};	
\node[V,color=purple] (2) at (2,0){};
\node[draw=none] (1.5L) at (1.5,0.5){{\color{purple}{\scriptsize$\geq 0$}}};	

\node[V, label=below:{{\scriptsize$i-1$}}] (i-1) at (3,0){};	
\node[V, label=below:{{\scriptsize$i$}}] (i) at (4,0){};
\draw (i-1) to [bend left] (i);

\node[V,color=blue] (i+1) at (5,0){};	
\node[draw=none] (i+1.5) at (5.5,0){{\color{blue}$\dots$}};	
\node[V,color=blue] (j-2) at (6,0){};	
\node[draw=none] (5.5L) at (5.5,0.5){{\color{blue}{\scriptsize$\geq 1$}}};

\node[V, label=below:{{\scriptsize$j-1$}}] (j-1) at (7,0){};	
\node[V] (j) at (8,0){};	
\node[V, label=below:{{\scriptsize$j+1$}}] (j+1) at (9,0){};
\draw (j-1) to [bend left] (j+1);

\node[V,color=red] (j+2) at (10,0){};
\node[draw=none] (j+2.5) at (10.5,0){{\color{red}$\dots$}};	
\node[V,color=red, label=below:{{\scriptsize$N$}}] (N) at (11,0){};	
\node[draw=none] (10.5L) at (10.5,0.5){{\color{red}{\scriptsize$\geq 0$}}};
\end{tikzpicture}\end{matrix}.
\]
As indicated, there can be any non-negative number of {\color{red}red} and {\color{purple}purple} dots, while there needs to be at least $1$ {\color{blue}blue} dot. These conditions correspond to the inequalities imposed on $i$ and $j$ which define the Type (6$^{\ast}$) involutions. From \Cref{Type6K}, a Type (6$^{\ast}$) involution is Kostant negative if and only if its one-line diagram, as displayed above, has a least one {\color{red}red} dot, in particular, $N\geq 7$. Therefore, to calculate the number of Kostant negative involutions of Type (6$^{\ast}$) and shape $(N-3,2,1)$, i.e. the quantity $\mathbf{n}_{N}^{(6^{\ast})}$, we need to count the number of one-line diagrams of the form
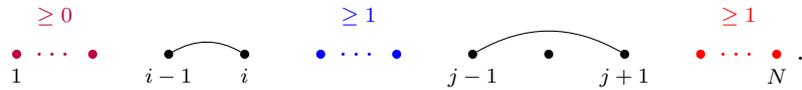
\[
\begin{matrix}\begin{tikzpicture}

\node[V,color=purple, label=below:{{\scriptsize$1$}}] (1) at (1,0){};
\node[draw=none] (1.5) at (1.5,0){{\color{purple}$\dots$}};	
\node[V,color=purple] (2) at (2,0){};
\node[draw=none] (1.5L) at (1.5,0.5){{\color{purple}{\scriptsize$\geq 0$}}};	

\node[V, label=below:{{\scriptsize$i-1$}}] (i-1) at (3,0){};	
\node[V, label=below:{{\scriptsize$i$}}] (i) at (4,0){};
\draw (i-1) to [bend left] (i);

\node[V,color=blue] (i+1) at (5,0){};	
\node[draw=none] (i+1.5) at (5.5,0){{\color{blue}$\dots$}};	
\node[V,color=blue] (j-2) at (6,0){};	
\node[draw=none] (5.5L) at (5.5,0.5){{\color{blue}{\scriptsize$\geq 1$}}};

\node[V, label=below:{{\scriptsize$j-1$}}] (j-1) at (7,0){};	
\node[V] (j) at (8,0){};	
\node[V, label=below:{{\scriptsize$j+1$}}] (j+1) at (9,0){};
\draw (j-1) to [bend left] (j+1);

\node[V,color=red] (j+2) at (10,0){};
\node[draw=none] (j+2.5) at (10.5,0){{\color{red}$\dots$}};	
\node[V,color=red, label=below:{{\scriptsize$N$}}] (N) at (11,0){};	
\node[draw=none] (10.5L) at (10.5,0.5){{\color{red}{\scriptsize$\geq 1$}}};
\end{tikzpicture}\end{matrix}.
\]
The number of such diagrams where there is precisely one {\color{blue}blue} dot is $N-6$, where there is precisely two {\color{blue}blue} dots is $N-7$, and so on, with the maximum number of {\color{blue}blue} dots being $N-6$. Thus,
\[ \mathbf{n}_{N}^{(6^{\ast})}=\sum_{i=1}^{N-6}(N-5-i)=\frac{1}{2}(N-6)(N-5). \]
Therefore, substituting in $3+n$ for $N$, we obtain
\[ \mathbf{n}_{3+n}^{(6^{\ast})}=\frac{1}{2}(n-3)(n-2). \]
For the other quantities $\mathbf{n}_{3+n}^{(T)}$, we immediately have that $\mathbf{n}_{3+n}^{(1)}=\mathbf{n}_{3+n}^{(3)}=\mathbf{n}_{3+n}^{(7^{\ast})}=0$ (since the corresponding types are always Kostant positive), while the remaining quantities can be calculated by similar means to what was done above. When doing this, by \Cref{Eq:Prop:5.8:4-1}, we obtain
\begin{equation*}
\mathbf{ki}_{3+n}^{-}((2,1)^{\langle n\rangle})=2(n-3)+2(n-1)+2\left((n-1)^{2}-\frac{(n-1)n}{2}\right) +2\left(\frac{1}{2}(n-3)(n-2)\right). 
\end{equation*}
In particular, $\mathbf{ki}_{3+n}^{-}((2,1)^{\langle n\rangle})$ is quadratic in $n$. On the other hand, by the hook-length formula,
\[ |\mathtt{SYT}_{3+n}((2,1)^{\langle n\rangle})|=\frac{(n+3)!}{(n-2)!(n)(n+2)(3)}=\frac{1}{3}(n+3)(n+1)(n-1), \]
which is cubic in $n$. Therefore, we see that
\[ \frac{\mathbf{ki}_{3+n}^{-}((2,1)^{\langle n\rangle})}{|\mathtt{SYT}_{3+n}((2,1)^{\langle n\rangle})|}\rightarrow0 \hspace{3mm} (\text{as } 2\leq n\rightarrow\infty). \]
\end{proof}


\vspace{2mm}

\noindent
Department of Mathematics, Uppsala University, Box. 480,
SE-75106, Uppsala, \\ SWEDEN, 
emails:
{\tt samuel.creedon\symbol{64}math.uu.se}\hspace{5mm}
{\tt mazor\symbol{64}math.uu.se}

\end{document}